\documentclass{amsart}

\usepackage{amssymb}
\usepackage{amsmath}
\usepackage{url}
\usepackage{enumerate}
\usepackage{graphicx}
\usepackage[utf8]{inputenc}
\usepackage[noend]{algpseudocode}
\usepackage{multicol}
\usepackage{datetime}
\usepackage{float}
\usepackage[normalem]{ulem}
\usepackage{multicol}
\usepackage{xcolor}
\usepackage{mathtools}

\usepackage[dvipdfmx]{hyperref}
\hypersetup{
    colorlinks  =   true,
    linkcolor   =   blue,
    filecolor   =   blue,
    urlcolor    =   blue,
    citecolor   =   blue,
}

\usepackage{newfloat}
\DeclareFloatingEnvironment[
  name = Procedure
]{procedure}

\numberwithin{procedure}{section}

\newcommand{\erase}[1]{}

\newtheorem{theorem}{Theorem}[section]
\newtheorem{lemma}[theorem]{Lemma}
\newtheorem{proposition}[theorem]{Proposition}
\newtheorem{corollary}[theorem]{Corollary}

\newtheorem{_algorithm}[theorem]{Algorithm}

\newtheorem{_definition}[theorem]{Definition}
\newenvironment{definition}{\begin{_definition}\rm}{\end{_definition}}

\newtheorem{_propositiondefinition}[theorem]{Proposition-Definition}

\newtheorem{_remark}[theorem]{\it Remark}
\newenvironment{remark}{\begin{_remark}\rm}{\end{_remark}}

\newtheorem{_example}[theorem]{Example}
\newenvironment{example}{\begin{_example}\rm}{\end{_example}}

\newtheorem{_assumption}[theorem]{Assumption}

\newtheorem{_construction}[theorem]{Construction}

\newtheorem{_claim}[theorem]{Claim}

\newtheorem{_conjecture}[theorem]{Conjecture}

\numberwithin{equation}{section}
\numberwithin{table}{section}
\numberwithin{figure}{section}
\renewcommand{\qed}{\hfill {$\Box$}}

\newcommand{\C}{\mathord{\mathbb C}}

\newcommand{\F}{\mathord{\mathbb F}}

\renewcommand{\P}{\mathord{\mathbb  P}}
\newcommand{\Q}{\mathord{\mathbb  Q}}
\newcommand{\R}{\mathord{\mathbb R}}

\newcommand{\Z}{\mathord{\mathbb Z}}

\newcommand{\AAA}{\mathord{\mathcal A}}

\newcommand{\EEE}{\mathord{\mathcal E}}

\newcommand{\GGG}{\mathord{\mathcal G}}
\newcommand{\HHH}{\mathord{\mathcal H}}

\newcommand{\PPP}{\mathord{\mathcal P}}

\newcommand{\RRR}{\mathord{\mathcal R}}

\newcommand{\VVV}{\mathord{\mathcal V}}

\newcommand{\SSSS}{\mathord{\mathfrak S}}

\newcommand{\maprightsp}[1]{\; \smash{\mathop{\; \longrightarrow \; }\limits\sp{#1}}\; }

\newcommand{\maprightinjsp}[1]{\; \smash{\mathop{\; \inj \; }\limits\sp{#1}}\; }

\newcommand{\maprightisom}{\; \smash{\mathop{\; \longrightarrow \; }\limits\sp{\sim}}\; }

\newcommand{\mapleftisom}{\; \smash{\mathop{\; \longleftarrow \; }\limits\sp{\sim}}\; }

\newcommand{\mapdown}{\phantom{\Big\downarrow}\hskip -8pt \downarrow}
\newcommand{\mapdownright}[1]{\mapdown\rlap{$\vcenter{\hbox{$\scriptstyle#1$}}$}}
\newcommand{\mapdownleft}[1]{\llap{$\vcenter{\hbox{$\scriptstyle#1$\;}}$}\mapdown}
\newcommand{\mapdownsurj}{
\hbox{$\downarrow$}
\llap{\hbox{\raise 2pt\hbox{$\downarrow$}}}%
}

\newcommand{\mapupsurj}{
\hbox{$\bigm\uparrow$}
\llap{\hbox{\raise 2pt\hbox{$\bigm\uparrow$}}}%
\vstrechmapup
}

\newcommand{\inj}{\hookrightarrow}
\newcommand{\surj}{\mathbin{\to \hskip -7pt \to}}

\newcommand{\isom}{\xrightarrow{\sim}}

\newcommand{\set}[2]{\{\,{#1}\mid {#2} \,\}}

\newcommand{\gen}[1]{\langle {#1}  \rangle}

\newcommand{\tensor}{\otimes}

\newcommand{\sprime}{\sp\prime}

\newcommand{\sbar}[1]{\sb{(#1)}}
\newcommand{\spprime}{\sp{\prime\prime}}

\newcommand{\sperp}{\sp{\perp}}

\newcommand{\dual}{\sp{\vee}}

\newcommand{\semidirectproduct}{\rtimes}

\newcommand{\inv}{\sp{-1}}

\newcommand{\Hom}{\mathord{\mathrm{Hom}}}

\newcommand{\OG}{\mathord{\mathrm{O}}}

\newcommand{\id}{\mathord{\mathrm{id}}}

\newcommand{\Ker}{\operatorname{\mathrm{Ker}}\nolimits}

\newcommand{\Image}{\operatorname{\mathrm{Im}}\nolimits}
\newcommand{\Aut}{\operatorname{\mathrm{Aut}}\nolimits}

\newcommand{\pr}{\mathord{\mathrm{pr}}}

\newcommand{\rank}{\operatorname{\mathrm{rank}}\nolimits}

\newcommand{\closure}[1]{\overline{#1}}

\newcommand{\mystruth}[1]{\phantom{\hbox{\vrule height #1}}}

\newcommand{\intf}[1]{\langle #1 \rangle}

\newcommand{\Rats}{\RRR}
\newcommand{\Ells}{\EEE}
\newcommand{\BP}{\mathord{\rm{BP}}}
\newcommand{\vol}{\mathord{\rm{vol}}}
\newcommand{\StdFD}{\VVV}
\newcommand{\enrinvol}{\varepsilon}
\newcommand{\intfX}[1]{\intf{#1}_X}
\newcommand{\intfY}[1]{\intf{#1}_Y}

\newcommand{\period}{\omega}
\newcommand{\OGP}{\OG^{\PPP}}
\newcommand{\aut}{\mathord{\rm{aut}}}

\newcommand{\Nef}{\mathord{\rm Nef}}
\newcommand{\NefY}{\Nef_{Y}}
\newcommand{\NefX}{\Nef_{X}}
\newcommand{\temp}{\mathord{\rm temp}}

\newcommand{\clD}{\closure{D}}
\newcommand{\LSY}{L_{26}/S_Y(2)}

\newcommand{\none}{\mathord{\textrm{none}}}
\newcommand{\ADE}{\mathord{\rm ADE}}
\newcommand{\RDP}{\mathord{\rm RDP}}
\newcommand{\Rbar}{\closure{R}}
\newcommand{\Rtil}{\widetilde{R}}
\newcommand{\Vtil}{\widetilde{V}}

\newcommand{\SXp}{S_{X+}}
\newcommand{\SXm}{S_{X-}}
\newcommand{\SX}{S_{X}}
\newcommand{\SY}{S_{Y}}
\newcommand{\gens}[1]{\langle#1\rangle}
\newcommand{\LLt}{L_{26}/L_{10}(2)}

\newcommand{\TG}{T_G}

\newcommand{\cloX}{\closure{X}}
\newcommand{\cloY}{\closure{Y}}

\newcommand{\wtC}{\widetilde{C}}

\newcommand{\embR}{\iota_R}
\newcommand{\Lten}{L_{10}}

\newcommand{\embLM}{\varpi}

\newcommand{\downisom}{\mapdown\llap{\small $\hskip -4pt \wr$}}
\newcommand{\taubar}{\bar{\tau}}

\newcommand{\tilg}{\tilde{g}}
\newcommand{\tilr}{\tilde{r}}

\newcommand{\tisom}{\mathord{\rm isom}}
\newcommand{\isoms}{\mathord{\rm isoms}}
\newcommand{\Gbar}{\closure{G}}

\newcommand{\tilC}{\widetilde{C}}

\DeclareMathOperator{\rk}{rank}

\newcommand{\rl}{R}
\newcommand{\brl}{\overline{\rl}}
\newcommand{\trl}{\widetilde{\rl}}
\newcommand{\arl}{\Aut(\tau(\rl))}

\renewcommand{\ker}{\Ker}

\begin{document}

\title[Automorphism groups of  Enriques surfaces]
{Automorphism groups of  certain Enriques surfaces}

\author[S. Brandhorst]{Simon Brandhorst}
\address{Fachbereich Mathematik, Saarland University,
Campus E2.4, 66123 Saarbr\"ucken, Germany}
\email{brandhorst@math.uni-sb.de}

\author[I. Shimada]{Ichiro Shimada}
\address{Department of Mathematics, Graduate School of Science, Hiroshima University,
1-3-1 Kagamiyama, Higashi-Hiroshima, 739-8526 JAPAN}
\email{ichiro-shimada@hiroshima-u.ac.jp}
\thanks{The first author gratefully acknowledges funding by the Deutsche Forschungsgemeinschaft (DFG, German Research Foundation) – Project-ID 286237555 – TRR 195.
The second author gratefully acknowledges financial support by JSPS KAKENHI Grant Number~16H03926 and~20H01798.}

\begin{abstract}
We calculate 
the automorphism group of certain Enriques surfaces.
The Enriques surfaces that we investigate 
include very general $n$-nodal Enriques surfaces
and very general cuspidal Enriques surfaces.
We also describe the action of the automorphism group
on the set of smooth rational curves
and on the set of elliptic fibrations.
\end{abstract}

\maketitle
%
%
\section{Introduction}\label{sec:Introduction}
A central theme in algebraic geometry is to study varieties using
convex geometry. The cone of curves of a variety is the convex hull of
the numerical equivalence classes of curves. 
Its dual is the cone of nef line bundles. 
Much of the birational geometry of a variety is encoded in these cones and their interplay with the canonical divisor.
While for Fano varieties the nef cone is rational polyhedral \cite[Theorem 3.7]{morikollar1998},
in general the nef cone is not well understood. For instance it can have infinitely many faces or be round.

The Morrison--Kawamata cone conjecture \cite{morrison1992,kawamata1997} gives a clear picture of the effective nef cone of a Calabi-Yau variety.
It predicts that the action of the automorphism group on the effective nef cone admits a fundamental domain which is a rational polyhedral cone.
\par
The conjecture is wide open in dimension three and beyond \cite{lazicoguiso2018}. 
But it has been verified for $K3$ surfaces by Sterk \cite{sterk1985},
and for Enriques surfaces by Namikawa \cite{namikawa1985}.
It follows that an Enriques surface admits up to the action of the automorphism group only finitely many \emph{smooth rational curves}, 
finitely many \emph{elliptic fibrations}, finitely many \emph{projective models} of a given degree and
its automorphism group is finitely generated and in fact finitely presented \cite[Corollaries 4.15, 4.16]{looijenga2014}.
\par
Naturally, enumerative questions arise:
\begin{itemize}
 \item Can one explicitly describe a fundamental domain?
 \item \emph{How many} smooth rational curves, elliptic fibrations or projective models are there up to the action of the automorphism group?
 \item Can one give generators for the automorphism group?
\end{itemize}
Barth--Peters \cite{BP1983} noted that very general Enriques surfaces do not contain smooth rational curves. 
Hence their nef cone is round - it is the entire positive cone, and they proceed to answer the three questions for very general Enriques surfaces.
\par
Enriques surfaces containing a smooth rational curve are called nodal. 
They form a subset of codimension one in the moduli space of Enriques surfaces. Very general nodal Enriques surfaces are treated by Cossec--Dolgachev \cite{CossecDolgachev1985}
 (see also the works of Allcock \cite{Allcock2018} and Peters--Sterk~\cite{PetersSterk2020}).
\par
 When an Enriques surface is deformed to one containing more rational curves several phenomena working against each other occur. 
 On the one hand the nef cone gets smaller and on the other hand the automorphism group may change drastically.
 Barth -- Peters \cite[p.395]{BP1983} write that they do not know whether one can control these effects.
 Albeit the behaviour of the nef cone and the automorphism group may be erratic, 
 the cone conjecture promises that the fundamental domain on the nef cone stays of finite volume at least. 
 Our first main result (Theorem \ref{thm:volumeformula}) states that we can control the (change of) volume in a precise way under mild assumptions.
\par
To generalize the aforementioned results of
Barth, Peters, Cossec and Dolgachev to Enriques surfaces with more nodes, we introduce the notion of $(\tau,\taubar)$-generic Enriques surfaces, which is closely related to the root invariant
introduced by  Nikulin~\cite{Nikulin1984}.
See the next subsection for the precise definition.
For instance the very general Enriques surface is $(0,0)$-generic,
a very general nodal Enriques surface is $(A_1,A_1)$-generic and
if $Y$ is an Enriques surface that is very general in the moduli
of Enriques surfaces containing $n$ disjoint smooth rational curves,
then $Y$ is  $(nA_1, nA_1)$-generic. 
\newcommand{\Atwo}{\setlength\unitlength{.2truecm}%
\begin{picture}(4.7,1.7)(0,0)%
\put(.9,.5){\circle{.7}}%
\put(3.9,.5){\circle{.7}}%
\put(1.25,.5){\line(1,0){2.3}}%
\end{picture}}%
If $Y$ is very general in the moduli
of Enriques surfaces containing
two smooth rational curves whose dual graph is $\Atwo$
(that is, $Y$ is a very general cuspidal Enriques surface),
then $Y$ is  $(A_2, A_2)$-generic.
\par
Next we give algorithms to compute generators for the automorphism group $\Aut(Y)$, 
a fundamental domain for $\Aut(Y)$ on the nef and big cone $\Nef(Y)$ and orbit representatives for its action on
\begin{eqnarray*}
\Rats(Y) &:=& \textrm{the set of  smooth rational curves on $Y$}, \\
\Ells(Y) &:=& \textrm{the set of elliptic
fibrations $Y\to \P^1$}.
\end{eqnarray*}
We apply Theorem~\ref{thm:volumeformula} and the aforementioned algorithms
to $(\tau, \taubar)$-generic Enriques surfaces. 
This results in our second, series of main results: 
Theorem \ref{thm:main1vol} expresses the volume of the fundamental domain of
 $\Aut(Y)$ on the nef cone $\Nef(Y)$ in terms of the Weyl group of $\tau$, 
Theorem \ref{thm:main2rats} relates the orbits of $\Aut(Y)$
on the set of smooth rational curves $\RRR(Y)$ 
to the connected components of the Dynkin diagram $\tau$ and 
Theorem \ref{thm:main3ells} counts the $\Aut(Y)$-orbits of the set of elliptic fibrations $\EEE(Y)$ and their fiber types.

Our new idea is the lattice theoretic result
obtained in~\cite{BS2019}
(see also Dolgachev--Kondo~\cite[Chapter~10]{DKbook}).
For a lattice $L$ with the intersection form $\intf{-,-}$,
let $L(m)$ denote the lattice with
the same underlying $\Z$-module as $L$ and
with the intersection form $m\,\intf{-,-}$.
A lattice $L$ of rank $n>1$ is said to be \emph{hyperbolic}   if the signature is $(1, n-1)$.
For a positive integer $n$ with $n\bmod 8=2$,
let $L_n$ denote an even unimodular hyperbolic lattice  of rank $n$,
which is unique up to isomorphism.
Borcherds~\cite{Bor1},~\cite{Bor2} developed a method
to calculate the orthogonal group of an even  hyperbolic lattice $S$
by embedding $S$ primitively into $L_{26}$
and using the result of Conway~\cite{Conway1983}.
This method has been applied to the study
of automorphism groups of $K3$ surfaces by many authors.
However, the method often requires impractically heavy computation
(see, for example,~\cite{KKS2014}~and~\cite{Shimada2015}).
\par
On the other hand,
in~\cite{BS2019},
we have classified all primitive embeddings of $L_{10}(2)$ into
$L_{26}$ and showed that they have
a remarkable property~(see Theorems~\ref{thm:17} and~\ref{thm:16simples}) which enables us to calculate automorphism
groups of Enriques surfaces efficiently and explicitly
for the first time.
The resulting speed up (roughly by a factor of $10^{20}$ in the best situation see Remark~\ref{rem:computationsize})
over a more direct approach, allows us to calculate the automorphism groups of the $184$ families of
$(\tau, \taubar)$-generic Enriques surfaces.
\subsection{Definition of $(\tau, \taubar)$-generic Enriques surfaces}\label{subsec:tautaubar}
\newcommand{\sametype}{-} 
First, we define $(\tau, \taubar)$-generic Enriques surfaces.
Let $L$ be a lattice.
We let 
the group $\OG(L)$ of isometries of $L$ act on $L$ from the right,
and write the action as
$v\mapsto v^g$ for $v\in L\tensor\R$ and $g\in \OG(L)$.
We have a natural identification $\OG(L)=\OG(L(m))$ for any non-zero integer $m$.
A vector $v$ of a lattice is called a \emph{$k$-vector} if  $\intf{v,v}=k$.
A $(-2)$-vector is called a \emph{root}.
\begin{definition}
An \emph{$\ADE$-lattice} is an even negative definite lattice generated by  roots.
An $\ADE$-lattice $R$ has a basis consisting of roots 
whose dual graph is a Dynkin diagram of an $\ADE$-type.
This $\ADE$-type  is denoted by $\tau(R)$.
\end{definition}
A \emph{positive half-cone} of a hyperbolic lattice $L$
is one of the two connected components of
$\set{x\in L\tensor\R}{\intf{x, x}>0}$.
Let $\PPP$ be a positive half-cone of a hyperbolic lattice $L$.
We put
\[
\OGP(L):=\set{g\in \OG(L)}{\PPP^g=\PPP}.
\]
In~\cite{RDPEnriques},
we classified  the $\ADE$-sublattices of $\Lten$ up to the action of $\OGP(\Lten)$.
Let $R$ be an $\ADE$-sublattice of $\Lten$,
and $\Rbar$ the primitive closure of $R$ in $L_{10}$.
It turned out that $\Rbar$ is also an $\ADE$-sublattice of $\Lten$.
\begin{proposition}[\cite{RDPEnriques}]\label{prop:184}
{\rm (1)} 
Let $R\sprime$ be another $\ADE$-sublattice of $\Lten$
with the primitive closure $\Rbar\sprime$.
Then $R$ and $R\sprime$ are in the same orbit under the action 
of $\OGP(\Lten)$ 
if and only if $(\tau(R), \tau(\Rbar))=(\tau(R\sprime), \tau(\Rbar\sprime))$.
\par
{\rm (2)}
The pair  $(\tau, \taubar)$ of $\ADE$-types is equal to 
$(\tau(R), \tau(\Rbar))$  of an $\ADE$-sublattice $R$ of $\Lten$
if and only if $(\tau, \taubar)$ is one of the $184$ pairs in Table~\ref{table:184},
where the 3rd column being ``$\sametype$" means $\tau=\taubar$. 
\qed
\end{proposition}
\begin{table}
{\scriptsize
\[
\begin{array}{clllccrl}
\textrm{No.} & \tau(R) & \tau(\Rbar) &\tau(\Rtil) & \texttt{exist} & c_{(\tau, \taubar)} & \texttt{rat} &\texttt{irec} \\ 
\hline 
1 & A_{1} & \sametype & \sametype &   & 1 & 1 & \texttt{96C} \\ 
2 & 2A_{1} & \sametype & \sametype &   & 1 & 2 & \texttt{96C} \\ 
3 & A_{2} & \sametype & \sametype &   & 1 & 1 & \texttt{96C} \\ 
4 & 3A_{1} & \sametype & \sametype &   & 1 & 3 & \texttt{96C} \\ 
5 & A_{2}+A_{1} & \sametype & \sametype &   & 1 & 2 & \texttt{96C} \\ 
6 & A_{3} & \sametype & \sametype &   & 1 & 1 & \texttt{96C} \\ 
7 & 4A_{1} & \sametype & \sametype &   & 1 & 4 & \texttt{96C} \\ 
8 & 4A_{1} & D_{4} & D_{4} &   & 1 & 4 & \texttt{96C} \\ 
9 & A_{2}+2A_{1} & \sametype & \sametype &   & 1 & 3 & \texttt{96C} \\ 
10 & A_{3}+A_{1} & \sametype & \sametype &   & 1 & 2 & \texttt{96C} \\ 
11 & 2A_{2} & \sametype & \sametype &   & 1 & 2 & \texttt{96C} \\ 
12 & A_{4} & \sametype & \sametype &   & 1 & 1 & \texttt{40E} \\ 
13 & D_{4} & \sametype & \sametype &   & 1 & 1 & \texttt{96A} \\ 
14 & 5A_{1} & \sametype & \sametype &   & 1 & 5 & \texttt{96C} \\ 
15 & 5A_{1} & D_{4}+A_{1} & D_{4}+A_{1} &   & 1 & 5 & \texttt{96C} \\ 
16 & A_{2}+3A_{1} & \sametype & \sametype &   & 1 & 4 & \texttt{96C} \\ 
17 & A_{3}+2A_{1} & \sametype & \sametype &   & 1 & 3 & \texttt{96C} \\ 
18 & A_{3}+2A_{1} & D_{5} & D_{5} &   & 1 & 3 & \texttt{96C} \\ 
19 & 2A_{2}+A_{1} & \sametype & \sametype &   & 1 & 3 & \texttt{96C} \\ 
20 & A_{4}+A_{1} & \sametype & \sametype &   & 1 & 2 & \texttt{40E} \\ 
21 & D_{4}+A_{1} & \sametype & \sametype &   & 1 & 2 & \texttt{96A} \\ 
22 & A_{3}+A_{2} & \sametype & \sametype &   & 1 & 2 & \texttt{96C} \\ 
23 & A_{5} & \sametype & \sametype &   & 1 & 1 & \texttt{40E} \\ 
24 & D_{5} & \sametype & \sametype &   & 1 & 1 & \texttt{40A} \\ 
25 & 6A_{1} & D_{4}+2A_{1} & D_{4}+2A_{1} &   & 1 & 6 & \texttt{96C} \\ 
26 & 6A_{1} & D_{6} & D_{6} & \times & 1 & 6 & \texttt{96C} \\ 
27 & A_{2}+4A_{1} & \sametype & \sametype &   & 1 & 5 & \texttt{96C} \\ 
28 & A_{2}+4A_{1} & D_{4}+A_{2} & D_{4}+A_{2} &   & 1 & 5 & \texttt{96C} \\ 
29 & A_{3}+3A_{1} & \sametype & \sametype &   & 1 & 4 & \texttt{96C} \\ 
30 & A_{3}+3A_{1} & D_{5}+A_{1} & D_{5}+A_{1} &   & 1 & 4 & \texttt{96C} \\ 
31 & 2A_{2}+2A_{1} & \sametype & \sametype &   & 1 & 4 & \texttt{96C} \\ 
32 & A_{4}+2A_{1} & \sametype & \sametype &   & 1 & 3 & \texttt{40E} \\ 
33 & D_{4}+2A_{1} & \sametype & \sametype &   & 1 & 3 & \texttt{96A} \\ 
34 & D_{4}+2A_{1} & D_{6} & D_{6} &   & 1 & 3 & \texttt{96A} \\ 
35 & A_{3}+A_{2}+A_{1} & \sametype & \sametype &   & 1 & 3 & \texttt{96C} \\ 
36 & A_{5}+A_{1} & \sametype & \sametype &   & 1 & 2 & \texttt{40E} \\ 
37 & A_{5}+A_{1} & E_{6} & E_{6} &   & 1 & 2 & \texttt{40E} \\ 
38 & D_{5}+A_{1} & \sametype & \sametype &   & 1 & 2 & \texttt{40A} \\ 
39 & 3A_{2} & \sametype & \sametype &   & 1 & 3 & \texttt{96C} \\ 
40 & 3A_{2} & E_{6} & 3A_{2} &   & 1 & 3 & \texttt{96C} \\ 
41 & A_{4}+A_{2} & \sametype & \sametype &   & 1 & 2 & \texttt{40E} \\ 
42 & D_{4}+A_{2} & \sametype & \sametype &   & 1 & 2 & \texttt{96A} \\ 
43 & 2A_{3} & \sametype & \sametype &   & 1 & 2 & \texttt{96A} \\ 
44 & 2A_{3} & D_{6} & D_{6} &   & 1 & 2 & \texttt{96C} \\ 
45 & A_{6} & \sametype & \sametype &   & 1 & 1 & \texttt{40C} \\ 
46 & D_{6} & \sametype & \sametype &   & 1 & 1 & \texttt{40A} \\ 
47 & E_{6} & \sametype & \sametype &   & 1 & 1 & \texttt{20E} \\ 
48 & 7A_{1} & D_{6}+A_{1} & D_{6}+A_{1} & \times & 1 & 7 & \texttt{96C} \\ 
49 & 7A_{1} & E_{7} & E_{7} & \times & 1 & 7 & \texttt{96A} \\ 
50 & A_{2}+5A_{1} & D_{4}+A_{2}+A_{1} & D_{4}+A_{2}+A_{1} &   & 1 & 6 & \texttt{96C} \\ 
51 & A_{3}+4A_{1} & D_{5}+2A_{1} & D_{5}+2A_{1} &   & 1 & 5 & \texttt{96C} \\ 
52 & A_{3}+4A_{1} & D_{4}+A_{3} & D_{4}+A_{3} &   & 1 & 5 & \texttt{96A} \\ 
53 & A_{3}+4A_{1} & D_{7} & D_{7} & \times & 1 & 5 & \texttt{96C} \\ 
54 & 2A_{2}+3A_{1} & \sametype & \sametype &   & 1 & 5 & \texttt{96C} \\ 
55 & A_{4}+3A_{1} & \sametype & \sametype &   & 1 & 4 & \texttt{40E} \\ 
56 & D_{4}+3A_{1} & D_{6}+A_{1} & D_{6}+A_{1} &   & 1 & 4 & \texttt{96A} \\ 
57 & D_{4}+3A_{1} & E_{7} & E_{7} & \times & 1 & 4 & \texttt{96A} \\ 
58 & A_{3}+A_{2}+2A_{1} & \sametype & \sametype &   & 1 & 4 & \texttt{96C} \\ 
59 & A_{3}+A_{2}+2A_{1} & D_{5}+A_{2} & D_{5}+A_{2} &   & 1 & 4 & \texttt{96C} \\ 
60 & A_{5}+2A_{1} & \sametype & \sametype &   & 1 & 3 & \texttt{40E} \\ 
61 & A_{5}+2A_{1} & E_{6}+A_{1} & E_{6}+A_{1} &   & 1 & 3 & \texttt{40E} \\ 
62 & D_{5}+2A_{1} & \sametype & \sametype &   & 1 & 3 & \texttt{40A} \\ 
63 & D_{5}+2A_{1} & D_{7} & D_{7} &   & 1 & 3 & \texttt{40A} \\ 
64 & 3A_{2}+A_{1} & \sametype & \sametype &   & 1 & 4 & \texttt{96C} \\ 
65 & 3A_{2}+A_{1} & E_{6}+A_{1} & 3A_{2}+A_{1} &   & 1 & 4 & \texttt{96C} \\ 
66 & A_{4}+A_{2}+A_{1} & \sametype & \sametype &   & 1 & 3 & \texttt{40E} \\ 
67 & D_{4}+A_{2}+A_{1} & \sametype & \sametype &   & 1 & 3 & \texttt{96A} \\ 
68 & 2A_{3}+A_{1} & \sametype & \sametype &   & 1 & 3 & \texttt{96A} \\ 
69 & 2A_{3}+A_{1} & D_{6}+A_{1} & D_{6}+A_{1} &   & 1 & 3 & \texttt{96C} \\ 
70 & 2A_{3}+A_{1} & E_{7} & D_{6}+A_{1} &   & 1 & 3 & \texttt{96C} \\ 
\end{array}

\]
}
\vskip .1cm
{\scriptsize 
\parbox{10cm}{For $\tau(R), \tau(\Rbar), \tau(\Rtil)$, see Propositions~\ref{prop:184} and~\ref{prop:Rtil}.
For~$\texttt{exist}$, see Proposition~\ref{prop:5th}.
For~$c_{(\tau, \taubar)} $, see Theorem~\ref{thm:main1vol}.
For~$ \texttt{rat} $, see Theorem~\ref{thm:main2rats}.
For~$\texttt{irec}$, see Section~\ref{subsec:construction}.}
}
\vskip .1cm
\caption{$\ADE$-sublattices of  $\Lten$ (continues)}
\label{table:184}
\end{table}
\setcounter{table}{0}
\begin{table}
{\scriptsize
\[
\begin{array}{clllccrl}
\textrm{No.} & \tau(R) & \tau(\Rbar) &\tau(\Rtil) & \texttt{exist} & c_{(\tau, \taubar)} & \texttt{rat} &\texttt{irec} \\ 
\hline 
71 & A_{6}+A_{1} & \sametype & \sametype &   & 1 & 2 & \texttt{40C} \\ 
72 & D_{6}+A_{1} & \sametype & \sametype &   & 1 & 2 & \texttt{40A} \\ 
73 & D_{6}+A_{1} & E_{7} & E_{7} &   & 1 & 2 & \texttt{40A} \\ 
74 & E_{6}+A_{1} & \sametype & \sametype &   & 1 & 2 & \texttt{20E} \\ 
75 & A_{3}+2A_{2} & \sametype & \sametype &   & 1 & 3 & \texttt{96C} \\ 
76 & A_{5}+A_{2} & \sametype & \sametype &   & 1 & 2 & \texttt{40E} \\ 
77 & A_{5}+A_{2} & E_{7} & A_{5}+A_{2} &   & 1 & 2 & \texttt{40E} \\ 
78 & D_{5}+A_{2} & \sametype & \sametype &   & 1 & 2 & \texttt{40A} \\ 
79 & A_{4}+A_{3} & \sametype & \sametype &   & 1 & 2 & \texttt{40E} \\ 
80 & D_{4}+A_{3} & \sametype & \sametype &   & 1 & 2 & \texttt{20F} \\ 
81 & D_{4}+A_{3} & D_{7} & D_{7} &   & 1 & 2 & \texttt{96A} \\ 
82 & A_{7} & \sametype & \sametype &   & 1 & 1 & \texttt{20D} \\ 
83 & A_{7} & E_{7} & E_{7} &   & 1 & 1 & \texttt{40C} \\ 
84 & D_{7} & \sametype & \sametype &   & 1 & 1 & \texttt{20B} \\ 
85 & E_{7} & \sametype & \sametype &   & 1 & \times2 & \texttt{20A} \\ 
86 & 8A_{1} & E_{7}+A_{1} & E_{7}+A_{1} & \times & 1 & 8 & \texttt{96A} \\ 
87 & 8A_{1} & D_{8} & D_{8} & \times & 1 & 8 & \texttt{96B} \\ 
88 & 8A_{1} & E_{8} & E_{8} & \times & \rlap{\hskip -2pt see Remark~\ref{rem:88and146}} &  &  \\ 
89 & A_{2}+6A_{1} & D_{6}+A_{2} & D_{6}+A_{2} & \times & 1 & 7 & \texttt{96C} \\ 
90 & A_{3}+5A_{1} & D_{7}+A_{1} & D_{7}+A_{1} & \times & 1 & 6 & \texttt{96C} \\ 
91 & A_{4}+4A_{1} & D_{4}+A_{4} & D_{4}+A_{4} &   & 1 & 5 & \texttt{40E} \\ 
92 & D_{4}+4A_{1} & E_{7}+A_{1} & E_{7}+A_{1} & \times & 1 & 5 & \texttt{96A} \\ 
93 & D_{4}+4A_{1} & D_{8} & D_{8} & \times & 1 & 5 & \texttt{96A} \\ 
94 & D_{4}+4A_{1} & E_{8} & E_{8} & \times & 2 & 5 & \texttt{96A} \\ 
95 & A_{3}+A_{2}+3A_{1} & D_{5}+A_{2}+A_{1} & D_{5}+A_{2}+A_{1} &   & 1 & 5 & \texttt{96C} \\ 
96 & A_{5}+3A_{1} & E_{6}+2A_{1} & E_{6}+2A_{1} &   & 1 & 4 & \texttt{40E} \\ 
97 & D_{5}+3A_{1} & D_{7}+A_{1} & D_{7}+A_{1} &   & 1 & 4 & \texttt{40A} \\ 
98 & 3A_{2}+2A_{1} & E_{6}+2A_{1} & 3A_{2}+2A_{1} &   & 1 & 5 & \texttt{96C} \\ 
99 & A_{4}+A_{2}+2A_{1} & \sametype & \sametype &   & 1 & 4 & \texttt{40E} \\ 
100 & D_{4}+A_{2}+2A_{1} & D_{6}+A_{2} & D_{6}+A_{2} &   & 1 & 4 & \texttt{96A} \\ 
101 & 2A_{3}+2A_{1} & E_{7}+A_{1} & D_{6}+2A_{1} &   & 1 & 4 & \texttt{96C} \\ 
102 & 2A_{3}+2A_{1} & D_{5}+A_{3} & D_{5}+A_{3} &   & 1 & 4 & \texttt{96A} \\ 
103 & 2A_{3}+2A_{1} & D_{8} & D_{8} & \times & 1 & 4 & \texttt{96C} \\ 
104 & 2A_{3}+2A_{1} & E_{8} & D_{8} & \times & 1 & 4 & \texttt{96C} \\ 
105 & A_{6}+2A_{1} & \sametype & \sametype &   & 1 & 3 & \texttt{40C} \\ 
106 & D_{6}+2A_{1} & E_{7}+A_{1} & E_{7}+A_{1} &   & 1 & 3 & \texttt{40A} \\ 
107 & D_{6}+2A_{1} & D_{8} & D_{8} &   & 1 & 3 & \texttt{40A} \\ 
108 & D_{6}+2A_{1} & E_{8} & E_{8} & \times & 2 & 3 & \texttt{40A} \\ 
109 & E_{6}+2A_{1} & \sametype & \sametype &   & 1 & 3 & \texttt{20E} \\ 
110 & A_{3}+2A_{2}+A_{1} & \sametype & \sametype &   & 1 & 4 & \texttt{96C} \\ 
111 & A_{5}+A_{2}+A_{1} & \sametype & \sametype &   & 1 & 3 & \texttt{40E} \\ 
112 & A_{5}+A_{2}+A_{1} & E_{7}+A_{1} & A_{5}+A_{2}+A_{1} &   & 1 & 3 & \texttt{40E} \\ 
113 & A_{5}+A_{2}+A_{1} & E_{6}+A_{2} & E_{6}+A_{2} &   & 1 & 3 & \texttt{40E} \\ 
114 & A_{5}+A_{2}+A_{1} & E_{8} & E_{6}+A_{2} &   & 1 & 3 & \texttt{40E} \\ 
115 & D_{5}+A_{2}+A_{1} & \sametype & \sametype &   & 1 & 3 & \texttt{40A} \\ 
116 & A_{4}+A_{3}+A_{1} & \sametype & \sametype &   & 1 & 3 & \texttt{40E} \\ 
117 & D_{4}+A_{3}+A_{1} & D_{7}+A_{1} & D_{7}+A_{1} &   & 1 & 3 & \texttt{96A} \\ 
118 & A_{7}+A_{1} & \sametype & \sametype &   & 1 & 2 & \texttt{20D} \\ 
119 & A_{7}+A_{1} & E_{7}+A_{1} & E_{7}+A_{1} &   & 1 & 2 & \texttt{40C} \\ 
120 & A_{7}+A_{1} & E_{8} & E_{7}+A_{1} &   & 1 & 2 & \texttt{40C} \\ 
121 & D_{7}+A_{1} & \sametype & \sametype &   & 1 & 2 & \texttt{20B} \\ 
122 & E_{7}+A_{1} & \sametype & \sametype &   & 1 & \times3 & \texttt{20A} \\ 
123 & E_{7}+A_{1} & E_{8} & E_{8} &   & 2 & \times3 & \texttt{20A} \\ 
124 & 4A_{2} & E_{6}+A_{2} & 4A_{2} &   & 1 & 4 & \texttt{96C} \\ 
125 & 4A_{2} & E_{8} & 4A_{2} &   & 1 & 4 & \texttt{96C} \\ 
126 & A_{4}+2A_{2} & \sametype & \sametype &   & 1 & 3 & \texttt{40E} \\ 
127 & 2A_{3}+A_{2} & D_{6}+A_{2} & D_{6}+A_{2} &   & 1 & 3 & \texttt{96C} \\ 
128 & A_{6}+A_{2} & \sametype & \sametype &   & 1 & 2 & \texttt{40C} \\ 
129 & D_{6}+A_{2} & \sametype & \sametype &   & 1 & 2 & \texttt{40A} \\ 
130 & E_{6}+A_{2} & \sametype & \sametype &   & 1 & 2 & \texttt{20E} \\ 
131 & E_{6}+A_{2} & E_{8} & E_{6}+A_{2} &   & 1 & 2 & \texttt{20E} \\ 
132 & A_{5}+A_{3} & \sametype & \sametype &   & 1 & 2 & \texttt{40E} \\ 
133 & D_{5}+A_{3} & \sametype & \sametype &   & 1 & 2 & \texttt{20F} \\ 
134 & D_{5}+A_{3} & D_{8} & D_{8} &   & 1 & 2 & \texttt{40A} \\ 
135 & D_{5}+A_{3} & E_{8} & D_{8} &   & 1 & 2 & \texttt{40A} \\ 
136 & 2A_{4} & \sametype & \sametype &   & 1 & 2 & \texttt{40E} \\ 
137 & 2A_{4} & E_{8} & 2A_{4} &   & 1 & 2 & \texttt{40E} \\ 
138 & D_{4}+A_{4} & \sametype & \sametype &   & 1 & 2 & \texttt{20F} \\ 
139 & A_{8} & \sametype & \sametype &   & 1 & 1 & \texttt{20D} \\ 
140 & A_{8} & E_{8} & A_{8} &   & 1 & 1 & \texttt{20D} \\ 
\end{array}

\]
}
\vskip .1cm
\caption{$\ADE$-sublattices of  $\Lten$ (continued and continues)}
\end{table}
\setcounter{table}{0}
\begin{table}
{\scriptsize
\[
\begin{array}{clllccrl}
\textrm{No.} & \tau(R) & \tau(\Rbar) &\tau(\Rtil) & \texttt{exist} & c_{(\tau, \taubar)} & \texttt{rat} &\texttt{irec} \\ 
\hline 
141 & 2D_{4} & D_{8} & D_{8} &   & 1 & 2 & \texttt{20F} \\ 
142 & 2D_{4} & E_{8} & E_{8} & \times & 1 & \times1 & \texttt{96A} \\ 
143 & D_{8} & \sametype & \sametype &   & 1 & 1 & \texttt{12B} \\ 
144 & D_{8} & E_{8} & E_{8} &   & 2 & \times2 & \texttt{20B} \\ 
145 & E_{8} & \sametype & \sametype &   & 2 & \times4 & \texttt{12A} \\ 
146 & 9A_{1} & E_{8}+A_{1} & E_{8}+A_{1} & \times & \rlap{\hskip -2pt see Remark~\ref{rem:88and146}} &  &  \\ 
147 & A_{2}+7A_{1} & E_{7}+A_{2} & E_{7}+A_{2} & \times & 1 & 8 & \texttt{96A} \\ 
148 & A_{3}+6A_{1} & D_{9} & D_{9} & \times & 1 & 7 & \texttt{96B} \\ 
149 & D_{4}+5A_{1} & E_{8}+A_{1} & E_{8}+A_{1} & \times & 2 & 6 & \texttt{96A} \\ 
150 & D_{5}+4A_{1} & D_{9} & D_{9} & \times & 1 & 5 & \texttt{40A} \\ 
151 & D_{4}+A_{2}+3A_{1} & E_{7}+A_{2} & E_{7}+A_{2} & \times & 1 & 5 & \texttt{96A} \\ 
152 & 2A_{3}+3A_{1} & E_{8}+A_{1} & D_{8}+A_{1} & \times & 1 & 5 & \texttt{96C} \\ 
153 & D_{6}+3A_{1} & E_{8}+A_{1} & E_{8}+A_{1} & \times & 2 & 4 & \texttt{40A} \\ 
154 & A_{5}+A_{2}+2A_{1} & E_{8}+A_{1} & E_{6}+A_{2}+A_{1} &   & 1 & 4 & \texttt{40E} \\ 
155 & A_{4}+A_{3}+2A_{1} & D_{5}+A_{4} & D_{5}+A_{4} &   & 1 & 4 & \texttt{40E} \\ 
156 & D_{4}+A_{3}+2A_{1} & D_{9} & D_{9} & \times & 1 & 4 & \texttt{96A} \\ 
157 & A_{7}+2A_{1} & E_{8}+A_{1} & E_{7}+2A_{1} &   & 1 & 3 & \texttt{40C} \\ 
158 & D_{7}+2A_{1} & D_{9} & D_{9} &   & 1 & 3 & \texttt{20B} \\ 
159 & E_{7}+2A_{1} & E_{8}+A_{1} & E_{8}+A_{1} &   & 2 & \times4 & \texttt{20A} \\ 
160 & 4A_{2}+A_{1} & E_{8}+A_{1} & 4A_{2}+A_{1} & \times & 1 & 5 & \texttt{40E} \\ 
161 & 2A_{3}+A_{2}+A_{1} & E_{7}+A_{2} & D_{6}+A_{2}+A_{1} &   & 1 & 4 & \texttt{96C} \\ 
162 & A_{6}+A_{2}+A_{1} & \sametype & \sametype &   & 1 & 3 & \texttt{40C} \\ 
163 & D_{6}+A_{2}+A_{1} & E_{7}+A_{2} & E_{7}+A_{2} &   & 1 & 3 & \texttt{40A} \\ 
164 & E_{6}+A_{2}+A_{1} & E_{8}+A_{1} & E_{6}+A_{2}+A_{1} &   & 1 & 3 & \texttt{20E} \\ 
165 & A_{5}+A_{3}+A_{1} & E_{6}+A_{3} & E_{6}+A_{3} &   & 1 & 3 & \texttt{40E} \\ 
166 & D_{5}+A_{3}+A_{1} & E_{8}+A_{1} & D_{8}+A_{1} &   & 1 & 3 & \texttt{40A} \\ 
167 & 2A_{4}+A_{1} & E_{8}+A_{1} & 2A_{4}+A_{1} &   & 1 & 3 & \texttt{40E} \\ 
168 & A_{8}+A_{1} & \sametype & \sametype &   & 1 & 2 & \texttt{20D} \\ 
169 & A_{8}+A_{1} & E_{8}+A_{1} & A_{8}+A_{1} &   & 1 & 2 & \texttt{20D} \\ 
170 & 2D_{4}+A_{1} & E_{8}+A_{1} & E_{8}+A_{1} & \times & 1 & \times2 & \texttt{96A} \\ 
171 & D_{8}+A_{1} & E_{8}+A_{1} & E_{8}+A_{1} &   & 2 & \times3 & \texttt{20B} \\ 
172 & E_{8}+A_{1} & \sametype & \sametype &   & 2 & \times5 & \texttt{12A} \\ 
173 & A_{3}+3A_{2} & E_{6}+A_{3} & A_{3}+3A_{2} &   & 1 & 4 & \texttt{96C} \\ 
174 & A_{5}+2A_{2} & E_{7}+A_{2} & A_{5}+2A_{2} &   & 1 & 3 & \texttt{40E} \\ 
175 & A_{7}+A_{2} & E_{7}+A_{2} & E_{7}+A_{2} &   & 1 & 2 & \texttt{40C} \\ 
176 & E_{7}+A_{2} & \sametype & \sametype &   & 1 & \times3 & \texttt{20A} \\ 
177 & 3A_{3} & D_{9} & D_{9} & \times & 1 & 3 & \texttt{96C} \\ 
178 & D_{6}+A_{3} & D_{9} & D_{9} &   & 1 & 2 & \texttt{40A} \\ 
179 & E_{6}+A_{3} & \sametype & \sametype &   & 1 & 2 & \texttt{20E} \\ 
180 & A_{5}+A_{4} & \sametype & \sametype &   & 1 & 2 & \texttt{40E} \\ 
181 & D_{5}+A_{4} & \sametype & \sametype &   & 1 & 2 & \texttt{20F} \\ 
182 & A_{9} & \sametype & \sametype &   & 1 & 1 & \texttt{20D} \\ 
183 & D_{5}+D_{4} & D_{9} & D_{9} &   & 1 & 2 & \texttt{20F} \\ 
184 & D_{9} & \sametype & \sametype &   & 1 & \times2 & \texttt{12B} \\ 
\end{array}

\]
}
\vskip .1cm
\caption{$\ADE$-sublattices of  $\Lten$ (continued)}
\end{table}
Let $R$ be an $\ADE$-sublattice of $\Lten$.
We denote by $\embR\colon R\inj \Lten$ the inclusion.
We define $M_R$ to be 
the $\Z$-submodule of   $(\Lten(2)\oplus R(2))\tensor \Q$
generated by $\Lten(2)$ and 
$(\embR(v), \pm v)/2\in(\Lten\oplus R)\tensor \Q$, 
where  $v$ runs through $R$,
and equip $M_R$ with  
an intersection form by extending the intersection form of $\Lten(2)\oplus R(2)$.
By definition,
 $M_R$ is an even hyperbolic lattice
with a  chosen primitive embedding $\embLM_R \colon \Lten (2) \inj M_R$.
If $R\sprime$ is another $\ADE$-sublattice of $\Lten$
such that   
$(\tau(R\sprime), \tau(\Rbar\sprime))=(\tau(R), \tau(\Rbar))$, 
then,
by Proposition~\ref{prop:184},  we have 
an isometry $g\colon \Lten\isom \Lten$ 
that induces an  isometry $g|_R\colon R\isom R\sprime$, and hence 
we obtain an isometry 
 $\tilde{g}\colon M_R\isom M_{R\sprime}$ 
induced by $g \oplus g|_R$,
which  makes
the following  diagram  commutative:
\[
\begin{array}{ccc}
\Lten (2) &\maprightinjsp{\embLM_R}  & M_R
\mystruth{12pt} \\
\llap{\scriptsize $g$\;}\downisom &&  \llap{\scriptsize $\tilde{g}$\;}\downisom \\
\Lten (2) &\maprightinjsp{\embLM_{R\sprime}} & M_{R\sprime} \rlap{.}
\end{array}
\]
By an explicit calculation, we obtain the following:
\begin{proposition}\label{prop:Rtil}
Let $R$ be an $\ADE$-sublattice of $\Lten$.
Then the orthogonal complement of $\embLM_R \colon \Lten (2) \inj M_R$
is isomorphic to $\Rtil(2)$ for some 
$\ADE$-lattice $\Rtil$.
In the 4th column of Table~\ref{table:184}, 
we give the $\ADE$-type $\tau(\Rtil)$ of  $\Rtil$,
where  ``$\sametype$" means $\tau(R)=\tau(\Rtil)$. 
\qed
\end{proposition}
Let $Y$ be an Enriques surface.
We denote by $\SY$ the lattice of numerical equivalence classes of
divisors of  $Y$.
It is well-known that $\SY$ is isomorphic to 
$L_{10}$.
Let $\pi\colon X\to Y$ be the universal covering of $Y$,
and let $\SX$ denote the lattice of numerical equivalence classes of
divisors of the $K3$ surface $X$.
Then the \'etale double covering $\pi$ induces a primitive embedding
\[
\pi^*\colon \SY(2)\inj \SX.
\]
\begin{definition}\label{def:tautaubargeneric}
Let $(\tau, \taubar)$ be one of the $184$ pairs in Table~\ref{table:184}.
An Enriques surface $Y$ is said to be \emph{$(\tau, \taubar)$-generic} 
if the following conditions are satisfied.
%
\begin{enumerate}[(i)]
\item 
Let $T_X$ be the transcendental lattice of $X$,
and $\omega$ a non-zero holomorphic $2$-form of $X$,
so that we have $\C\omega=H^{2, 0}(X) \subset T_X\tensor\C$.
Then the group
\[
\OG(T_X, \omega):=\set{g\in \OG(T_X)}{\textrm{$g$ preserves
$\C\omega\subset T_X\tensor\C$}}
\]
is equal to $\{\pm 1\}$.
\item 
Let $R$ be an $\ADE$-sublattice of $\Lten$ with $(\tau(R), \tau(\Rbar))=(\tau, \taubar)$.
Then there exist isometries $g\colon \Lten\isom S_Y$
and $\tilde{g}\colon M_R\isom S_X$
that  make 
the following commutative diagram
\begin{equation}\label{eq:generic-diagram}
\begin{array}{ccc}
\Lten (2) &\maprightinjsp{\embLM_R}  & M_R
\mystruth{12pt} \\
\llap{\scriptsize $g$\;}\downisom &&  \llap{\scriptsize $\tilde{g}$\;}\downisom \\
S_Y(2) &\maprightinjsp{\pi^*} & S_X \rlap{.}
\end{array}
\end{equation}
\end{enumerate}
\end{definition}
The numbering of the $\ADE$-types in Table~\ref{table:184}~of the present article 
is the same as the numbering in Table 1.1~of our previous paper~\cite{RDPEnriques}, and hence
the 1st-3rd columns of the two tables 
are identical.
By definition, a $(\tau, \taubar)$-generic Enriques surface exists 
if and only if the 4th column of the corresponding row of 
Table 1.1~of~\cite{RDPEnriques} contains $0$.
Hence we obtain the following:
\begin{proposition}[\cite{RDPEnriques}]\label{prop:5th}
A $(\tau, \taubar)$-generic Enriques surface exists 
if and only if the 5th column of 
the corresponding row in Table~\ref{table:184}
is not marked by $\times$.
\qed
\end{proposition}
%
%
\begin{example}\label{example:8A1}
Consider the case 
where $\tau(R)=8 A_1$ (Nos.~86, 87, 88).
By~\cite{RDPEnriques},
we have no $S_X$ (No.~86),
or $S_X/M_R$ is non-trivial ($(\Z/2\Z)^2$ for No.~87 and $(\Z/2\Z)^3$ for No.~88), 
that is, the inclusion $\tilde{g}$ is not an isometry.
Hence there  exist no $(\tau, \taubar)$-generic Enriques surfaces with $\tau=8A_1$,
even though there exist surfaces with $8$ ordinary nodes birational to Enriques surfaces.
\end{example}
\begin{remark}
The geometry of Enriques surfaces with $\OG(T_X, \period)=\{\pm 1\}$  but  
with $S_X/M_R$ being \emph{non-trivial} and finite is left for future studies.
\end{remark}
Let $\PPP_Y$ (resp.~$\PPP_X$)
be the positive half-cone of $\SY$
(resp.~$\SX$)
containing an ample class.
We regard $\PPP_Y$ as a subspace of $\PPP_X$ by
the embedding $\pi^*\tensor\R$.
We  put
\begin{eqnarray*}
\NefX&:=&
\set{x\in \PPP_X}{\intf{x, [\wtC]}\ge 0 \;\;
\textrm{for all curves $\wtC$ on $X$}}, \\
\NefY&:=&
\set{y\in \PPP_Y\hskip 1pt}{\intf{y, [C]}\ge 0 \;\;\textrm{for all curves $C$ on $Y$}}=
\PPP_Y\cap \NefX,
\end{eqnarray*}
where $[D]$ is the class of a divisor $D$.
The following will be proved in Section~\ref{subsec:proofprop:nefisom}.
\begin{proposition}\label{prop:nefisom}
Let $Y$ and $Y\sprime$  be $(\tau, \taubar)$-generic Enriques surfaces
with the universal coverings $\pi\colon X\to Y$ and $\pi\sprime\colon X\sprime\to Y\sprime$,
respectively.
Then there exist isometries $\psi_X\colon \SX\isom S_{X\sprime}$ and 
$\psi_Y\colon \SY\isom S_{Y\sprime}$
that make the diagram
\begin{equation}\label{eq:nefisomdiagram}
\begin{array}{ccc}
\SY(2)& \maprightsp{\pi^*} & \SX \\
\mapdownleft{\psi_Y} & &\mapdownright{\psi_X}\\
S_{Y\sprime} (2) & \maprightsp{\pi^{\prime*}}& S_{X\sprime} 
\end{array}
\end{equation}
commutative and that induce $\NefX\cong \Nef_{X\sprime}$ and $\NefY\cong \Nef_{Y\sprime}$.
\end{proposition}
We denote by $\aut(Y)$ the image of the natural representation
$\Aut(Y) \to \OGP(\SY)$.
We embed the set $\Rats(Y)$ 
of smooth rational curves $C$ on $Y$ into $\SY$ by $C\mapsto [C]$,
and the set  $\Ells(Y)$ 
of elliptic fibrations $\phi\colon Y\to \P^1$ into $\SY$ by $\phi\mapsto [F]/2$,
where $F$ is a general fiber of $\phi$.
In Section~\ref{sec:proofs}, we will see that 
$\aut(Y)$ and its actions on $\NefY$, $\Rats(Y)$, $\Ells(Y)$
depend only on the  data 
$\pi^* \colon \SY(2)\inj \SX$ and 
$\NefX$.
Therefore we obtain the following: 
\begin{corollary}\label{cor:nefisom}
Let $Y$ and $Y\sprime$ be 
as in Proposition~\ref{prop:nefisom}.
Then there exist an isomorphism $\aut(Y)\cong \aut(Y\sprime)$
and  bijections $\Rats(Y)\cong \Rats(Y\sprime)$ and $\Ells(Y)\cong \Ells(Y\sprime)$
that are compatible with $\aut(Y)\cong \aut(Y\sprime)$.
\qed
\end{corollary}
\begin{remark}
The root invariant of a $(\tau, \taubar)$-generic Enriques surface
(defined by Nikulin~\cite{Nikulin1984})
is equal to $(\tau, \Ker \xi)$, where $\xi\colon R\tensor  \F_2\to L_{10}\tensor \F_2$
is the linear homomorphism induced by the inclusion $R\inj L_{10}$
of  the $\ADE$-sublattice $R$ of $L_{10}$ such that $(\tau, \taubar)=(\tau(R), \tau(\Rbar))$.
\end{remark}
\subsection{Chambers}\label{subsec:chambers}
Before we state our geometric results,
we define the notion of \emph{chambers} of hyperbolic lattices,
and recall the classical result of Vinberg~\cite{Vinberg1973}.
\par
A root $r$ of an even lattice
$L$ defines  the \emph{reflection}  $s_r\colon x\mapsto x+\intf{x, r}r$ of $L$
with respect to $r$.
The \emph{Weyl group} $W(L)$ of $L$ is 
the subgroup  of  $\OG(L)$
generated by all the reflections $s_r$
with respect to the roots of $L$.
Let $L$ be an even hyperbolic lattice
with a positive half-cone $\PPP$.
For $v\in L\tensor\R$ with $\intf{v,v}<0$,
let $(v)\sperp$ denote the hyperplane of $\PPP$ defined by
$\intf{x, v}=0$.
Then we have    $W(L)\subset \OGP(L)$,
and the action of $s_r$ on $\PPP$ is the reflection into the mirror $(r)\sperp$.
A closed subset $D$ of $\PPP$ is called a \emph{chamber}
if $D$ contains a non-empty open subset of $\PPP$
and $D$ is defined by inequalities
\[
\intf{x, v_i}\ge 0\quad(i\in I),
\]
where $\{(v_i)\sperp\}_{i\in I}$ is a locally finite family of hyperplanes
of $\PPP$.
A \emph{wall} of a chamber $D$ is
a closed subset of $D$ of the form $D\cap (v)\sperp$
such that $(v)\sperp$ is disjoint from the interior of $D$
and that $D\cap (v)\sperp$  contains a non-empty open subset of $(v)\sperp$.
We say that  a vector $v\in L\tensor\R$ \emph{defines a wall $D\cap (v)\sperp$ of $D$}
if $D\cap (v)\sperp$ is a wall of $D$ and $\intf{x, v}>0$
holds for one (and hence any) point $x$ in the interior of $D$.
We say that a closed subset $A$ of $\PPP$ is
\emph{tessellated by a set $\{D_j\}_{j\in J}$
of chambers} if
$A$ is the union of $D_j$ ($j\in J$)  and
the interiors of two distinct chambers $D_j$ and $D_{j\sprime}$ in
the family $\{D_j\}_{j\in J}$ have no common  points.
\begin{definition}
Let $L$ be an even hyperbolic lattice
with a positive half-cone $\PPP$.
An \emph{$L$-chamber} is 
the closure in $\PPP$ of a connected component of
\[
\PPP\;\;\setminus\;\; \bigcup_{r}\; (r)\sperp,
\]
where $r$ runs through the set of roots of $L$.
For an $L$-chamber $D$, 
we denote 
the stabilizer of $D$ by
\[
\OG(L, D):=\set{g\in \OGP(L)}{D^g=D}.
\]
\end{definition}
\begin{remark}
In Section~\ref{subsec:LMchams}, we extend the notion
of $L$-chambers to the notion of $L/M$-chambers
in the positive-half cone $\PPP_M$ of 
a primitive lattice $M$ of $L$. 
\end{remark}
The group $\OGP(L)$ acts on the set of $L$-chambers.
The action of the subgroup $W(L)$ of $\OGP(L)$ on  
this set  is free and transitive. 
Hence an $L$-chamber is a standard fundamental domain of the Weyl group $W(L)$.
Let $D$ be an $L$-chamber.
Then we have $\OGP(L)=W(L)\semidirectproduct \OG(L, D)$,
and moreover,  $W(L)$ is generated by the reflections $s_r$
with respect to the roots $r$ that define the walls of $D$.
\par
Recall that  $L_{10}$ is an even unimodular hyperbolic lattice of rank $10$.
Then $L_{10}$ has a basis $e_1, \dots, e_{10}$
consisting of roots
whose dual graph is given in Figure~\ref{fig:E10}.
Let  $\PPP_{10}$ be the positive half-cone of $L_{10}$
containing $e_1\dual+\cdots+e_{10}\dual$,
where $\{e_1\dual, \dots, e_{10}\dual\}$ is the  basis of $L_{10}\dual=L_{10}$
dual to $\{e_1, \dots, e_{10}\}$.
\begin{theorem}[Vinberg~\cite{Vinberg1973}]
The chamber $D_0$ in $\PPP_{10}$
defined by $\intf{x, e_i}\ge 0$ for $i=1, \dots, 10$
is an $L_{10}$-chamber, and
$\{e_1, \dots, e_{10}\}$ is the set of roots defining walls of $D_0$.
\qed
\label{thm:Vinberg}
\end{theorem}
\begin{definition}
We call an $L_{10}$-chamber a \emph{Vinberg chamber}.
\end{definition}
Let $D_0$ be a Vinberg chamber.
Since the  dual graph in Figure~\ref{fig:E10} has no non-trivial symmetries, 
we have $\OG(L_{10}, D_0)=\{1\}$
and hence 
\begin{equation}\label{eq:OGPWL10}
\OGP(L_{10})=W(L_{10}).
\end{equation}

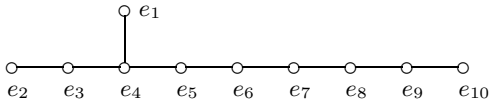
\begin{figure}
\def\ha{40}
\def\hav{37}
\def\hd{25}
\def\hdv{22}
\def\he{10}
\def\hev{7}
\setlength{\unitlength}{1.25mm}
{\small
\begin{picture}(80,11)(-5, 6)
\put(22, 16){\circle{1}}
\put(23.5, 15.5){$e\sb 1$}
\put(22, 10.5){\line(0,1){5}}
\put(9.5, \hev){$e\sb 2$}
\put(15.5, \hev){$e\sb 3$}
\put(21.5, \hev){$e\sb 4$}
\put(27.5, \hev){$e\sb 5$}
\put(33.5, \hev){$e\sb 6$}
\put(39.5, \hev){$e\sb 7$}
\put(45.5, \hev){$e\sb {8}$}
\put(51.5, \hev){$e\sb {9}$}
\put(57.5, \hev){$e\sb {10}$}
\put(10, \he){\circle{1}}
\put(16, \he){\circle{1}}
\put(22, \he){\circle{1}}
\put(28, \he){\circle{1}}
\put(34, \he){\circle{1}}
\put(40, \he){\circle{1}}
\put(46, \he){\circle{1}}
\put(52, \he){\circle{1}}
\put(58, \he){\circle{1}}
\put(10.5, \he){\line(5, 0){5}}
\put(16.5, \he){\line(5, 0){5}}
\put(22.5, \he){\line(5, 0){5}}
\put(28.5, \he){\line(5, 0){5}}
\put(34.5, \he){\line(5, 0){5}}
\put(40.5, \he){\line(5, 0){5}}
\put(46.5, \he){\line(5, 0){5}}
\put(52.5, \he){\line(5, 0){5}}
\end{picture}
}
\caption{The basis $e_1, \dots, e_{10}$ of $L_{10}$}\label{fig:E10}
\end{figure}
\subsection{Main results}\label{subsec:mainresults}
We investigate the geometry of a $(\tau, \taubar)$-generic Enriques surface $Y$.
In particular, we calculate a finite generating set of $\aut(Y)$
and the action of $\aut(Y)$ on $\NefY$, $\Rats(Y)$ and $\Ells(Y)$.
\begin{remark}\label{rem:nongeometric}
Since our approach relies on the interplay between lattice theory and hyperbolic geometry, 
we can, except for the cases Nos.~88~and~146 in Table~\ref{table:184},
calculate the geometric data of a hypothetical $(\tau, \taubar)$-generic Enriques surface
even when it is not realized by an actual  
Enriques surface.
(See Remark~\ref{rem:nonexisting}).
\end{remark}
%
\begin{remark}\label{rem:88and146} 
For the cases Nos.~88~and~146 in Table~\ref{table:184},
we cannot construct  $S_X$ by means of the method described in Section~\ref{subsec:construction}.
Since there do not exist $(\tau, \bar\tau)$-generic Enriques surfaces in these cases,
we leave them blank.
\end{remark}
Let $Y$ be an Enriques surface.
Recall that $\aut(Y)\subset  \OGP(\SY)$ 
is the image of the natural homomorphism
$\Aut(Y)\to \OGP(\SY)$.
Since $\SY$ is isomorphic to $L_{10}$,
we have Vinberg chambers in the positive half-cone $\PPP_Y$.
Since $\NefY$ is bounded by
$([C])\sperp$,
where $C$ runs through $\Rats(Y)$, and $\intf{[C], [C]}=-2$, 
the cone $\NefY$ is  tessellated by Vinberg chambers.
We put
\[
\StdFD(\NefY):=
\textrm{the set of Vinberg chambers contained in $\NefY$}, 
\]
on which  $\aut(Y)$ acts, and define
\[
\vol (\NefY/\aut(Y)):=
\textrm{the number of orbits of the action of $\aut(Y)$ on 
$\StdFD(\NefY)$}.
\]
An Enriques surface that is very general in the sense of Barth--Peters~\cite{BP1983}
is ${(0,0)}$-generic,
and its automorphism group
was determined by Barth--Peters~\cite{BP1983} and  Nikulin~\cite[Theorem 10.1.2 (c)]{Nikulin1981}
independently.
%
\begin{theorem}[Barth--Peters~\cite{BP1983}, Nikulin~\cite{Nikulin1981}]\label{thm:BP}
Let $Y_0$ be a ${(0,0)}$-generic Enriques surface.
Then  $\aut({Y_0})\subset  \OGP(S_{Y_0})$
is equal to
the kernel of the  reduction homomorphism
$\OGP(S_{Y_0}) \to \OG(S_{Y_0}) \tensor \F_2$.
In particular,  the index of 
$\aut(Y_0)$ in  $\OGP(S_{Y_0})$ is equal to 
\begin{equation*}\label{eq:1BP}
2^{21}\cdot 3^5\cdot 5^2\cdot 7 \cdot 17\cdot 31=46998591897600.
\end{equation*}
\end{theorem}
%
Since  a $(0,0)$-generic  Enriques surface
$Y_0$ contains no smooth rational curves,
we have $\PPP_{Y_0}=\Nef_{Y_0}$.
Combining this with~\eqref{eq:OGPWL10}, 
we obtain  bijections
\[
\OGP(S_{Y_0})=W(S_{Y_0})\cong   \StdFD(\Nef_{Y_0}).
\]
We  define the unit $1_{\BP}$ ($\BP$ stands for Barth--Peters)
of volume  to be
\[
1_{\BP}:=\vol(\Nef_{Y_0}/\aut(Y_0))=[\OGP(S_{Y_0}): \aut(Y_0)]
=2^{21}\cdot 3^5\cdot 5^2\cdot 7 \cdot 17\cdot 31.
\]
Our first main result is as follows.
For an $\ADE$-type $\tau$, let $W(R_{\tau})$ denote the Weyl group
of the $\ADE$-lattice $R_{\tau}$ with $\tau(R_{\tau})=\tau$,
that is, the finite Coxeter group
defined by the Dynkin diagram of type $\tau$.
%
An automorphism of $Y$ is called \emph{numerically trivial} if it acts trivially on $\SY$.
\begin{theorem}\label{thm:main1vol}
Let $Y$ be a $(\tau, \taubar)$-generic Enriques surface.
Then we have
\[
\vol (\NefY/ \aut(Y))=\frac{c_{(\tau, \taubar)}}{|W(R_{\tau})|}  \hbox{\,$\cdot\, 1_{\BP}$}, 
\]
where $c_{(\tau, \taubar)}\in \{1, 2\}$ is the number of numerically trivial automorphisms of $Y$ and is given in $6$th column of Table~\ref{table:184}.
\end{theorem}
In two  
non-geometric cases Nos.~142~and~170 (Remark~\ref{rem:nongeometric}),  there exists a contribution to $c_{(\tau,\taubar)}$ not coming from a numerically trivial automorphism. 
(See Theorem~\ref{thm:cde} and Remark~\ref{rem:142and170}).
Theorem~\ref{thm:main1vol} is in fact obtained from a more general result
Theorem~\ref{thm:volumeformula} on  $\vol (\NefY/ \aut(Y))$. 
To obtain Theorem~\ref{thm:volumeformula}, we prove
a result~(Proposition~\ref{prop:03})
of the theory of discriminant forms 
in the spirit of Nikulin~\cite{Nikulin79}.
The proof of these theorems is conceptual.
Nevertheless the ability to compute examples played a crucial role in finding the correct statements.
\par
Next,  we calculate  explicitly 
a finite generating set  of $\aut(Y)$ 
and a complete set of representatives 
of the orbits of the action of $\aut(Y)$ on $\NefY$.
The algorithms  we use for this purpose are
variations of
a simple algorithm given in Section~\ref{subsec:graph},
which is an abstraction of the generalized Borcherds' method 
described in~\cite{Shimada2015}.
By means of these computational data,
we analyze the action of $\aut(Y)$ on 
$\Rats(Y)$ and $\Ells(Y)$.
(Recall that $\Rats(Y)$ and $\Ells(Y)$ 
are embedded into $\SY$.)
\par
Our second main result is as follows.
\begin{theorem}\label{thm:main2rats}
Let $Y$ be a $(\tau, \taubar)$-generic Enriques surface.
\par
{\rm (1)}
There exist smooth rational curves $C_1, \dots, C_m$ on $Y$ 
whose dual graph $\Gamma$ 
is a Dynkin diagram of type $\tau$.
Under the action of $\aut(Y)$, 
any smooth rational curve $C$ on $Y$
is in the same orbit 
as one of  $C_1, \dots, C_m$.
\par
{\rm (2)} 
The size of $\Rats(Y) /\aut(Y)$ is given in the $7$th column {\tt rat} of
Table~\ref{table:184}.
Except for the cases marked by $\times$ in this column,
two curves $C_i$ and $C_j$ are in the same orbit 
if and only if the vertices of the dual graph $\Gamma$ corresponding to
$C_i$ and $C_j$
belong to the same connected component of $\Gamma$,
and hence   
$|\Rats(Y)/\aut(Y)|$ is equal to the number of  connected components
of the Dynkin diagram of type $\tau$.
\end{theorem}
In~\cite{BP1983}, Barth and Peters also proved the following.
\begin{theorem}[Barth--Peters~\cite{BP1983}]\label{thm:BP2}
Let $Y_0$ be a $(0,0)$-generic  Enriques surface.
Then $Y_0$ has exactly $17\cdot 31=527$ elliptic fibrations modulo $\aut(Y_0)$.
\end{theorem}
We calculate $\Ells(Y)/\aut(Y)$ for $(\tau, \taubar)$-generic Enriques surfaces.
Since the tables span $7$ pages, we
relegate a part of it to the ancillary files.
\begin{theorem}\label{thm:main3ells}
Let $Y$ be a $(\tau,\taubar)$-generic Enriques surface.
Then  the orbits of the action of $\aut(Y)$ on the set $\Ells(Y)$
of elliptic fibrations of $Y$
are indicated in {\rm Section~\ref{subsec:tableells}} 
for $\rank \tau \le 7$ and in the ancillary files \cite{AutEnrVolCompdata} for $\rank \tau \geq 8$.
\end{theorem}
\subsection{The plan of the paper}\label{subsec:plan}
This paper is organized as follows.
In Section~\ref{sec:latandcham},
we prepare basic notions 
about finite quadratic forms, 
discriminant forms, lattices and chambers.
Proposition~\ref{prop:03}  in Section~\ref{subsec:FQF} plays 
a crucial role in the proof of the volume formula in the next section.
The notion of $L/M$-chambers given in Section~\ref{subsec:LMchams} is the main tool
of our computation.
In Section~\ref{sec:nefY}, we investigate 
the nef-and-big cone $\NefY$ of an Enriques surface $Y$
from the point of view of $L/M$-chambers,
and prove Proposition~\ref{prop:nefisom}.
Then,
by means of Proposition~\ref{prop:03},
we prove a formula (Theorem~\ref{thm:volumeformula})
for the volume of $\NefY/\aut(Y)$, and in Section~\ref{subsec:GXmWR},
 we deduce Theorem~\ref{thm:main1vol} from Theorem~\ref{thm:volumeformula}.
 \par
 In Section~\ref{sec:Borcherds},
 we present a computational procedure on a graph (Procedure~\ref{procedure:genB}),
 which is an abstraction of 
 the generalized Borcherds' method formulated in~\cite{Shimada2015}.
 Then we recall the classification of primitive embeddings $L_{10}(2)\inj L_{26}$
obtained in~\cite{BS2019},
and construct primitive embeddings  $\SY(2)\inj \SX\inj L_{26}$
for $(\tau, \taubar)$-generic Enriques surfaces $Y$.
In Section~\ref{sec:geomalgo},
we prepare some geometric algorithms used in the application of 
the generalized Borcherds' method to $(\tau, \taubar)$-generic Enriques surfaces.
In Section~\ref{sec:proofs}, 
we calculate $\aut(Y)$ and $\NefY/\aut(Y)$, and  prove 
Theorems~\ref{thm:main2rats}~and~\ref{thm:main3ells}.
The table of elliptic fibrations is given in Section~\ref{subsec:tableells}.

In Section~\ref{sec:examples},  we exhibit some examples.
In particular,  we treat 
an $(E_6, E_6)$-generic Enriques surface
(No.~47 of Table~\ref{table:184}) in detail,
because we  investigated this surface 
in~\cite{Shimada2019}.
Section~\ref{subsec:example47} contains 
a correction of a wrong assertion
made in~\cite{Shimada2019}.
\par 
In the second author's webpage and in~the repository ``zenodo"~\cite{AutEnrVolCompdata},
we put a detailed computation data made by {\tt GAP}~\cite{GAP}.
\par 
\medskip
Thanks are due to Professor Igor Dolgachev 
for his comments on the manuscript of this paper.
The authors also thank the referees for many valuable comments.
%
%
%
%
\section{Finite quadratic forms, lattices and chambers}\label{sec:latandcham}
We fix notions and terminologies about finite quadratic forms, 
discriminant forms, lattices and chambers.
\subsection{Finite quadratic forms}\label{subsec:FQF}
A \emph{finite quadratic form} is  a finite abelian group $A$ with a quadratic form
\[
q_A\colon A\to \Q/2\Z.
\]
We say that a finite quadratic form is \emph{non-degenerate}
if the  bilinear form 
\[
b_A\colon A\times A\to \Q/\Z
\]
associated with $q_A$ 
is non-degenerate.
The automorphism group of a finite quadratic form $A$ is denoted by $\OG(A)$,
and we let it act on $A$ from the right.
For a subgroup $D\subset A$,
let $D\sperp$ denote  the orthogonal complement of $D$ with respect to $b_A$,
and let $\OG(A, D)$ denote the subgroup $\set{g\in \OG(A)}{D^g=D}$ of $\OG(A)$.
\par
The following proposition will play a crucial role in the proof of
the volume formula~(Theorem~\ref{thm:volumeformula}).
\begin{proposition}\label{prop:03}
Let $(A, q_A)$ and $(B, q_B)$ be non-degenerate finite quadratic forms, and 
let $D_A\subset A$ and $D_B\subset B$ be subgroups.
Suppose that we have an isomorphism $\phi\colon D_A\isom D_B$ 
that induces an 
isometry $(D_A, -q_A|D_A)\cong (D_B, q_B|D_B)$
of finite quadratic forms.
Let $\Gamma\subset A\oplus B$ be the graph of $\phi$,
which is an isotropic subgroup with respect to $q_A\oplus q_B$.
We put $C:=\Gamma\sperp/\Gamma$.
Then $q_A\oplus q_B$ induces a quadratic form $q_C$ on $C$, and 
we have a natural homomorphism
\begin{equation*}
\set{(g, h) \in \OG(A)\times \OG(B)}{\Gamma^{(g, h)}=\Gamma}\to \OG(C).
\end{equation*}
We denote by $K$ the kernel of this homomorphism.
Then the homomorphism 
\[
i_A\colon K\inj \OG(A)\times \OG(B)  \to  \OG(A), \quad (g, h)\mapsto g
\]
is injective, 
and the image of $i_A$ is equal to the kernel 
of the natural homomorphism 
\[
\OG(A, D_A)\to \OG(D_A\sperp).
\]
\end{proposition}
\begin{proof}
First we prove that the natural projection $\Gamma\sperp\to B$ is surjective.
Since $q_A$ and $q_B$ are non-degenerate,
we have  natural isomorphisms  $A\cong \Hom(A, \Q/\Z)$ and $B\cong \Hom(B, \Q/\Z)$
induced by $b_A$ and $b_B$.
Hence we have natural isomorphisms $\Hom(D_A, \Q/\Z)\cong A/D_A\sperp$
and $\Hom(D_B, \Q/\Z)\cong B/D_B\sperp$.
We have an isomorphism 
\[
-\phi^*\colon  \Hom(D_B, \Q/\Z)\cong \Hom(D_A, \Q/\Z)
\]
induced by $-\phi\colon D_A\isom D_B$.
Combining them, 
we obtain a homomorphism 
\begin{equation}\label{eq:psi}
\psi\colon B\surj B/D_B\sperp\cong \Hom(D_B, \Q/\Z)\cong \Hom(D_A, \Q/\Z) \cong A/D_A\sperp.
\end{equation}
For $\alpha\in A$, 
we put
\[
\bar{\alpha}:=\alpha \bmod D_A\sperp \;\;\in \;\; A/D_A\sperp.
\]
Then, for $\alpha\in A$ and $\beta\in B$,  we have 
\begin{equation}\label{eq:aabbequiv}
\bar{\alpha}=\psi(\beta) 
\;\Longleftrightarrow \;  
b_A(\alpha, x)=-b_B(\beta, \phi(x)) \;\textrm{for all $x\in D_A$}
\; \Longleftrightarrow \;
(\alpha, \beta)\in \Gamma\sperp.
\end{equation}
In particular, 
for any $\beta\in B$, 
we have  $\alpha\in A$ such that  $(\alpha, \beta)\in \Gamma\sperp$.
\par
Next we prove that $i_A\colon K\to \OG(A)$ is injective.
Let $(1, h)\in K $ be an element of $\Ker i_A$.
For  $\beta\in B$,
we choose $\alpha\in A$ such that $(\alpha, \beta)\in \Gamma\sperp$.
Since $(1, h)$ acts on $C=\Gamma\sperp/\Gamma$ trivially, 
we have $(\alpha, \beta)-(\alpha, \beta^h)=(0, \beta-\beta^h)\in \Gamma$.
Since $\Gamma\cap B=0$, we have $\beta^h=\beta$.
Since $\beta\in B$ is arbitrary, we have $h=1$.
\par
Now we determine the image of $i_A$. ``$\subset$'':
Suppose that $(g, h)\in K$.
Since $(g, h)$ preserves  $\Gamma$,
we see that $g=i_A(g, h)$ preserves the image $D_A$ 
of the projection $\Gamma\to A$.
For any $\alpha\in D_A\sperp$, we have 
$(\alpha, 0)\in \Gamma\sperp$.
Since  $(g, h)$ acts on $C=\Gamma\sperp/\Gamma$ trivially,
we have $\alpha^g-\alpha\in \Gamma\cap A=0$.
Therefore 
 $\Image i_A$ is contained in $\Ker (\OG(A, D_A)\to \OG(D_A\sperp))$.
\par 
``$\supset$'':
To show the opposite inclusion,
we fix $g\in \Ker (\OG(A, D_A)\to \OG(D_A\sperp))$ and 
construct $h\in \OG(B)$  such that $(g, h)\in K$.
Since $g$ acts on $D_A\sperp$ trivially,
the linear map
\[
l_g\colon A/D_A\sperp \to A, \quad \bar{\alpha}\mapsto \alpha^g-\alpha
\]
is well-defined. The image of $l_g$ is contained in $D_A=(D_A\sperp)\sperp$: 
indeed,
for any $\alpha\in A$ and $y\in D_A\sperp$, we have
\[
b_A(l_g(\bar{\alpha}), y)=b_A(\alpha^g, y)-b_A(\alpha, y)=b_A(\alpha^g, y^g)-b_A(\alpha, y)=0.
\]
We define $h: B\to B$ by
\[
\beta^h:=\beta+\phi l_g \psi (\beta),
\]
where $\psi$ is given in~\eqref{eq:psi}. 
We show that  $h\in \OG(B)$.
We put $\bar{\alpha}=\psi(\beta)$.
Then we have 
\[
q_B(\beta^h)-q_B(\beta) = 2b_B(\beta, \phi l_g (\bar{\alpha}))+q_B(\phi l_g(\bar{\alpha}))= 
 -2b_A(\alpha, \alpha^g-\alpha)-q_A( \alpha^g-\alpha)=0,
\]
because $g\in \OG(A)$.
It only remains to show that $(g, h)\in \OG(A)\times \OG(B)$ preserves $\Gamma$ and acts on
$C=\Gamma\sperp/\Gamma$ trivially.
Using (\ref{eq:aabbequiv}) and $\Gamma \subset \Gamma^\perp$,
 we see that for any $\alpha\in D_A$, we have $\bar{\alpha} =\psi\phi (\alpha)$, and therefore
\[
\phi (\alpha)^h=\phi (\alpha)+\phi l_g (\bar{\alpha})=\phi (\alpha)+\phi(\alpha^g)-\phi (\alpha)=\phi(\alpha^g).
\]
Since $g$ preserves $D_A$, we have $(\alpha, \phi(\alpha))^{(g, h)}=(\alpha^g, \phi(\alpha^g))\in \Gamma$
for any $\alpha\in D_A$.
Therefore $(g, h)$ preserves $\Gamma$.
Suppose that $(\alpha, \beta)\in \Gamma\sperp$.
Then we have $\bar{\alpha} =\psi(\beta)$ by~\eqref{eq:aabbequiv}, and 
\[
(\alpha^g, \beta^h)-(\alpha, \beta)=(l_g(\bar{\alpha}), \phi l_g(\bar{\alpha}) )\;\; \in \;\; \Gamma.
\]
Therefore $(g, h)$ acts on
$\Gamma\sperp/\Gamma$ trivially.
\end{proof}
\begin{remark}
 Proposition \ref{prop:03} holds for 
 non-degenerate finite bilinear forms $(A, b_A)$ and $(B, b_B)$
 as well.
\end{remark}
\subsection{Discriminant forms and overlattices}\label{subsec:disc}
Let $L$ be an even lattice.
We put
\[
L\dual:=\set{x\in L\tensor\Q}{\intf{x, v}\in \Z\;\;\textrm{for all}\;\; v\in L},
\]
on which $\OG(L)$ acts naturally.
The finite abelian group $L\dual /L$ is called
the \emph{discriminant group} of $L$.
Then 
\[
q(\bar{x})=\intf{x, x}\bmod 2\Z
\quad \textrm{for $x\in L\dual$ and $\bar{x}=x \bmod L$}
\]
defines a finite quadratic form $q\colon L\dual/L \to \Q/2\Z$,
which is called 
the \emph{discriminant form} of $L$.
An even lattice $L\sprime$ is an \emph{overlattice} of $L$
if  we have $L\subset L\sprime\subset L\dual$
and the intersection form of $L\sprime$ is
the extension of that of $L$.
See Nikulin~\cite{Nikulin79}
for the details of the theory of discriminant forms
and its application to the enumeration of
even overlattices of a given even lattice.
\par
To illustrate Proposition \ref{prop:03}, we apply it to two known extreme cases.
\begin{example}
 Let $M,N \subset L$ be primitive sublattices of an even lattice $L$
 such that $M\perp N$ 
 and $\rank M+\rank N=\rank L$. Then we have
 \[M \oplus N \subset L \subset L\dual  \subset M\dual  \oplus N\dual , \]
 and $L$
 is an overlattice of $M\oplus N$.
 Let $(A,q_A)= (M\dual /M,q_M)$ and $(B,q_B)= (N\dual /N,q_N)$
 be the respective discriminant forms.
 Then $\Gamma = L/(M \oplus N)$ is the graph of an anti-isometry
 $\phi\colon A \supset D_A \rightarrow D_B \subset B$ and $\Gamma^\perp/\Gamma \cong L\dual /L$.
 \par
 First suppose that $L$ is unimodular. Then,
 by a result of Nikulin~\cite{Nikulin79}, $D_A = A$ and $D_B=B$. 
 Since $L\dual /L \cong \Gamma^\perp/\Gamma$ is trivial,
 we have 
 \[
 K=\{(g,h) \in \OG(A) \times  \OG(B) : h \circ \phi = \phi \circ g \}.
 \]
 We see that $i_A\colon K \rightarrow  \OG(A)$ is an isomorphism
 as predicted by Proposition \ref{prop:03}.
 Indeed, since $D_A^\perp = A^\perp =0$, the homomorphism $ \OG(A,D_A) \rightarrow  \OG(D_A^\perp)$ is trivial.
 \par
 For the other extreme suppose that $M\oplus N = L$. Then $D_A = 0$, $D_B=0$, $K = 1$ and $D_A^\perp = A$.
 \end{example}
\subsection{Faces of a chamber}\label{subsec:faces}
Let $L$ be a hyperbolic lattice
with a positive half-cone $\PPP$,
and $D$ a chamber in $\PPP$.
A \emph{face} of $D$ is a closed subset of $D$
that is an intersection of some walls of $D$.
Let $f$ be a face of $D$.
The \emph{dimension $\dim f$}  of $f$ is the dimension
of the minimal linear subspace of $L\tensor\R$ containing $f$,
and the \emph{codimension} of $f$ is $\rank L-\dim f$.
The walls of $D$ are exactly the faces of $D$ with codimension $1$.
\par
Let $\closure{\PPP}$ and $\closure{D}$ be the closures
of $\PPP$ and $D$ in $L\tensor\R$, respectively.
A half-line contained in
$(\closure{\PPP}\setminus \PPP) \cap \closure{D}$
is called an \emph{isotropic ray} of $D$.
\par
Suppose that $D$ has only finitely many walls,
that they are defined by vectors in $L\tensor\Q$,
and that the list of defining vectors of these walls in $L\tensor\Q$ is available.
Then we can make the list of faces of $D$
by means of linear programming.
For each isotropic ray $\R_{\ge 0}v $,
we have a unique primitive vector $v\in L$ that generates $\R_{\ge 0}v $,
which we call  a \emph{primitive isotropic ray} of $D$.
We can also make the list of primitive isotropic rays of $D$.
\subsection{$L/M$-chambers}\label{subsec:LMchams}
Let $(L, \intf{\;,\;}_L)$ and $(M, \intf{\;,\;}_M)$
be even hyperbolic lattices
with  fixed positive half-cones $\PPP_L$ and $\PPP_M$, respectively.
Suppose that we have an embedding $M\inj L$
that maps $\PPP_M$ into $\PPP_L$.
We regard $\PPP_M$ as a subspace of $\PPP_L$ by this embedding.
The notion of $L$-chambers was introduced in Section~\ref{subsec:chambers}.  
The following class of chambers plays an important role in this paper.
\begin{definition}\label{def:wall}
A chamber $D_M$ in $\PPP_M$ is called
an \emph{$L/M$-chamber} 
if there exists 
an $L$-chamber $D_L \subset \PPP_L $
such that $D_M= \PPP_M\cap D_L $.
In this case,
we say that $D_M$ \emph{is induced by} $D_L$.
\end{definition}
In particular, an $L$-chamber is an $L/L$-chamber. 
\begin{definition}\label{def:Weyl-chambers}
Let $N$ be a negative definite even lattice.
For a root $r$ of $N$, let $[r]\sperp$ denote the hyperplane of $N\tensor \R$
defined by $\intf{x, r}=0$.
The connected components of $(N\tensor\R)\setminus \bigcup\, [r]\sperp$,
where $r$ runs through the set of roots of $N$,
are called the \emph{Weyl-chambers} of  $N$.
The Weyl group $W(N)$ acts simply transitively on the set of 
Weyl-chambers.
\end{definition}
\begin{remark}
Let $D_M$ be an $L/M$-chamber.
Then the number of $L$-chambers that induce $D_M$ 
is equal to the number of Weyl-chambers of 
the orthogonal complement $(M\inj L)\sperp$ of $M$ in $L$.
In particular, if $(M\inj L)\sperp$ contains no roots,
then each $L/M$-chamber is 
induced by a unique $L$-chamber.
\end{remark}
\begin{definition}\label{def:adjacent}
Two distinct $L/M$-chambers $D_1$ and $D_2$ are \emph{adjacent}
if there exists a hyperplane $(v)\sperp$ of $\PPP_M$ such that
$D_1\cap (v)\sperp$ is a wall of $D_1$,
that $D_2\cap (v)\sperp$ is a wall of $D_2$,
and that $D_1\cap (v)\sperp=D_2\cap (v)\sperp$ holds.
In this case, we say that
$D_2$ is adjacent to $D_1$ \emph{across the wall $D_1\cap (v)\sperp$}.
\end{definition}
Let $\pr\colon L\to M\tensor \Q$ be the orthogonal projection.
Then an $L/M$-chamber is the closure in $\PPP_M$
of a connected component of
\[
\PPP_M\;\setminus\; \bigcup_r \,(\pr(r))\sperp,
\]
where $r$ runs through the set
of  roots $r$ of $L$ such that $\intf{\pr(r), \pr(r)}_M<0$ holds, 
and $(\pr(r))\sperp=\PPP_M\cap (r)\sperp$ is the hyperplane of $\PPP_M$ defined by
$\pr(r)$.
Hence, for each wall $D_M\cap(v)\sperp$ of an  $L/M$-chamber $D_M$,
there exists a unique $L/M$-chamber adjacent to $D_M$
across the wall $D_M\cap(v)\sperp$.
\par
Since a root of $M$ is mapped to a root of $L$ by the embedding $M\inj L$,
an $M$-chamber
is tessellated by $L/M$-chambers.
More generally, we have the following proposition, which is easy to prove:
\begin{proposition}\label{prop:LM1M2}
Suppose that  $M_1 \inj M_2 \inj  L$
is a sequence of embeddings of even hyperbolic lattices
that induces a sequence of embeddings
$\PPP_{M_1}\inj \PPP_{M_2} \inj \PPP_{L}$
of  fixed positive half-cones.
Then
each $M_2/M_1$-chamber
is tessellated by $L/M_1$-chambers.
\qed
\end{proposition}
If $\tilg\in \OGP(L)$ satisfies $M^{\tilg}=M$, then
$\tilg|M\in \OGP(M)$ preserves the tessellation of $\PPP_M$ by $L/M$-chambers.
\par
In general,
two distinct  $L/M$-chambers are not isomorphic to each other.
See~\cite{KKS2014} and~\cite{Shimada2015} for examples of $K3$ surfaces $X$ 
with a primitive embedding $\SX\inj L_{26}$ 
such that  $\PPP_X$ 
is tessellated by $L_{26}/\SX$-chambers of various shapes.
\begin{definition}\label{def:reflexively}
We say that the tessellation of $\PPP_M$ by $L/M$-chambers is
\emph{reflexively simple}
if, for  each wall  $D_M\cap(v)\sperp$ of an $L/M$-chamber $D_M$,
there exists an isometry $\tilg$ of $L$ preserving $M$ 
such that the restriction $\tilg |M$ of $\tilg$ to $M$ is an involution that  fixes 
every point of 
the hyperplane $(v)\sperp$.
Note that, if this is the case,
the isometry $\tilg |M$ of $M$ 
 maps $D_M$ to 
the $L/M$-chamber adjacent to $D_M$ across the wall $D_M\cap(v)\sperp$.
\end{definition}
The tessellation of $\PPP_L$ by $L/L$-chambers is
obviously reflexively simple.
%
%
%
\section{The cone $\NefY$}\label{sec:nefY} 
Let $Y$ be an Enriques surface
with the universal covering $\pi\colon X\to Y$.
Let $\enrinvol\in \Aut(X)$ be the deck-transformation of $\pi\colon X\to Y$, 
and we put
\[
\SXp:=\set{v\in \SX}{v^\enrinvol=v}, \quad
\SXm:=\set{v\in \SX}{v^\enrinvol=-v}.
\]
Then 
$\SXp$  is equal to the image of $\pi^*\colon \SY(2)\inj \SX$,
and  $\SXm$  is  the orthogonal complement 
of $\SXp$.
We regard $\PPP_Y$ as a subspace of $\PPP_X$ by $\pi^*\tensor\R$.
\subsection{$\SX/\SY(2)$-chambers}\label{subsec:SYSY2chambers} 
It is well-known that  $\NefX$ is  
an $\SX$-chamber.
Therefore the chamber $\NefY=\PPP_Y\cap \NefX$ is an $\SX/\SY(2)$-chamber.
Since $\pi$ is \'etale, the lattice $\SXm$  contains no roots, and hence 
each $\SX/\SY(2)$-chamber $D_Y$ is induced by a \emph{unique} $\SX$-chamber $D_X$,
that is,  $D_Y$ contains an interior point of $D_X$.
\begin{proposition}\label{prop:SXSY2simple}
The tessellation of $\PPP_Y$  by $\SX/\SY(2)$-chambers is reflexively simple.
More precisely,
every wall of an $\SX/\SY(2)$-chamber $D_Y$ is defined by a root $r$ of $\SY$,
and the reflection $s_r\in \OGP(\SY)$ with respect to the root $r$ 
is the restriction $s_{\tilr_+} s_{\tilr_-}|\SY(2)$ of the product 
of two reflections with respect to roots $\tilr_+, \tilr_-$ of $\SX$.
\end{proposition}
\begin{proof}
Let $\intfX{-,-}$ and $\intfY{-,-}$ be the intersection forms of $\SX$ and $\SY$,
respectively.
We denote by $(u)\sperp_X$ the hyperplane of $\PPP_X$
defined by $u\in \SX\tensor\R$,
and by $(v)\sperp_Y$ the hyperplane of $\PPP_Y$
defined by $v\in \SY\tensor\R$.
Let  $D_Y$ be an $\SX/\SY(2)$-chamber,
and let $D_Y \cap (v)\sperp_Y$ be a wall of $D_Y$.
\par
By the definition of $\SX/\SY(2)$-chambers,
there exists a root $\tilr$ of $\SX$
such that
$(v)_Y\sperp=\PPP_Y\cap (\tilr)_X\sperp$.
We first prove that  $\intfX{\tilr, \tilr^{\enrinvol}}=0$.
Let $\tilr$ be written as $v_L+v_R$,
where $v_L\in \SY(2)\dual$ and $v_R\in \SXm\dual$.
We have $\intfX{v_L, v_L}+\intfX{v_R, v_R}=-2$.
Since $\tilr^{\enrinvol}=v_L-v_R$,
it is enough to show that $\intfX{v_L, v_L}=-1$.
Since 
\[
\PPP_Y\cap (\tilr)_X\sperp=(v_L)_Y\sperp
\]
is non-empty,
we have $\intfY{v_L, v_L}<0$.
Note that
$2 v_L\in \SY$ because $2\SY(2)\dual=\SY(2)$.
Since  $\SY$ is even,
 $\intfX{v_L, v_L}=2\intfY{v_L, v_L}$ must be an integer.
Since $\SXm$ is negative definite,
we have $\intfX{v_R, v_R}\le 0$ and hence $\intfX{v_L, v_L}$ is $-2$ or $-1$.
If $\intfX{v_L, v_L}=-2$, then $v_R=0$ and $\tilr=v_L\in \SY(2)$,
which is absurd.
\par
Let $s$ and $s\sprime$ be the reflections with respect to the roots
$\tilr$ and $\tilr^{\enrinvol}$ of $\SX$, respectively.
By $\intfX{\tilr, \tilr^{\enrinvol}}=0$,
we have $s s\sprime=s\sprime s$.
Since $s\sprime=\enrinvol s \enrinvol$,
we see that $ss\sprime$ commutes with $\enrinvol$ and hence 
$ss\sprime$ preserves $\PPP_Y$.
The vector $r:=\tilr+\tilr^{\enrinvol}$ is contained in $\SY$.
Moreover we have $\intfY{r, r}=-2$ and
\[
(v)_Y\sperp=\PPP_Y\cap (\tilr)_X=(v_L)_Y\sperp=(r)_Y\sperp.
\]
Therefore the wall 
$D_Y \cap (v)\sperp_Y$ of $D_Y$ is defined by a root $r$ or $-r$ of $\SY$.
It is easy to confirm  that the restriction of $ss\sprime$ to $\SY$ is equal
to the reflection with respect to the root $r$ of $\SY$ and therefore
maps $D_Y$
to the $\SX/\SY(2)$-chamber $D\sprime_Y$ adjacent to $D_Y$ across
the wall $D_Y\cap (v)_Y\sperp=D_Y\cap (r)_Y\sperp$.
\end{proof}
\subsection{Proof of Proposition~\ref{prop:nefisom}}\label{subsec:proofprop:nefisom}
We prove Proposition~\ref{prop:nefisom}.
By Proposition~\ref{prop:184}, we have isomorphisms $\psi_X$ and $\psi_Y$ 
that make the diagram~\eqref{eq:nefisomdiagram} commutative.
By Proposition~\ref{prop:SXSY2simple},
we have $\tilg\in \OGP(\SX)$  commuting with $\enrinvol$ 
such that $\tilg|\SY(2)$ maps $\NefY$ to  the inverse image of $\Nef_{Y\sprime}$ by $\psi_Y$.
Then the isometries $\tilg\circ \psi_X\colon S_{X\sprime}\isom \SX$ and 
$\tilg|\SY(2)\circ \psi_Y \colon S_{Y\sprime}\isom \SY$ satisfy the required properties.
\qed
%
%
\subsection{The volume of $\NefY/\aut(Y)$}\label{subsec:vol}
In this subsection,
we give a  formula (Theorem~\ref{thm:volumeformula}) 
for $\vol(\NefY/\aut(Y))$
under the assumption that 
\begin{equation}\label{eq:assumpTXomega}
\textrm{the group $\OG(T_X, \omega)$  in Definition~\ref{def:tautaubargeneric} is $\{\pm 1\}$.}
\end{equation}
We put 
\begin{equation}\label{eq:GX}
G_X:=\set{\tilg \in \OGP(\SX)}{\textrm{$\tilg$ commutes with $\enrinvol$ and acts on $\SX\dual/\SX$ trivially}}.
\end{equation}
Then $\tilg\mapsto (\tilg | \SXp, \tilg | \SXm)$ embeds $G_X$ into  $\OGP(\SXp)\times \OG(\SXm)$.
Let $G_{X+}$ and $G_{X-}$ denote the images  of
the projections $G_X\to \OGP(\SXp)$ and $G_X\to \OG(\SXm)$,
respectively.
When we regard $G_{X+}$ as a subgroup of $\OGP(\SY)$ via
the identification $\SXp=\SY(2)$ induced by $\pi^*$, we write  $G_Y$ instead of $G_{X+}$.
%
%
Recall that the set $\Rats(Y)$ of smooth rational curves on $Y$
is embedded into $\SY$ by $C\mapsto [C]$.
The correspondence
\[
C \mapsto \NefY\cap ([C])\sperp
\]
gives a bijection from $\Rats(Y)$ to the set of walls of the $\SX/\SY(2)$-chamber $\NefY$.
We denote by $W(\Rats(Y))$ the subgroup of $\OGP(\SY)$ 
generated by the reflections $s_{[C]}$ with respect to the roots $[C]\in \Rats(Y)$.
Recall also that  $\aut(Y)$ is the image of the natural representation $\Aut(Y)\to \OGP(\SY)$.
\begin{proposition}\label{prop:GY}
Suppose that $Y$ satisfies~\eqref{eq:assumpTXomega}.
\par
{\rm (1)} 
The action of $G_Y$ on $\PPP_Y$ preserves the set of  $\SX/\SY(2)$-chambers,
and $\aut(Y)$ 
is equal to the stabilizer subgroup of $\NefY$  in $G_Y$.
\par
{\rm (2)} 
The group $W(\Rats(Y))$ is contained in $G_Y$ as a normal subgroup,
and we have $G_Y=W(\Rats(Y))\semidirectproduct \aut(Y)$.
\end{proposition} 
\begin{proof}
Since every $g\in G_Y$ lifts to an element $\tilg$ of $G_X\subset \OGP(\SX)$,
the action of $G_Y$ on $\PPP_Y$ preserves the tessellation of $\PPP_Y$ by $\SX/\SY(2)$-chambers.
\par
Let $\aut(X)$ be the image of the natural representation $\Aut(X)\to \OGP(\SX)$.
By the Torelli theorem for complex $K3$ surfaces~(\cite[Chapter~VIII]{CCSBook}),
we have a natural embedding 
\begin{equation}\label{eq:AutX}
\Aut(X)\inj \OGP(\SX)\times \OG(T_X, \period),
\end{equation}
and  an element 
$(\tilg,  f)$ of $\OGP(\SX)\times \OG(T_X, \period)$ belongs to $\Aut(X)$
if and only if 
$(\tilg,  f)$ preserves the overlattice $H^2(X, \Z)$ of $\SX\oplus T_X$ and 
$\tilg$ preserves $\NefX$.
The even unimodular overlattice $H^2(X, \Z)$ of $\SX\oplus T_X$
induces an isomorphism
\[
i_{H(X)} \colon \SX\dual/\SX\;\cong\; T_X\dual/T_X
\]
of discriminant groups, and 
$(\tilg,  f)$ preserves  $H^2(X, \Z)$ 
if and only if the action of $\tilg$ on $\SX\dual/\SX$ is compatible with the action 
of $f$ on $T_X\dual/T_X$ via $i_{H(X)} $~(see~\cite{Nikulin79}).
Therefore, by  assumption~\eqref{eq:assumpTXomega},
an isometry $\tilg\in \OGP(\SX)$ belongs to $\aut(X)$ 
if and only if $\tilg$ preserves $\NefX$ and  acts on $\SX\dual/\SX$  as $\pm 1$.

\par
Let $\Aut(X, \enrinvol)$ denote the centralizer of $\enrinvol$ in $\Aut(X)$.
We have a natural identification $\Aut(Y)\cong \Aut(X, \enrinvol)/\gen{\enrinvol}$.
Suppose that $g\in \aut(Y)$.
We will show that $g$ belongs to the stabilizer subgroup of $\NefY$ in $G_Y$.
It is obvious that $g$ preserves $\NefY$.
Let $\tilde{\gamma}$ be an element  of $\Aut(X, \enrinvol)$
that induces $g$ on $\SY$.
We write $\tilde{\gamma}$ as $(\tilde{g}, f)$ by~\eqref{eq:AutX}.
Note that $\enrinvol$  acts on $T_X$ as $-1$.
Hence, replacing $\tilde{\gamma}  $ with $\tilde{\gamma}\enrinvol  $ if  $f=-1$,
we can assume $f=1$.
Then the action $\tilde{g}\in \OGP(\SX)$ of $\tilde{\gamma}$ on $\SX$
induces the trivial action  on $\SX\dual/\SX$, which means  $\tilde{g}\in G_X$.
Hence $g=\tilde{g}|\SY$ belongs to $G_Y$.
\par
Conversely, 
suppose that $g$ is an element of  the stabilizer subgroup of $\NefY$ in $G_Y$.
We will show that $g\in \aut(Y)$.
Let $\tilde{g}$ be an element of $G_X$  such that $g=\tilde{g}|\SY$.
Since $\NefY$ contains an interior point of $\NefX$,
$\tilde{g}$  preserves $\NefX$,
and hence $\tilde{g}$ belongs to $\aut(X)$.
Let $\tilde{\gamma}  =(\tilde{g}, f)$ be an element of $\Aut(X)$ that induces $\tilde{g}$.
Since   $\tilde{g}\in G_X$ commutes with the action of $\enrinvol$ on $\SX$,
the first factor of the commutator $[\tilde{\gamma}, \enrinvol]\in \Aut(X)$ is $1$.
Since $\OG(T_X, \period)=\{\pm 1\}$ is abelian, the second factor  of $[\tilde{\gamma}, \enrinvol]$
is also $1$.
Hence $\tilde{\gamma} \in \Aut(X, \enrinvol)$,
and therefore $g$ is induced by an element of $\Aut(Y)$.
Thus  assertion~(1) is proved.
\par
By Proposition~\ref{prop:SXSY2simple}, 
for each $r\in \Rats(Y)$,
the reflection $s_r=s_{\tilr_+} s_{\tilr_-}|\SY(2)$ belongs to $G_Y$,
because the reflections $s_{\tilr_+}$ and  $s_{\tilr_-}$ 
act on $\SX\dual/\SX$ trivially and hence $s_{\tilr_+} s_{\tilr_-}\in G_X$.
Therefore we have $W(\Rats(Y))\subset G_Y$.
Moreover, 
by Proposition~\ref{prop:SXSY2simple} again, 
we see that $W(\Rats(Y))$ acts on the set of $\SX/\SY(2)$-chambers transitively.
\par
If   $C_1, C_2\in \Rats(Y)$ 
satisfy
$\intf{C_1, C_2}_Y>1$,
then the walls $\NefY\cap ([C_1])\sperp$ and   $\NefY\cap ([C_2])\sperp$ 
of $\NefY$ do not intersect.
Hence
each face of $\NefY$ with codimension $2$ 
is of the form
\[\NefY\cap ([C_1])\sperp \cap ([C_2])\sperp \quad \mbox{ with } \quad  \intf{C_1, C_2}_Y\in \{0, 1\},\]
and we have $(s_{[C_1]} s_{[C_2]})^m=1$, 
where $m=2$ if $\intf{C_1, C_2}_Y=0$ and $m=3$ if $\intf{C_1, C_2}_Y=1$.
Therefore, by the standard method of geometric group theory
(see, for example, Section 1.5 of~\cite{VinbergShvartsman1993}), 
we see that $\NefY$ is 
a standard fundamental domain of the action of $W(\Rats(Y))$ on $\PPP_Y$,
and $W(\Rats(Y))$ acts on the set of $\SX/\SY(2)$-chambers simply-transitively.
Recalling that $\aut(Y)$ is the stabilizer subgroup of
$\NefY$ in $G_Y$,  we have $W(\Rats(Y))\cap  \aut(Y)=\{1\}$.
Moreover $G_Y$ is generated by the union of $W(\Rats(Y))$ and $\aut(Y)$.
\par
It remains to show that $W(\Rats(Y))$ is a normal subgroup of 
$G_Y$.
Let $r$ be a root in $\Rats(Y)$ and $g$ an arbitrary element of $G_Y$.
It is enough to show that $g\inv s_r g$ 
belongs to   $W(\Rats(Y))$.
Note that  $g\inv s_r g=s_{r^g}$ and $r^g$ defines a wall of 
the $\SX/\SY(2)$-chamber $D_Y:={\NefY}^g$.
We have an element $w\in W(\Rats(Y))$ such that $D_Y={\NefY}^w$.
Then $r\sprime:=r^{gw\inv}$ defines a wall of $\NefY$, and 
$w s_{r^g} w\inv =s_{r\sprime}$ is an element of $W(\Rats(Y))$.
Hence  $g\inv s_r g=s_{r^g}=w\inv s_{r\sprime} w \in W(\Rats(Y))$.
\end{proof}
Let $(A_+, q_+)$ and $(A_-, q_-)$ be
the discriminant forms of $\SXp=\SY(2)$ and $\SXm$, respectively.
We put 
\[
\Gamma_X:=\SX/(\SXp\oplus \SXm)\;\; \subset\;\; A_+\oplus A_-,
\]
and let $D_+\subset A_+$ and $D_-\subset A_-$ be the image of the projections of $\Gamma_X$.
Then $\Gamma_X$ is the graph of an isometry
$(D_+, q_+|D_+)\cong (D_-, -q_-|D_-)$,
and the discriminant group of $\SX$ is canonically isomorphic to
$\Gamma_X\sperp/\Gamma_X$.
We denote 
by $\Gbar_{X+}$ and $\Gbar_{X-}$ the images of
$G_{X+}$ and $G_{X-}$ by the natural homomorphisms
$\OGP(\SXp)\to \OG(A_{+})$ and $\OG(\SXm)\to \OG(A_{-})$,
respectively,
and 
by  $\Gbar_{X}$ the image of $G_X$ 
 by the natural homomorphism to $\OG(A_{+}) \times \OG(A_{-})$.
Note that  $\Gbar_X$ is a subgroup
of the kernel $K$ of the natural homomorphism 
\[
\set{(g, h) \in \OG(A_+)\times \OG(A_-)}{\Gamma_X^{(g, h)}=\Gamma_X}\to 
\OG(\Gamma_X\sperp/\Gamma_X)=\OG(\SX\dual/\SX).
\]
Then we have a commutative diagram
\begin{equation}\label{eq:Gdiagram}
\renewcommand{\arraystretch}{1.2}
\begin{array}{ccccc}
G_{X+} & \longleftarrow & G_X & \longrightarrow  & G_{X-} \\
\mapdownsurj& & \mapdownsurj& & \mapdownsurj \\
\Gbar_{X+} &\mapleftisom & \Gbar_X & \maprightisom  & \Gbar_{X-} \\
\end{array}
\end{equation}
where the two arrows below are isomorphisms by 
the first part of Proposition~\ref{prop:03} applied to to $(A, B)=(A_+, A_-)$ and $(A, B)=(A_-, A_+)$.

\begin{lemma}\label{lem:numerically-trivial}
Suppose that $Y$ satisfies~\eqref{eq:assumpTXomega}.
Then the group 
\[
\Aut_{nt}(Y):=\Ker (\Aut(Y) \to \aut(Y))
\]
 of numerically trivial automorphisms of $Y$ is isomorphic to the kernel of the natural homomorphism $G_{X-} \to \Gbar_{X-}$.
\end{lemma}
\begin{proof}
There is an isomorphism of $\Aut_{nt}(Y)$ with $\ker \left(G_X \to G_{X+}\right)$ given by 
mapping a numerically trivial automorphism $g$ to its lift $\tilde{g} \in \Aut(X)$ acting trivially on the $2$-form $\omega$ 
and restricting to its action on $\SX$.
By the diagram~\eqref{eq:Gdiagram} and $G_X\subset G_{X+}\times G_{X-}$, 
the kernel  $\ker \left(G_X \to G_{X+}\right)$ injects into $\ker(G_{X-} \to \Gbar_{X-})$.
Conversely any element of  $\ker(G_{X-} \to \Gbar_{X-})$ can be extended to an element of $\ker \left(G_X \to G_{X+}\right)$
by complementing it with the trivial action of $\SXp$. 
\end{proof}
\begin{theorem}\label{thm:volumeformula}
Suppose that $Y$ satisfies~\eqref{eq:assumpTXomega}.
Let $\OG(\SXm, D_-)$ be the subgroup of $\OG(\SXm)$
consisting of isometries $g$
whose action
on $A_-$ preserves $D_-$. 
Then we have  
\begin{equation}\label{eq:GXm}
G_{X-}=\Ker (\OG(\SXm, D_-) \to \OG(D_-\sperp)).
\end{equation}
%
%
Moreover we have 
\begin{equation}\label{eq:nt}
\vol(\NefY/\aut(Y))=\frac{1_{\BP}}{|\Gbar_{X-}|}=\frac{|\Aut_{nt}(Y)|}{|G_{X-}|}1_{\BP}.
\end{equation}
\end{theorem}
\begin{proof}
Recall that we have
$|\Gbar_X|=|\Gbar_{X+}|=|\Gbar_{X-}|$.
Let $G_{\BP}$ be the kernel of the natural  homomorphism
$\OGP(\SXp)=\OGP(\SY(2))\to \OG(A_+)$.
Then $G_{\BP}$ is equal to $\aut(Y_0)$ by Theorem~\ref{thm:BP} and 
hence the index of $G_{\BP}$ in $\OGP(\SXp)= \OGP(\SY)$  is $1_{\BP}$.
If $g\in G_{\BP}$, then $(g, 1)\in \OGP(\SXp)\times \OG(\SXm)$ 
acts trivially on $A_+\oplus A_-$,
 and hence preserves $\Gamma_X$ and acts on
$\Gamma_X\sperp/\Gamma_X$ trivially.
Therefore 
the action of $(g, 1)$  on $\SXp\oplus \SXm$ 
preserves the overlattice $\SX$,  and $(g, 1)|\SX$ is an element of $G_X$.
Thus  $G_{\BP}$ is contained in  $G_{X+}=G_Y$.
Since the natural homomorphism
$\OGP(\SY(2))\to \OG(A_+)$ is surjective~(see~\cite{BP1983}),
the index of $G_{\BP}$ in $G_Y$ is equal to $|\Gbar_{X+}|=|\Gbar_{X-}|$.
\par
Applying the second part of Proposition~\ref{prop:03} to
$(A, B)=(A_-, A_+)$,
we see that 
\[
\Gbar_{X-} \subset \Image i_{A_-}=\Ker (\OG(A_-, D_-)\to \OG(D_-\sperp)).
\]
Hence the inclusion $\subset$ in~\eqref{eq:GXm} is proved.
Conversely, let $f$ be an element of the right-hand side of~\eqref{eq:GXm},
and denote by $\bar{f}\in \OG(A_-)$ 
the action of $f$ on $A_-$.
By Proposition~\ref{prop:03}, we have $\bar{f}\in \Image i_{A_-}$  and hence 
there exists a unique element $\bar{h}\in K$ such that $ i_{A_-}(\bar{h})=\bar{f}$.
We put  $\bar{g}:=i_{A_+}(\bar{h})$.
Since the natural homomorphism $\OGP(\SXp)\to \OG(A_+)$ is surjective,
we have $g\in \OGP(\SXp)$ that maps to $\bar{g}$.
Since $\bar{h}=(\bar{g}, \bar{f})\in K$, we have $(g, f)\in G_X$,
which implies $f\in G_{X-}$. Thus~\eqref{eq:GXm} is proved.
Moreover  we have 
\[
\vol(\NefY/\aut(Y))=\vol(\PPP_Y/G_Y)=\frac{1}{[G_Y: G_{\BP}]}\cdot \vol(\PPP_Y/G_{\BP})=\frac{1_{\BP}}{|\Gbar_{X-}|},
\]
where the first equality follows from Proposition~\ref{prop:GY}.
From Lemma~\ref{lem:numerically-trivial}
we get the second equality of~\eqref{eq:nt}.
\end{proof}
Since $\SXm$ is negative definite,
$\OG(\SXm)$ is a finite group and can be computed easily.
Thus this formula enables us to calculate $\vol(\NefY/\aut(Y))$.

\subsection{Proof of Theorem~\ref{thm:main1vol}}\label{subsec:GXmWR}
In what follows we calculate the finite group $\Gbar_{X-}$
of a $(\tau, \taubar)$-generic Enriques surface. It is
closely related to the Weyl group $W(R_{\tau})$.
\par
For a sublattice $L\sprime$ of a lattice $L$,  we denote by 
$\OG(L, L\sprime)$  the group 
of isometries of $L$ preserving $L\sprime$.
When $L$ is an overlattice of $L\sprime$,
then  $\OG(L, L\sprime)$ is 
the group 
of isometries of $L\sprime$ preserving the overlattice $L$,
or equivalently the intersection $\OG(L)\cap \OG(L\sprime)$ 
in $\OG(L\tensor \Q)=\OG(L\sprime\tensor \Q)$,
and hence sometimes is written as $\OG(L\sprime, L)$.
\par
\begin{lemma}\label{lem:dict}
Let $Y$ be $(\tau,\taubar)$-generic.
Recall the commutative diagram~\eqref{eq:generic-diagram}
\[
\begin{array}{ccc}
\Lten (2) &\maprightinjsp{\embLM_R}  & M_R
\mystruth{12pt} \\
\llap{\scriptsize $g$\;}\downisom &&  \llap{\scriptsize $\tilde{g}$\;}\downisom \\
S_Y(2) &\maprightinjsp{\pi^*} & S_X \rlap{.}
\end{array}
\]
Denote by
$\pi_-\colon \SX \rightarrow \SXm\dual $ the orthogonal projection.
Identify $M_R$ with $S_X$ via $\tilde{g}$. 
Then the following equalities hold:
\begin{multicols}{2}
\begin{enumerate}[{\rm (1)}] 
 \item $\trl = \SXm$,
 \item $\rl = \pi_-(2 \SX)$,
 \item $\tfrac{1}{2}\trl\dual  / \trl=A_- $, 
 \item $\tfrac{1}{2}\rl / \trl = D_{-}$,
 \item $\rl\dual  / \trl=D_{-}^\perp $,
 \item $\OG(\trl,\rl)=\OG(\SXm,D_-)$.
\end{enumerate}
\end{multicols}
Note that we neglect the quadratic forms in {\rm (1)--(5)}
and just consider them as equalities of abelian groups.
\end{lemma}
\begin{proof}
 The equality (1) is by the definition.
 \par
 (2) Note that $M_R$ is spanned by $\Image \embLM_R$ 
 and $\{(i_R(v) \pm v)/2 \mid v \in R\}$. 
 Hence $\pi_-(M_R)$ is spanned by $0$ and $\tfrac{1}{2}R$.
\par
 (3) As lattices we have $\Rtil(2) = \SXm$,
  and $(\Rtil(2))\dual  = \tfrac{1}{2}\Rtil\dual $ yields the claim.
\par
 (4) By definition,  we have $\pi_{-}(\SX)/\SXm=D_{-}$.
\par
 (5) Let $x \in \tfrac{1}{2}R$ and $y \in R\dual $. Then
 $\langle x, y \rangle_{M_R} = 2 \langle x, y \rangle_R\equiv 0 \mod \Z$
 and $x +\trl \in D_-$.  This shows that $\rl\dual /\trl\subset D_-^\perp$.
 Conversely let $x +\trl \in D_-^\perp$. 
 For $y \in R$ we have
 $\langle x ,y \rangle_R 
 = \tfrac{1}{2} \langle x , y \rangle_{M_R}
 =\langle x ,\tfrac{1}{2}  y \rangle_{M_R} \equiv 0 \mod \Z$ 
 because $\tfrac{y}{2}+\trl \in D_-=\tfrac{1}{2}R/\Rtil$.
 This shows that $x \in R\dual $.
\par
 (6) $\OG(\trl,\tfrac{1}{2}\rl/\trl) = \OG(\trl,\tfrac{1}{2}\rl) = \OG(\trl,\rl)$.
 \end{proof}
Let $\rl$ be an $\ADE$-lattice and $\Phi$ the set of its roots.
We fix a subset $\Phi^+\subset \Phi$ of positive roots. 
There exists
a unique Weyl-chamber $C$ of $\rl$ (see Definition~\ref{def:Weyl-chambers})
such that for all $r \in \Phi^+$ and $c \in C$ we have $\langle r, c\rangle>0$. 
We call $C$ the \emph{fundamental chamber}.
The positive roots perpendicular to the walls of $C$ are the so-called \emph{simple roots}.
The simple roots form a basis of $\rl$ whose Dynkin diagram is of $\ADE$-type $\tau(\rl)$.
As before we have $\OG(\rl) = W(\rl) \rtimes \OG(\rl,C)$,
where $\OG(\rl,C)$ is the 
stabilizer of $C$ in $\OG(\rl)$.
Via the action of $\OG(\rl,C)$ on the vertices of the Dynkin diagram,
we identify $\OG(\rl,C)$ 
with the symmetry group $\arl$ of the Dynkin diagram $\tau(\rl)$,
that is, we have 
\[
\OG(\rl)=W(\rl) \rtimes \Aut(\tau(\rl)).
\]
A lattice is called \emph{irreducible} if it cannot be written as a non-trivial
orthogonal sum of two sublattices. 
Definite lattices admit an orthogonal decomposition
into irreducible sublattices which is unique up to reordering (cf. \cite[27.1]{kneser}).
\begin{lemma}\label{lem:imWR}
 Let $\rl$ be an $\ADE$-lattice, and
let $\OG_0(\rl)$ be the kernel of the natural homomorphism
$\OG(\rl) \rightarrow \OG(\rl\dual /\rl)$.
 Then we have
 \[[\OG_0(\rl):W(\rl)]=n!,\]
 where $n$ is the
 number of $E_8$ components of $\tau(\rl)$.
\end{lemma}
\begin{proof}
 Since reflections with respect to roots act trivially on the discriminant group, we have $W(\rl) \subseteq \OG(\rl)_0$.
 Thus it suffices to compute the kernel of
 \[\psi\colon \arl \rightarrow \OG(\rl\dual /\rl).\]
 If $\tau(\rl)$ is irreducible, a case by case analysis shows that this map is injective: indeed
 for $A_1$, $E_7$ and $E_8$, $\arl=1$;
 for $A_k$ with $k \geq 1$, $D_k$ with $k>4$ and $E_6$ the group $\arl$ is of order two. 
 A direct computation shows that it acts faithfully on the discriminant group.
 \par
 Suppose that the root system $\tau(\rl)$ is reducible.
 The decomposition of $\tau(\rl)$ into connected components corresponds 
 to a decomposition of $\rl$ into an orthogonal sum of irreducible $\ADE$-lattices,
 which in turn induces a corresponding decomposition of the discriminant group $\rl\dual /\rl$. 
 The action of $\arl$ preserves the three decompositions.
 Hence the elements of $\ker \psi$ must preserve the components
 which have a non-trivial discriminant group, that is, all components which are not of type $E_8$.
 By the first part, they must act trivially on these components.
 Finally, since the $E_8$ diagram has no symmetry, 
 the elements in the kernel act as a permutation of the connected components of $\tau(\rl)$ of type $E_8$.
 \end{proof}
\begin{lemma}\label{lem:Gm}
 Let $\rl$ be  an $\ADE$-lattice of rank at most $10$ and
 $\trl$ an \emph{even} overlattice.
 Consider the homomorphism
 \begin{equation}\label{eq:fromOGRR}
 \OG(\rl,\trl) \rightarrow \OG(\rl\dual /\trl).
 \end{equation}
 If there is a component $\trl_j$ of $\trl$ with $\tau(\trl_j)=E_8$ and
 $\tau(\trl_j \cap \rl)=2D_4$, then the kernel of \eqref{eq:fromOGRR} is
 $W(\rl)\rtimes \langle h \rangle$ where $h \in \Aut(\tau(\rl),\trl)$
 is an involution.
 Otherwise the kernel is just the Weyl group $W(\rl)$.
 \end{lemma}
\begin{proof}
 Let $\Aut(\tau(\rl), \trl) \leq \arl$ be the stabilizer of $\trl$. 
 Since the elements of $W(\rl)$ act trivially on $\rl\dual /\rl$, they preserve $\trl$ and
\[
\OG(\trl, \rl)=W(\rl) \rtimes \Aut(\tau(\rl),\trl) \leq W(\rl) \rtimes \Aut(\tau(\rl)).
\]
 The elements of $W(\rl)$ act trivially on the domain of $\rl\dual /\rl \twoheadrightarrow \rl\dual  / \trl$, 
 so they lie in the kernel of~\eqref{eq:fromOGRR}.
 Thus it suffices to compute the kernel of
 \[\varphi\colon \Aut(\tau(\rl), \trl) \rightarrow \OG(\rl\dual /\trl).\]
Indeed, the kernel of~\eqref{eq:fromOGRR} is given by $W(R) \rtimes \ker \varphi$.
 \par
 First we suppose that $\tau(\rl)$ is irreducible.
 If $\rl = \trl$, then $W(\rl)=\OG_0(\rl)$ by Lemma \ref{lem:imWR}, and hence $\varphi$ is injective.
 Otherwise (as $\rk \rl \leq 10$) the pair
 $(\tau(\rl),\tau(\trl)) \in \{(A_7,E_7),(A_8,E_8), (D_8,E_8)\}$.
 Suppose we are in the case $(A_7,E_7)$. Then $\rl\dual /\rl \cong \Z/8\Z$ and $\trl/\rl = 4 (\rl\dual /\rl)$.
 Then $\arl$ is of order two and acts as $\pm 1$ on $\rl\dual /\rl$ which is non-trivial in
 $\rl\dual /\trl \cong \Z/4\Z$.
 A similar argument applies to $(A_8,E_8)$.
 Finally the symmetry of the $D_8$ diagram
 exchanges the two isotropic vectors of its discriminant.
 In particular it does not fix any non-trivial even overlattice 
 which implies that $\Aut(\tau(\rl),\trl)=1$ in the $(D_8,E_8)$ case. In any case $\varphi$ is injective.
\par
 Now suppose that $\rl=\bigoplus \rl_i$ has several irreducible components $\rl_i$ and
 let $h \in \ker \varphi$. 
 Note that $h$ preserves the decomposition $\rl\dual=\bigoplus \rl_i\dual $.
 Let $x \in \rl_i\dual $ be a non-zero element.
\par
 If $x^h$ lies in the same component $\rl_i\dual $ as $x$, then $h$ must preserve it.
 Hence we may restrict $h$ to this component and the previous paragraph yields $x^h=x$.
\par
 If $x$ and $x^h$ lie in different components $\rl_i\dual $ and $\rl_j\dual $,
 then these components are isomorphic and $q(x^h-x)=q(x^h)+q(x) = 2q(x)$.
 Since $h \in \ker \varphi$,
 we have $x^h - x \in \trl$ 
 Further $\trl/\rl$ is totally isotropic with respect to the discriminant form. 
 Thus $q(x^h-x)=2q(x) \equiv 0 \mod 2\Z$, i.e. $q(x) \equiv 0 \mod \Z$.
 If $y$ is any non-trivial element of $\rl_i\dual $, 
 then $x^h$ and $y^h$ lie in the same connected component $\rl_j\dual $ 
 and the same reasoning applies. 
 In particular
 \[\forall y \in \rl_i\dual \colon q(y) \equiv 0 \mod \Z\]
 which implies that $\rl_i$ is $2$-elementary and $q_{\rl_i}$ has values in $\Z/2\Z$.
 Under the constraint $\rk \rl \leq 10$, this is possible only if $\tau(\rl_i)=\tau(\rl_j)=D_4$.
 To sum up $\varphi$ is injective, except possibly if $\tau(\rl)$ has two $D_4$ components. We analyse this case in detail.
 \par
 We may assume that $\rl=\rl_1\oplus \rl_2$ is of type $2D_4$ and
 $\tilde{R}$ an overlattice of $R$. If $\trl = \rl$, then $\varphi$ is injective by Lemma \ref{lem:imWR}.
 Hence we may further assume that $\rl \subsetneq \trl$.
 Suppose there exists a non-trivial element $h$ in the kernel of $\varphi$. 
 By the previous part this implies that $R_1^h=R_2$.
\par
Let $e_1,e_2,e_3,e_4$ be the simple roots of $\rl_1$
with $e_4$ giving the central vertex of the Dynkin diagram of type $D_4$,
i.e. $\langle e_4,e_i \rangle =1$ for $i=1,2,3$.
Let $(e_1\dual , \dots,  e_4\dual ) \in \rl_1\dual $ be the dual basis. 
The four elements of $\rl_1\dual /\rl_1$ are represented by
$e_1\dual , e_2\dual , e_3\dual $ and $e_4\dual $
representing $0$.
Set $f_i = e_i^h \in \rl_2$. Then $f_i^h = e_{\sigma(i)}$
for some permutation $\sigma \in S_4$ with $\sigma(4)=4$.
Since $h \in \ker \varphi$,
we have $t_i:=e_i\dual  - f_i\dual  \in \trl$ for $i\in \{1,2,3\}$.
Now the cosets of $0$, $t_1$, $t_2$ and $t_3$ constitute a maximal totally isotropic subspace of $\rl\dual /\rl$ contained in $\trl/\rl$.
Since $\trl/R$ is totally isotropic as well, the subspaces must be equal. We conclude that $\tau(\trl)=E_8$.
By the same reasoning we have $f_i\dual-e_{\sigma(i)}\dual \in \trl$.
As $\trl/\rl$ has only four elements, this is possible only if $\sigma =1$. Hence $h$ is an involution and uniquely determined by $\trl/\rl$. This shows that the kernel of $\varphi$ is of order $2$.
\end{proof}
\begin{lemma}\label{lem:roots-mod2}
 Let $\trl$ be  an $\ADE$-lattice
  and $\Phi^+$ the set of its positive roots.
 Then the natural map $\Phi^+ \rightarrow\trl/2\trl$ is injective.
\end{lemma}
\begin{proof}
 We may assume that $\trl$ is irreducible. 
 In what follows we explicitly compute $\eta\colon \Phi^+ \rightarrow\trl/2\trl$ 
 for each case using 
 classical constructions
 of the $\ADE$-lattices (see e.g. \cite[Theorem 1.2]{ebeling:lattices-and-codes}).
\par
 Let $(\epsilon_1, \dots, \epsilon_{n+1})$ be the standard basis of $\Z^{n+1}$.
 The $n(n+1)$ roots of the lattice
 \[A_n = \left\{(x_i) \in \Z^{n+1} : \sum_{i=1}^{n+1} x_i = 0 \right\}\]
 are given by
 \[\Phi(A_n)=\{\alpha_{ij}=\epsilon_i - \epsilon_j \mid 1\leq i \neq j \leq n+1\}.\]
 Suppose that $\alpha_{ij}\equiv \alpha_{lk} \mod 2A_n\subseteq 2\Z^{n+1}$.
 Then
  we have that $\epsilon_i - \epsilon_j +\epsilon_k - \epsilon_l \equiv 0 \mod 2\Z^{n+1}$.
 This is possible only if each standard basis vector appears twice, i.e. $(i,j)=(k,l)$ or $(i,j)=(l,k)$ which means that
 $\alpha_{ij} = \pm \alpha_{lk}$. 
 Since either $\alpha_{lk}\in \Phi^+$ or $-\alpha_{lk} \in \Phi^+$, the map $\eta$ is injective.
\par
 Let $(\epsilon_1, \dots, \epsilon_{n})$ be the standard basis of $\Z^{n}$, $n\geq 4$.
 The $2n(n-1)$ roots of  the lattice
 \[D_n = \left\{(x_i) \in \Z^{n} : \sum_{i=1}^n x_i \equiv 0 \mod 2 \right\}\]
 are given by
 $\pm (\epsilon_i + \epsilon_j)$ and $\pm (\epsilon_i -  \epsilon_j)$ for $1\leq i < j \leq n$.
 Suppose that $\pm \epsilon_i \pm \epsilon_j \equiv \pm \epsilon_k \pm \epsilon_l \mod 2D_n$.
 As before this implies that $\{i,j\} = \{k,l\}$.
 Since \[(\epsilon_i + \epsilon_j)- (\epsilon_i - \epsilon_j)=2\epsilon_j \notin 2D_4,\]
 the map $\eta$ is injective.
We leave the exceptional cases $E_6,E_7,E_8$ to the reader.
\end{proof}
\begin{lemma}\label{lem:imGm}
Let $\trl=\bigoplus_{j \in J} \trl_j$ be  an $\ADE$-lattice with $\trl_j$ irreducible.
Then the kernel of the natural homomorphism
\[
\psi \colon \OG(\trl)=\OG(\trl(2)) \rightarrow \OG(\tfrac{1}{2}\trl\dual /\trl),
\]
where  $\tfrac{1}{2}\trl\dual /\trl$ is the discriminant form of $\trl(2)$,
is generated by the elements $\oplus_{j\in J} g_j$ with $g_j = \pm 1_{\trl_j}$
 if $\trl_j$ is unimodular and $g_j = 1_{\trl_j}$  otherwise.
\end{lemma}
\begin{proof}
 We identify $\tfrac{1}{2}\trl\dual /\trl$ and $\trl\dual /2\trl$. Let $g \in \ker \psi$.
 Since $\trl \subseteq \trl\dual $, $g$ acts trivially
 on $\trl\dual /2\trl\dual $. The action of $\OG(\trl)$ preserves the decomposition $\trl = \bigoplus_{j \in J} \trl_j$.
 In particular $g$ acts on the set $J$. As $\trl\dual /2\trl\dual = \bigoplus_{j \in J}\trl_j\dual /2\trl_j\dual $ and $g$ is in $\ker \psi$ we have $j^g=j$. Hence $g$ must fix each connected component of $\trl$ and we may and will assume that $\trl$ is irreducible.
\par
 We tensor the perfect pairing $\trl\dual  \times \trl \rightarrow \Z$ with $\F_2$,
 to obtain a perfect pairing $\trl\dual /2\trl\dual   \times \trl/2\trl \rightarrow \F_2$. Since $g$ acts trivially on the first factor, so does it on the second factor $\trl/2\trl$.
 By Lemma \ref{lem:roots-mod2} $\Phi(\trl)/\{\pm 1\} \cong \Phi^+(\trl)$ injects into $\trl/2\trl$, which implies that
 $g(r)= \pm r$ for every root $r\in \Phi(\trl)$.
 As any simple root system of $\trl$ is connected, the sign is the same for each simple root.
 Since the simple roots form a basis, $g=\pm 1$.
\par
 Set $\trl_\pm= \ker(g\mp 1)\subset \trl$.
 We apply Proposition \ref{prop:03} to the primitive extension
 \[\trl_+(2) \oplus \trl_-(2)
 \subseteq \trl(2).\]
 Since $g$ acts trivially on the discriminant
 group $\tfrac{1}{2}\trl\dual /\trl$ of $\trl(2)$,
 the implication
 \begin{equation}\label{eq:g-implies}
g|_{\tfrac{1}{2}\trl_{+}\dual /\trl_+}=1 \implies
 g|_{\tfrac{1}{2}\trl_{-}\dual /\trl_-}=1
 \end{equation}
 holds.
 By definition $g|_{\trl_-}=-1_{\trl_-}$ and then by the right hand side of (\ref{eq:g-implies}), 
 the lattice $\trl_{-}(2)$ must be $2$-elementary, i.e. $\trl_-$ is unimodular.
 In particular $\trl_-$ is a direct summand of $\trl$.
 But we assumed the latter to be irreducible, so that
 $\trl\in\{0,\trl_-\}$. Thus $g=\pm 1$ if $\trl$ is unimodular and $g = 1$ else.
\end{proof}
\begin{remark}\label{rem:-1}
Let $\rl$ be an irreducible $\ADE$-lattice.
By \cite[Proposition 1.5]{ebeling:lattices-and-codes}, we have $-1 \in W(\rl)$ if and only if
$\rl$ contains $\rk \rl$ pairwise orthogonal roots,
if and only if $\tau(\rl)$ is one of $A_n$ ($n\geq 1$), $D_n$ ($n\geq 4$, $n$ even), $E_7$, $E_8$.
\end{remark}
\begin{theorem}\label{thm:cde}
 Let $Y$ be a $(\tau, \taubar)$-generic Enriques surface,
  and let $\rl$, $\brl$, $\trl$ be as in Table \ref{table:184}.
Let $\trl = \bigoplus_j \trl_j$ be the decomposition into irreducible components. 
 Then we have 
 \[
 |\closure{G}_{X-}|  = |W(R)|\;  \frac{d_{(\tau,\taubar)}}{ e_{(\tau,\taubar)}},
 \]
 where $d_{(\tau,\taubar)}, e_{(\tau,\taubar)}$ are given as follows.
\begin{eqnarray*}
 d_{(\tau,\taubar)}&:=&
 \begin{cases}
 2 & \text{$\exists$ $j$    such that
  $\tau(\trl_j) = E_8$ and $\tau(\trl_j\cap \rl) = 2D_4$, }\\
 1 & \text{otherwise,}
 \end{cases} 
 \\
  e_{(\tau,\taubar)}&:=&
  \begin{cases}
  2 & \text{$\exists$ $j$     such that    $\tau(\trl_j) = E_8$ and $\trl_j \cap \rl$ contains $8$ orthogonal roots,}\\
 1 & \text{otherwise.}
 \end{cases}
\end{eqnarray*}
 Hence the value of $c_{(\tau,\taubar)}$  in Table \ref{table:184} is 
 equal to 
$e_{(\tau,\taubar)}/  d_{(\tau,\taubar)}=|\Aut_{nt}(Y)|/d_{(\tau,\taubar)}$.
\end{theorem}
\begin{proof}
By Theorem~\ref{thm:volumeformula} and Lemma \ref{lem:dict}, we have
\[G_{X-}= 
\ker (\OG(\trl, \rl) \rightarrow \OG(\rl\dual / \trl)), 
\]
which, by Lemma \ref{lem:Gm}, is given
 by $W(\rl)$, or by $W(\rl)\rtimes \langle h \rangle$ for some involution $h \in \Aut(\tau(\rl),\trl)$ if there is some component $\trl_j$ with $\tau(\trl_j)=E_8$ and $\tau(\trl_j \cap \rl)=2D_4$.
Consider the natural homomorphism $\psi\colon \OG(\trl) \rightarrow \OG(\tfrac{1}{2}\trl\dual /\trl)$
in Lemma~\ref{lem:imGm}.
By our dictionary in Lemma \ref{lem:dict}, 
we have $\closure{G}_{X-} = \psi(G_{X-})$.
 By Lemma \ref{lem:imGm},
 the kernel 
 of 
 $\psi$ consists of those $g=\oplus_{j \in J} g_j$ with $g_j = \pm 1_{\trl_j}$ 
 if $\trl_j$ is unimodular and $g_j = 1_{\trl_j}$ else.
Further $\Ker \psi  \cap W(R)$ consists of those
$g$ with $g_j = \pm 1$ if $\trl_j$ is unimodular and $-1 \in W(R\cap \trl_j)$,  and $g_j=1$ else. 
Now Remark \ref{rem:-1} yields the condition for $e_{(\tau,\taubar)}$.
Since the $g_j=\pm 1$ do not preserve any positive root system, 
the involution $h$ is not in $\Ker \psi $. 
This explains the presence of $d_{(\tau,\taubar)}$.
Finally, in the geometric situation,
we have $G_{X-}=W(R)$
(see Remark~\ref{rem:142and170} below), 
 and hence $\Aut_{nt}(Y) \cong \ker\left( G_{X-} \to \Gbar_{X-}\right) = \Ker \psi \cap W(R)$ gives $e_{(\tau,\taubar)}=|\Aut_{nt}(Y)|$,
 where the isomorphism 
follows from Lemma~\ref{lem:numerically-trivial}.
\end{proof}
%
%
%
\begin{remark}\label{rem:142and170}
The factor $d_{(\tau,\taubar)}$ is nontrivial only for
Nos.~142 and~170 which are not realized geometrically.
This is explained by an extra ``automorphism" of $Y$ which exchanges two $D_4$ configurations of ``smooth rational curves" and acts trivially on their orthogonal complement in $\SY$. 
This is not visible in the Weyl group.
Thus in the geometric cases a nontrivial contribution of $c_{(\tau,\taubar)}=e_{(\tau,\taubar)}$ is indeed explained by the presence of a numerically trivial involution of $Y$.
\end{remark}

\section{Borcherds' method}\label{sec:Borcherds}
\subsection{An algorithm on a graph}\label{subsec:graph}
The algorithms to prove our main results are variations of 
the following computational procedure.
\par
Let $(V, E)$ be a simple non-oriented connected graph,
where $V$ is the set of vertices
and $E$ is the set of edges,
which is a set of non-ordered pairs of distinct elements of $V$.
The set  $V$ may be  infinite.
Suppose that a group $G$ acts on $(V, E)$ from the right.
We assume the following.
\newcommand{\AssumpVE}{$\mathord{\rm VE}$}
\begin{enumerate}[(\AssumpVE-1)]
\item
For any vertex $v\in V$,
the set $\set{v\sprime \in V}{\{v, v\sprime\}\in E}$
of vertices adjacent to $v$ is finite and can be calculated
effectively.
\item
For any vertices $v, v\sprime\in V$, we can determine effectively
whether the set
\begin{equation}\label{eq:TG}
\TG (v, v\sprime):=\set{g\in G}{v^g=v\sprime}
\end{equation}
is empty or not,
and when it is non-empty,
we can calculate an element of $\TG (v, v\sprime)$.
\item
For any  $v\in V$,
the stabilizer subgroup $\TG (v, v)$ of $v$ in $G$
is finitely generated, and a finite set of generators of $\TG (v, v)$
can be calculated
effectively.
\end{enumerate}
We define the \emph{$G$-equivalence relation} $\sim$ on $V$ by
\[
v\sim v\sprime\;\;\Longleftrightarrow\;\; \TG (v, v\sprime)\ne\emptyset.
 \]
Suppose that $V_0$ is
a non-empty finite subset of $V$ with the following properties.
\newcommand{\PropertyVzero}{$\mathord{\rm V}_0$}
\begin{enumerate}[(\PropertyVzero-1)]
\item If $v, v\sprime \in V_0$ are distinct,
then $v\not\sim v\sprime$.
\item
We put
$\widetilde{V}_0:=\set{v\in V}{ \textrm{$v$ is adjacent to a vertex belonging to $V_0$}}$.
Then, for each  $v\in \widetilde{V}_0$,
there exists a vertex $v\sprime\in V_0$ such that  $v\sim v\sprime$.
\end{enumerate}
For each $v\in \widetilde{V}_0$,
we choose an element
$h(v)\in \TG (v, v\sprime)$,
where $v\sprime$ is
the unique vertex in $V_0$ such that  $v\sim v\sprime$,
and put
\[
\HHH:=\set{h(v)}{v\in \widetilde{V}_0}.
\]
We fix an element $v_0\in V_0$.
\begin{proposition}
The natural mapping
\begin{equation}\label{eq:V0bijection}
V_0\inj V\surj V/\mathord{\sim}=V/G
\end{equation}
is a bijection,
and the group $G$ is generated by the union of $\TG (v_0, v_0)$ and  $\HHH$.
\end{proposition}
\begin{proof}
Let $\gen{\HHH}$ be the subgroup of $G$ generated by $\HHH$.
First we prove that,
for any $v\in V$, there exists an element $h\in \gen{\HHH}$ such that
$v^h\in V_0$.
Let an element $v\in V$ be fixed.
A sequence
\begin{equation}\label{eq:sequence}
v\sbar{0}, v\sbar{1}, \dots, v\sbar{l}
\end{equation}
of vertices is said to be a \emph{path from $V_0$ to $v^{\gen{\HHH}}$}
if $v\sbar{i-1}$ and $v\sbar{i}$ are adjacent for $i=1, \dots, l$,
the starting vertex $v\sbar{0}$ is in $V_0$,
and the ending vertex $v\sbar{l}$ belongs to the orbit $v^{\gen{\HHH}}$ of
the fixed vertex $v$
under the action of $\gen{\HHH}$.
Since $(V, E)$ is connected and $V_0$ is non-empty,
there exists at least one path from $V_0$ to $v^{\gen{\HHH}}$.
Suppose that the sequence~\eqref{eq:sequence} is
a path from $V_0$ to $v^{\gen{\HHH}}$  of length $l>0$.
Since $v\sbar{1}$ is adjacent to the vertex $v\sbar{0}$ in $V_0$,
we have $v\sbar{1}\in \widetilde{V}_0$ and
there exists an element $h_1:=h(v\sbar{1})\in \HHH$ that maps $v\sbar{1}$ to
an element of $V_0$.
Then
\[
v\sbar{1} \sp{h_1}, \dots, v\sbar{l} \sp{h_1}
\]
is a path from $V_0$ to $v^{\gen{\HHH}}$  of length $l-1$.
Thus we obtain a path from $V_0$ to $v^{\gen{\HHH}}$  of length $0$,
which implies the claim.
\par
The injectivity of~\eqref{eq:V0bijection}
follows from property~(\PropertyVzero-1) of $V_0$.
The surjectivity
follows from the claim above.
Suppose that $g\in G$.
By the claim,
there exists an element $h\in \gen{\HHH}$ such that $v_0^{gh}\in V_0$.
By property~(\PropertyVzero-1) of $V_0$,
we have $v_0=v_0^{gh}$ and hence $gh\in \TG (v_0,v_0)$.
Therefore $G$ is generated by the union of $\HHH$ and $\TG (v_0, v_0)$.
\end{proof}
To obtain  $V_0$ and $\HHH$,
we employ Procedure~\ref{procedure:genB}.
This procedure terminates
if and only if $|V/G|<\infty$.
\begin{procedure}
\begin{algorithmic} 
\State Initialize $V_0:=[v_0]$,  $\HHH:= \{\}$, and $i:= 0$.
\While{$i <|V_0|$}
    \State Let $v_i$ be the $(i+1)$st entry of the list $V_0$.
    \State Let $\AAA(v_i)$ be the set of vertices adjacent to $v_i$.
    \For {each vertex $v\sprime$ in $\AAA(v_i)$}
    	\State  Set ${\tt flag}:=\mathrm{true}$.
    	\For {each  $v\spprime$ in $V_0$}
    		\If{$\TG (v\sprime, v\spprime)\ne \emptyset$}
            	\State Add an element $h$ of $\TG (v\sprime, v\spprime)$ to $\HHH$.
            	\State Replace ${\tt flag}$ by $\mathrm{false}$.
            	\State Break from the innermost for--loop.
            \EndIf
        \EndFor
        \If{${\tt flag}=\mathrm{true}$}
        	\State Append $v\sprime$ to the list $V_0$ as the last entry.
        \EndIf
    \EndFor
    \State Replace $i$ by $i+1$.
\EndWhile
\end{algorithmic}
 \caption{A computational procedure on a graph}\label{procedure:genB} 
\end{procedure}
\subsection{$17$ primitive embeddings}\label{subsec:17}
Recall that $L_{26}$ is an even unimodular hyperbolic lattice
of rank $26$.
The $L_{26}$-chamber (that is, the standard fundamental domain of $W(L_{26})$)
was studied by Conway~\cite{Conway1983}.
He constructed a bijection between the set of walls of 
an $L_{26}$-chamber $D$ and the set of vectors of the Leech lattice,
and showed that the automorphism group $\OG(L_{26}, D)$ of $D$ is
isomorphic to the group of \emph{affine} isometries of the Leech lattice.
Using this result,
Borcherds~\cite{Bor1},~\cite{Bor2} developed a method
to calculate the orthogonal group of an even  hyperbolic lattice $S$
by embedding $S$ primitively into $L_{26}$
and investigating the tessellation of 
an $S$-chamber (that is, a standard fundamental domain of $W(S)$)
by $L_{26}/S$-chambers.
\par
In~\cite{BS2019},
we apply this method to $S=L_{10}(2)$.
We fix positive half-cones  $\PPP_{10}$ of $L_{10}$ and $\PPP_{26}$ of $L_{26}$.
In~\cite{BS2019},
we have proved the following.
\begin{theorem}[\cite{BS2019}]\label{thm:17}
Up to the action of $\OG(L_{10})$ and $\OG(L_{26})$,
there exist exactly $17$ primitive embeddings of
$L_{10}(2)$ into $L_{26}$.
\qed
\end{theorem}
These $17$ primitive embeddings of $L_{10}(2)$ into $L_{26}$ are named as
\[
{\tt 12A},\;
{\tt 12B},\;
{\tt  20A},\; \dots\;,
{\tt 20F},\;
{\tt 40A},\;  \dots\;,
{\tt 40E},\;
{\tt 96A},\;
{\tt 96B},\;
{\tt 96C},\;
{\tt infty}.
\]
Recall the notion of being reflexively simple from Definition~\ref{def:reflexively}.
\begin{theorem}[\cite{BS2019}]\label{thm:16simples}
Suppose that 
a primitive embedding $\Lten(2)\inj L_{26}$ 
is  not of type~{\tt infty},
Then 
each $\LLt$-chamber has only finitely many walls, and 
they are defined by roots of $L_{10}$.
Moreover the tessellation of $\PPP_{10}$ by $\LLt$-chambers is reflexively simple.
\qed
\end{theorem}
The explicit description of the $17$ primitive embeddings
and $\LLt$-chambers
is given in~\cite{BS2019} and~\cite{BrandhorstShimadaCompData}.
From these data, we see the following.
Let $L_{10}(2) \inj L_{26}$ be a primitive embedding
whose type is not  {\tt infty},
and $D$ an $\LLt$-chamber.
The automorphism group of $D$ is denoted by
\[
\OG(L_{10}, D):=\set{g\in \OGP(L_{10})}{D^g=D}.
\]
Since the walls of $D$ are defined by  roots of $L_{10}$, 
the chamber $D$ is tessellated by Vinberg chambers.
The volume of $D$ is defined by
\[
\vol(D):=\textrm{the number of Vinberg chambers contained in $D$}.
\]
Let $f$ be a face of $D$ with codimension $k$.
Then the defining roots of the walls of $D$ containing $f$
form a configuration whose dual graph is a Dynkin diagram
of an $\ADE$-type.
The \emph{$\ADE$-type of $f$} is the $\ADE$-type of this Dynkin diagram.
The closure $\clD$ of $D$ in $\SX\tensor\R$ 
contains only a finite number of isotropic rays.
Let $v\in \SX\cap \clD$ be a primitive isotropic ray (see Section~\ref{subsec:faces}).
Then the defining roots $r$ of walls of $D$
such that $\intf{r, v}=0$
form a configuration whose dual graph is a Dynkin diagram
of an \emph{affine} $\ADE$-type.
The \emph{affine $\ADE$-type} of the  isotropic ray $\R_{>0}v$ is
the affine $\ADE$-type of this Dynkin diagram.
\begin{example}\label{example:96C}
Let $L_{10}(2) \inj L_{26}$ be the primitive embedding of type ${\tt 96C}$,
and  $D_0$  an $\LLt$-chamber.
Then $D_0$ has exactly $96$ walls.
The group $\OG(L_{10}, D_0)$ is 
of order $110592=2^{12}\cdot 3^3$, and
this group acts on the set of walls of $D_0$ transitively.
We have
\[
\vol(D_0)=\frac{1_{\BP}}{72}=652758220800.
\]
The $\LLt$-chamber  $D_0$ has $1728+768+144$ faces of codimension $2$,
which are decomposed into orbits of size $1728$, $768$, $144$
under the action of $\OG(L_{10}, D_0)$.
Hence each wall of $D_0$
is bounded by $36+16+3=55$ faces of codimension $2$ of $D_0$.
The $\ADE$-types of faces in these orbits are 
$2A_1$, $2A_1$, $A_2$, 
respectively.
The $\LLt$-chamber  $D_0$ has $18+256+256+864$  isotropic rays,
which are decomposed into orbits of size $18$, $256$, $256$, $864$
by the action of $\OG(L_{10}, D_0)$.
The affine $\ADE$-types of  isotropic rays of these  orbits
are $8{A}_1$, $4{A}_2$,  $4{A}_2$,  $2{A}_1+2{A}_3$,
respectively.
\end{example}
\subsection{Constructing $\SX$}\label{subsec:construction}
Let $Y$ be an Enriques surface 
with the universal covering $\pi\colon X\to Y$.
We consider the following assumption:
\begin{equation}\label{eq:assumpSXL26}
\parbox{10.5cm}{
we have a primitive embedding $\SX\inj L_{26}$
such that  the composite $ \SY(2)\cong L_{10}(2) \inj L_{26}$ of 
$\pi\sp*\colon \SY(2)\inj \SX$ and
$\SX\inj L_{26}$ 
is not of type {\tt infty}, and we have the list of walls 
of an $\LSY$-chamber $D_0$ that is contained in $\NefY$.}
\end{equation}
Suppose that~\eqref{eq:assumpSXL26} holds.
Then  $\PPP_Y$
has the following three tessellations,
each of which  is a refinement of the one below.
\begin{itemize}
\item by Vinberg chambers, 
\item by $\LSY$-chambers, each of which has only finite number of walls, and 
\item by $\SX/\SY(2)$-chambers, one of which is $\NefY$.
\end{itemize}
The tessellation of 
$\NefY$ by $\LSY$-chambers is very useful in analyzing  $\NefY$.
Recall that $G_Y\subset \OGP(\SY)$ is the image of the projection of $G_X\subset \OGP(\SX)$
defined by~\eqref{eq:GX}.
\begin{proposition}\label{prop:tressellationpreserve}
Suppose that $Y$ satisfies~\eqref{eq:assumpTXomega}~and~\eqref{eq:assumpSXL26}. 
Then the action of $G_Y$ on
$\PPP_Y$ preserves the  tessellation of $\PPP_Y$ by $\LSY$-chambers.
In particular, the action of $\aut(Y)$ 
on $\NefY$ preserves the  tessellation of $\NefY$ by $\LSY$-chambers.
\end{proposition}
\begin{proof}
It is enough to prove that 
the action of $\tilg\in G_X$ on
$\PPP_X$ preserves the  tessellation of $\PPP_X$ by $L_{26}/\SX$-chambers.
Let  $\id_P$ be the identity  of the orthogonal complement $P$ of $\SX$ in $L_{26}$.
Since
 the action of $\tilg$ on $\SX\dual/\SX$ is  $1$,
the action of $(\tilg, \id_P)$ on $\SX\oplus P$
 preserves the even unimodular overlattice $L_{26}$ of $\SX\oplus P$.
 Thus
$\tilg$ extends to
 an isometry of $L_{26}$, and hence
 its action on $\PPP_X$
  preserves the $L_{26}/\SX$-chambers.
  The second assertion follows from the fact that $\aut(Y)$ is 
  the stabilizer subgroup of $\NefY$ in $G_Y$.
  \end{proof}
The purpose of this section is 
to construct a primitive embedding $\SX\inj L_{26}$ 
for a $(\tau, \taubar)$-generic Enriques surface $Y$, so that we can assume
\eqref{eq:assumpSXL26}.
We start from a  primitive embedding
$\iota\colon L_{10}(2)\inj L_{26}$
whose type  is not {\tt infty} and which has  a fixed  $\LLt$-chamber $D_0$,
and then proceed to
the  construction of $\SX$  between $L_{10}(2)\cong \SY(2)$ and $L_{26}$
such that the inclusion of 
$L_{10}(2)\cong \SY(2)$ into $\SX$ is the embedding $\pi^*$,
and that the fixed $\LLt$-chamber $D_0$ is contained in $\NefY$.
\par
Recall that, for a $(\tau, \taubar)$-generic Enriques surface $Y$,
the lattice $\SX$ 
is obtained from $\SY(2)$ by adding roots 
of the form $(r+v)/2$,
where 
$r$ is a root of $\SY$ and $v$ is a $(-4)$-vector 
in $\SXm$.
To find roots in $L_{26}$ that yield 
an appropriate extension from $\SY(2)$ to $\SX$,
we search for pairs
$\alpha=(r, v)$ of a root $r$ of $L_{10}$ defining a wall of $D_0$
and a $(-4)$-vector $v$ of 
$Q_{\iota}$
such  that $(r+v)/2$ is in $L_{26}$,
where $Q_{\iota}$ is the orthogonal complement of $L_{10}(2)$ in $L_{26}$.
For a finite set
$p=\{\alpha_1, \dots, \alpha_m\}$
of such pairs,
we consider the sublattice $M_p$
of $L_{26}$ generated by
$L_{10}(2)$ and the roots
$(r_1+v_1)/2, \dots, (r_m+v_m)/2$ of $L_{26}$,
where  $\alpha_i=(r_i, v_i)$.
Suppose that $p=\{\alpha_1, \dots, \alpha_m\}$ satisfies the following:
\begin{enumerate}[(i)]
\item 

The dual graph of
$r_1, \dots, r_m$ is a Dynkin diagram of some $\ADE$-type $\tau$.
By Proposition~\ref{prop:184}, 
the primitive closure  $\Rbar$ of
the  $\ADE$-sublattice $R$ of $L_{10}$ generated by $r_1, \dots, r_m$
is also an $\ADE$-sublattice of $L_{10}$.
Let  $\taubar$ denote the $\ADE$-type of $\Rbar$.
%
\end{enumerate}
By Proposition~\ref{prop:184}, 
 the embedding $\Lten (2) \inj M_p$ is isomorphic to $\Lten  (2) \inj M_R$,
 and hence,  by Proposition~\ref{prop:Rtil}, 
 we see that $L_{10}(2)$ is a primitive sublattice of $M_p$,
 and  the orthogonal complement  of $L_{10}(2)$ in $M_p$
contains no roots.
 %
%
We consider the following condition:
\begin{enumerate}[(i)]
\setcounter{enumi}{1}
\item $M_p$ can be embedded primitively
into the $K3$ lattice
(an even unimodular lattice of rank $22$ with signature $(3, 19)$).
This condition is checked by calculating the discriminant form of $M_p$
and applying the theory of genera~(see~\cite{Nikulin79}).
\end{enumerate}
Suppose that $M_p$ satisfies condition (ii).
Since $22-\rank M_{p}=12-m>2$,
the surjectivity of the period mapping
of complex $K3$ surfaces~(\cite[Chapter~VIII]{CCSBook}) implies that
there exists a $K3$ surface $X$ with $M_p \cong \SX $
such that $\OG(T_X, \omega)=\{\pm 1\}$.
Moreover,
by~\cite{Keum1990},
the $K3$ surface $X$ has a fixed point free involution $\enrinvol$
with the quotient morphism $\pi\colon X\to Y=X/\gens{\enrinvol}$
to the Enriques surface $Y$
such that, under suitable choices of
isometries $M_p \cong \SX $, the embedding $L_{10}(2)\inj M_p$
is identified with $\pi^*\colon \SY(2)\inj \SX$.
By the construction of $M_p$,
this Enriques surface $Y$ is $(\tau, \taubar)$-generic.
Thanks to Proposition~\ref{prop:SXSY2simple},
we can further assume that
$D_0$ is contained in $\NefY$ by
changing the isometry $M_p \cong \SX$.
\par
Except for the type $(\tau, \taubar)$ of Nos~88~and~146,
we can find a set $p=\{\alpha_1, \dots, \alpha_m\}$ satisfying condition (i)  above
using the primitive embedding $\iota\colon L_{10}(2)\inj L_{26}$
given in the 8th column ({\tt irec})  of Table~\ref{table:184}.
If the 5th column ({\tt exist}) is not marked by $\times$, then  $M_p$  satisfies 
condition (ii). 
\begin{example}\label{example:constructSX}
Let $\iota\colon L_{10}(2)\inj L_{26}$
be the primitive embedding of type {\tt 96C}
(see Example~\ref{example:96C}).
Then the even negative definite lattice $Q_{\iota}$
contains  $2208$ vectors $v$ of square-norm  $-4$,
and we have $192$ pairs $\alpha=(r, v)$ such that $(r+v)/2\in L_{26}$.
Choosing appropriate subsets  from these $192$ pairs,
we can construct $\SX$ for many types $(\tau, \taubar)$ (Nos.~1, 2, \dots).
\end{example}
\begin{remark}\label{rem:nonexisting}
Even when $M_p$ does not satisfy condition (ii),
we can use $M_p$ as the N\'eron-Severi lattice $\SX$ of a  ``non-existing $K3$ surface" $X$
and run the geometric algorithms below.
\end{remark}
\section{Geometric algorithms}\label{sec:geomalgo}
We prepare some algorithms that will be used in the application of
 the generalized Borcherds' method to geometric situations.
  \par
 Let $Y$ be an Enriques surface 
 with the universal covering $\pi\colon X\to Y$.
 We assume~\eqref{eq:assumpTXomega}~and~\eqref{eq:assumpSXL26}.
First we prepare the following computational data:
\begin{enumerate}[(i)]
\item an integral interior point $a_{Y0}\in \SY$ of $D_0$, which is an ample class of $Y$, 
\item the list of roots defining the walls of $D_0$, 
\item the finite group $\OGP(\SY, D_0)=\set{g\in \OGP(\SY)}{D_0^g=D_0}$, 
\item the finite group $\OG(\SXm)$, and
\item the list of $(-4)$-vectors of $\SXm$.
\end{enumerate}
\subsection{Separating roots}\label{subsec:separating}
\begin{definition}
Let $L$ be an even  hyperbolic lattice with a positive half-cone $\PPP$, and
let $a_1, a_2$ be elements of $\PPP\cap L$.
We say that a hyperplane $(v)\sperp$ of $\PPP$
\emph{separates $a_1$ and $a_2$}
if $\intf{v, a_1}$ and $\intf{v, a_2}$ are non-zero and have different signs.
We say that a vector $v\in L\tensor\Q$
with $\intf{v, v}<0$ \emph{separates $a_1$ and $a_2$}
if $(v)\sperp$ separates $a_1$ and $a_2$.
\end{definition}
By an algorithm given in~\cite{ShimadaChar5},
we can calculate, for any $a_1, a_2\in \PPP\cap L$, 
 the set of roots of $L$ that
separate $a_1$ and $a_2$.
\subsection{Splitting roots}\label{subsec:splitting}
\begin{definition}
We say that a root $r$ of $\SY$ \emph{splits in $\SX$}
if there exists a root  $\tilr$ of $\SX$ such that $\pi^*(r)=\tilr+\tilr^\enrinvol$.
 \end{definition}
A root $r$ of $\SY$ splits in $\SX$ if and only if 
there exists a $(-4)$-vector $v$ of $\SXm$ such that $(\pi^*(r)+v)/2\in \SX$.
Hence we can effectively determine whether a given root $r$ of $\SY$ splits in $\SX$ or not.
Moreover, when $r$ splits,
we can calculate the  roots  $\tilr=(\pi^*(r)+v)/2$ and $\tilr^\enrinvol=(\pi^*(r)-v)/2$ of $\SX$
such that $\pi^*(r)=\tilr+\tilr^\enrinvol$.
\par
Suppose that  a  root $r$ of $\SY$  satisfies that 
$\NefY\cap (r)\sperp $ contains a non-empty open subset of $(r)\sperp $
and that $\intf{r, a_Y}>0$ for an ample class $a_Y$ of $Y$.
Then the following are equivalent:
\begin{itemize} 
\item $\NefY\cap (r)\sperp $ is a wall of $\NefY$ 
(that is, the hyperplane $(r)\sperp$ is disjoint from  the interior of $\NefY$),
\item $r$ splits in $\SX$, and 
\item $r$ is the class of a smooth rational curve $C$ on $Y$.
\end{itemize}
In this case, 
the  roots  $\tilr$ and $\tilr^\enrinvol$ of $\SX$ 
are the classes of  the smooth rational curves 
 $\tilC$ and $\tilC^\enrinvol$ on $X$ such that $\pi\inv(C)=\tilC+\tilC^\enrinvol$.
\subsection{Membership criterion of $G_Y$ in  $ \OGP(\SY)$}\label{subsec:Membership1}
An element $g$ of $ \OGP(\SY)$ belongs to $G_Y$
if and only if 
there exists an isometry $h\in \OG(\SXm)$ such that the action of 
$(g, h)$ on $\SXp\oplus \SXm$ preserves 
the overlattice $\SX$ and that $\tilg:=(g, h)|\SX$ 
acts on  $\SX\dual/\SX$ trivially.
Since we have the list of elements of the finite group $\OG(\SXm)$,
we can determine whether an element $g\in \OGP(\SY)$
belongs to $G_Y$ or not,
and if  $g\in \OGP(\SY)$,
we can calculate a  lift $\tilg\in G_X$ of $g$.
\subsection{Membership criterion of $\aut(Y)$ in  $ G_Y$}\label{subsec:Membership2}
Suppose that $g\in G_Y$, 
and let $\tilg \in G_X$ be a lift  
of $g$.
Recall from Proposition~\ref{prop:GY}
 that $g$ belongs to $\aut(Y)$ if and only if $g$ preserves $\NefY$,
or equivalently $\tilg$ preserves $\NefX$.
Hence $g\in\aut(Y)$ holds if and only if one of the following conditions
that are mutually equivalent is satisfied:
\begin{itemize}
\item For any ample classes $a_X$ and $a_X\sprime$ of $X$,
there exist no root 
of $\SX$ separating $a_X^{\tilg}$ and $a_X\sprime$.
\item For any ample classes $a_Y$ and $a_Y\sprime$ of $Y$,
any  roots of $\SY$ separating $a_Y^{g}$ and $a_Y\sprime$
does not split in $\SX$.
\item There exist ample classes $a_X$ and $a_X\sprime$ of $X$
such that 
there exist no roots of $\SX$ separating $a_X^{\tilg}$ and $a_X\sprime$.
\item  There exist ample classes $a_Y$ and $a_Y\sprime$ of $Y$ such that 
any root 
of $\SY$ separating $a_Y^{g}$ and $a_Y\sprime$
does not split in $\SX$.
\end{itemize}
Thus we can determine effectively whether 
a given isometry $g\in G_Y$ belongs to $\aut(Y)$ or not,
because we have at least one ample class $a_{Y0}$ of $Y$.
\subsection{Criterion for $\aut(Y)$-equivalence}\label{subsec:criterionGequiv}
Recall from Theorem~\ref{thm:16simples} that, 
for every $\LSY$-chamber $D$,
we have an isometry $g\in \OGP(\SY)$ such that $D=D_0^g$.
Let $D_1$ and $D_2$ be $\LSY$-chambers.
Suppose that we have
isometries $g_1, g_2\in \OGP(\SY)$  such that $D_1=D_0^{g_1}$ and $D_2=D_0^{g_2}$.
Then the set
\[
\isoms(D_1, D_2):=\set{g\in \OGP(\SY)}{D_1^g=D_2}=g_1\inv \cdot \OGP(\SY, D_0) \cdot g_2
\]
is finite, and can be explicitly calculated.
Therefore we can calculate the set 
\[
\isoms(Y, D_1, D_2):=\aut (Y)\cap \isoms(D_1, D_2)
\]
explicitly, and in particular, we can calculate the group
$\aut(Y, D):=\isoms(Y, D, D)$
for an $\LSY$-chamber $D$.
\section{Proofs of main theorems}\label{sec:proofs}
We present algorithms that
prove
Theorems~\ref{thm:main2rats}~and~\ref{thm:main3ells}.
Let $Y$ be an Enriques surface with the universal covering $\pi\colon X\to Y$.
Suppose that $Y$ is $(\tau, \taubar)$-generic,
where $(\tau, \taubar)$ is \emph{not} equal to No.~88 nor No.~146
in Table~\ref{table:184}, so that 
we can assume~\eqref{eq:assumpTXomega}~and~\eqref{eq:assumpSXL26}.
\subsection{Generators of $\aut(Y)$ and representatives of $\NefY/\aut(Y)$}
\label{subsec:NefY}
 We calculate a finite generating set of $\aut(Y)$ and 
 a complete set of representatives  of $\NefY/\aut(Y)$.
 This calculation affirms Theorem~\ref{thm:main1vol} computationally.
 Moreover 
the results will be used in 
the proofs of Theorems~\ref{thm:main2rats}~and~\ref{thm:main3ells} below.
\par
Let $(V, E)$ be the graph where
$V$ is the set of $\LSY$-chambers contained in $\NefY$
and  $E$ is defined by the adjacency relation of $\LSY$-chambers.
Let $G$ be the group $\aut(Y)$,
and let $v_0\in V$ be the $\LSY$-chamber $D_0$ in $\NefY$.
Let   $D=D_0^g$ be an  $\LSY$-chamber contained in $\NefY$,
where $g\in \OGP(\SY)$.
Then we can calculate  the set of roots defining the walls of  $D$
by mapping the set of roots  defining the walls of  $D_0$ by the isometry $g$.
For each root $r$ defining a wall of $D$,
the  chamber $D^{s_r}=D_0^{gs_r}$ adjacent to $D$ across the wall $D\cap(r)\sperp$
of $D$
is contained in $\NefY$ if and only if 
$r$ does \emph{not} split in $\SX$.
Therefore we can determine $D^{s_r}\subset \NefY$ or not by
the method in Section~\ref{subsec:splitting}.
Therefore condition~(\AssumpVE-1)
in Section~\ref{subsec:graph} is satisfied.
Since we can calculate $\tisom(Y, D_0^g, D_0^{g\sprime})$
for any $g, g\sprime\in \OGP(\SY)$ by Section~\ref{subsec:criterionGequiv}, 
conditions~(\AssumpVE-2) and~(\AssumpVE-3) are also satisfied.
Therefore we can apply Procedure~\ref{procedure:genB} to 
the graph $(V, E)$ and the group $G$,
and obtain
a complete set $V_0$ of representatives of orbits of the action of $G$ on $V$,
the stabilizer subgroups 
$\tisom(Y, D, D)=\aut(Y, D)$ of these representatives $D\in V_0$, and
a generating set
\[
\GGG:=\HHH\cup \aut(Y, D_0)
\]
of $\aut(Y)$.
Then we have
\begin{equation}\label{eq:voloverAut}
\vol(\NefY/\aut(Y))=\vol(D_0) \sum_{D\in V_0} \frac{1}{| \aut(Y, D)|}.
\end{equation}
Thus  Theorem~\ref{thm:main1vol} is computationally affirmed.
\begin{remark}\label{rem:computationsize}
The amount of computation of Procedure~\ref{procedure:genB} grows quadratically
as $|V/G|$
becomes large,
because we have to check $T_G(v, v\sprime)=\emptyset$
 for all pairs of distinct $v, v\sprime \in V_0$.
We could calculate a finite generating set
of $\aut(Y)$ 
by using, naively,
the graph $(V\sprime, E\sprime)$,
where $V\sprime$ is the set of Vinberg chambers
contained in $\NefY$ and
$E\sprime$ is
the adjacency relation of Vinberg chambers.
However, 
the size of $V\sprime/\aut(Y)$
is  approximately $\vol (D_0)$ times the size of $V/\aut(Y)$.
Thus, very roughly speaking,   using the primitive embedding
$\SY(2)\inj L_{26}$ of type {\tt 96C}
gives us
computational advantage of
multiplicative factor the square of $\vol(D_0)=652758220800$.
\end{remark}
\subsection{Calculating  $\Rats_{\temp}$, $\Ells_{\temp}$  and $\GGG_X$}
\label{subsec:RRREEEtemp}
From  $V_0$ and $\GGG$ calculated above,
we compute the following data,
which will be used in Sections~\ref{subsec:Rats}~and~\ref{subsec:Ells}. 
\par
Recall that $\Rats(Y)$ is embedded in $\SY$ 
by $C\mapsto [C]$.
For each $D\in V_0$,
let $\Rats(Y, D)$ be the set of roots $r=[C]$ in $\Rats(Y)$
such that  $D\cap (r)\sperp$ is a wall of $D$.
Since $D\subset\NefY$, 
a root $r$ defining a wall of $D$ belongs to $\Rats(Y)$ 
if and only if $r$ splits in $\SX$.
Therefore 
we can calculate $\Rats(Y, D)$ by the method
in Section~\ref{subsec:splitting}.
We put
\[
\Rats_{\temp}:=\bigcup_{D\in V_0} \Rats(Y, D).
\]
Then the mapping
\[
\Rats_{\temp}\inj \Rats(Y)\surj \Rats(Y)/\aut(Y)
\]
is surjective.
Via the generating set $\GGG$,
we can generate (pseudo-)random elements of $\aut(Y)=\gen{\GGG}$.
For $[C], [C\sprime]\in \Rats_{\temp}$,
if we find  $g\in \aut(Y)$
such that $[C]^g=[C\sprime]$,
then we remove $[C\sprime]$ from $\Rats_{\temp}$.
Repeating this process many times, 
we obtain a smaller subset $\Rats_{\temp}\sprime$ of $\Rats(Y)$
that is mapped to $\Rats(Y)/\aut(Y)$
surjectively.
\par
Let $\phi\colon Y\to \P^1$ be an elliptic fibration of  $Y$,
and $F$ a  general fiber of $\phi$.
Then   $f_{\phi}:=[F]/2 \in \SY$
is a primitive isotropic ray  (see Section~\ref{subsec:faces} for the definition)
contained in the closure of $\NefY$ in $\closure{\PPP}_Y$.
For each $D\in V_0$,
let $\Ells(Y, D)$ be the set of primitive isotropic rays
contained in the closure $\closure{D}$ of $D$ in $\closure{\PPP}_Y$.
We put
\[
\Ells_{\temp}:=\bigcup_{D\in V_0} \Ells(Y, D).
\]
Then the mapping
\[
\Ells_{\temp}\inj \Ells(Y)\surj \Ells(Y)/\aut(Y)
\]
is surjective.
As above,
from  $\Ells_{\temp}$ and  using  $\GGG$,
we  obtain a smaller subset $\Ells_{\temp}\sprime$ of $\Ells(Y)$
that is mapped to $\Ells(Y)/\aut(Y)$
surjectively.
\par
Let  $\Aut(X, \enrinvol)$ be the centralizer of $\enrinvol\in \Aut(X)$ in $\Aut(X)$, 
and let $\aut(X, \enrinvol)$ be the image of $\Aut(X, \enrinvol)$ in $\aut(X)$.
We write an element $\tilde{\gamma}\in \Aut(X)$ as $(\tilde{g}, f)$ by~\eqref{eq:AutX}.
Since $\OG(T_X, \period)=\{\pm 1\}$ is abelian,
we see that $\tilde{\gamma}$ commutes with $\enrinvol\in \Aut(X)$
if and only if $\tilde{g}$ commutes with $\enrinvol\in \aut(X)$.
Hence $\aut(X, \enrinvol)$ is equal to the centralizer of $\enrinvol\in\aut(X)$ in $\aut(X)$.
 By the Torelli theorem~(see the proof of Proposition~\ref{prop:GY}),
 an element $\tilde{g}$ of  $\OGP(\SX)$ belongs to 
 $\aut(X, \enrinvol)$  if and only if 
 $\tilde{g}$ acts on $\SX\dual/\SX$ as $\pm 1$,
 preserves $\NefX$, and commutes with $\enrinvol\in \OGP(\SX)$.
 Let 
 $\aut(X,\enrinvol)_0$ be the group 
 consisting of elements  $\tilde{g} \in \aut(X, \enrinvol)$  that 
act on $\SX\dual/\SX$ as $1$.
 We have $\aut(X,\enrinvol)_0=\aut(X,\enrinvol)\cap G_X$.
\par
The restriction homomorphism 
$\tilg\mapsto \tilg|\SY$
gives a surjective homomorphism $\aut(X, \enrinvol) \to \aut(Y)$.
We calculate the kernel
\[
K:=\Ker(\aut(X, \enrinvol) \to \aut(Y)).
\]
The kernel $K$ is naturally embedded into $\OG(\SXm)$
by $\tilg\mapsto \tilg|\SXm$.
We  put 
\[
K_0:=\Ker(\aut(X,\enrinvol)_0 \to \aut(Y))\subset G_X.
\]
By definition $K_0$ acts trivially on $\SXp\dual/ \SXp$ and by
Proposition \ref{prop:03} it must act trivially on $\SXm\dual/\SXm$ as well. 
Hence, regarded as a subgroup of $G_{X-}\subset \OG(\SXm)$, $K_0$ is contained in the kernel of
\[\psi\colon G_{X-}\rightarrow \OG(\SXm\dual/\SXm).\]
Conversely the elements of $\ker \psi$ can be extended by the identity on $\SXp$ to elements of $G_X$ 
which trivially preserve $\Nef_Y$.
Hence they are induced by automorphisms of $Y$ and we have $K_0 = \ker \psi$. 
The kernel of $\psi$ is explicitly computed in the proof of Theorem \ref{thm:cde}. 
Its order is given by $e_{\tau,\taubar} \in \{1,2\}$.
Suppose that $e_{\tau,\taubar} =2$.
If $\enrinvol \in K_0$, then $K=K_0=\langle \enrinvol \rangle$.
This is the case if in addition $\tau(\Rtil)=E_8$.
Otherwise $K=K_0 \times \langle \enrinvol\rangle$ is of order $4$.
\par
For each $g$ in the generating set $\GGG$ of $\aut(Y)$, we calculate a lift
$\tilg\in \aut(X, \enrinvol)$ of $g$, 
and put 
\[
\GGG_X:=\set{\tilg}{g\in \GGG} \; \cup\;K.
\]
Then $\aut(X, \enrinvol)$ is generated by $\GGG_X$.
\subsection{Rational curves on $Y$}\label{subsec:Rats}
We prove Theorem~\ref{thm:main2rats}.
By the construction of $\SX$ given in Section~\ref{subsec:construction}, 
we have  a set of splitting roots that define some walls of $D_0\subset \NefY$
and form the dual graph of $\ADE$-type $\tau$. 
Therefore the existence of $C_1, \dots, C_m$ 
in assertion~(1) is proved.
\par
Let $C$ be a smooth rational curve on $Y$,
and  $r:=[C]$ the class of $C$.
Let $\Vtil_C$ be the set of
$\LSY$-chambers $D$ such that
$D\cap (r)\sperp$ is a wall of $D$ and that
$D$ is located on the same side of $(r)\sperp$ as $\NefY$.
Let $D$ be an element of $\Vtil_C$,
and suppose that $F:=D\cap (r)\sperp\cap (r\sprime)\sperp$
is a face of codimension $2$ of $D$ that is a boundary of the wall $D\cap (r)\sperp$,
where $r\sprime$ is a root of $\SY$ defining a wall of $D$.
Then there exists a unique element $D\sprime$ of $\Vtil_C$
such that $D\cap D\sprime=F$ holds.
We say that this chamber $D\sprime$ is \emph{adjacent in $\Vtil_C$} to $D$
across $F$.
This $\LSY$-chamber $D\sprime$ is calculated as follows.
As is seen from the set of faces of $\LSY$-chambers (see~\cite{BrandhorstShimadaCompData}),
we have $\intf{r, r\sprime}=0$ or $\intf{r, r\sprime}=1$.
Let $s$ and $s\sprime$ be the reflections
with respect to the roots $r=[C]$ and $r\sprime$,
respectively.
Then
\[
D\sprime=
\begin{cases}
 D^{s\sprime} &\textrm{if $\intf{r, r\sprime}=0$,} \\
 D^{ss\sprime} &\textrm{if $\intf{r, r\sprime}=1$.}
\end{cases}
\]
Suppose that $D$ is contained in $\NefY$.
Then $D\sprime$ is contained in $\NefY$ if and only if
$r\sprime$ is \emph{not} the class of a smooth rational curve on $Y$,
or equivalently, $r\sprime$ does not split in $\SX$.
We consider the  graph  $(V_C, E_C)$,
where
$V_C$ is the set of $\LSY$-chambers $D\in \Vtil_C$ contained in $\NefY$,
and $E_C$  is the restriction to $V_C\subset \Vtil_C$
of the adjacency relation on $\Vtil_C$ defined above.
Then the stabilizer subgroup
\[
G_C:=\aut(Y, C)=\set{g\in \aut(Y)}{r^g=r}
\]
of $C$ in $ \aut(Y)$
acts on $(V_C, E_C)$.
For $D, D\sprime\in V_C$,  we have 
\[
\TG (D, D\sprime)=\set{g\in \isoms(Y, D, D\sprime)}{ r^g=r},
\]
where 
 $\TG (D, D\sprime)\subset G_C$ is 
defined by~\eqref{eq:TG}, 
and $ \isoms(Y, D, D\sprime)$ is
defined in Section~\ref{subsec:criterionGequiv}.
Therefore $(V_C, E_C)$ and $G_C$
satisfy conditions~(\AssumpVE-1),~\dots,~(\AssumpVE-3)
in Section~\ref{subsec:graph}.
We apply Procedure~\ref{procedure:genB}
to every $C\in \Rats_{\temp}\sprime$
and obtain a complete  set $V_{C,0}$
of representatives of orbits of 
the action of $G_C$ on $V_C$.
\par
Two elements $C$ and $C\sprime$ of $\Rats_{\temp}\sprime$
are contained in the same orbit
under the action of $\aut(Y)$ on $\Rats(Y)$
if and only if we have one of the following conditions 
that are mutually equivalent.
\begin{itemize}
\item
Let $D$ be an arbitrary  element of $V_{C, 0}$.
Then there exists an $\LSY$-chamber $D\sprime $ in $V_{C\sprime, 0}$
such that $\isoms(Y, D, D\sprime)$ contains an isometry
$g$ such  that $[C]^g=[C\sprime]$.
\item
There exist a pair of  $\LSY$-chambers 
$D\in V_{C, 0}$ and 
$D\sprime \in V_{C\sprime, 0}$ and an isometry 
$g\in \isoms(Y, D, D\sprime)$ 
such  that $[C]^g=[C\sprime]$.
\end{itemize}
Applying this method to all pairs $C, C\sprime$ of
distinct elements of $\RRR_{\temp}\sprime$,
we obtain a complete set of representatives 
$C_1\sprime, \dots, C_k\sprime$  of orbits of 
the action of $\aut(Y)$ on $\Rats(Y)$.
We then apply this method to the representatives $C_1\sprime, \dots, C_k\sprime$ 
and the smooth rational curves $C_1, \dots, C_m$
in assertion~(1),
and complete the proof of Theorem~\ref{thm:main2rats}.
\par
The algorithm given above is a priori guaranteed to work.
A posteriori, Theorem~\ref{thm:main2rats} can be verified by the following simple strategy.
Let $\aut(X, \enrinvol)|\SXm$ be the image of 
 the homomorphism
\[
\aut(X, \enrinvol)\to \OG(\SXm)
\]
given by $\tilg\mapsto \tilg|\SXm$.
Since we have calculated a finite generating set $\GGG_X$ of $\aut(X, \enrinvol)$,
we can calculate the elements of the finite group $\aut(X, \enrinvol)|\SXm$.
Let  $C, C\sprime$ be elements of $\Rats(Y)$.
If the orbit of $\{\pm v_C\}\subset \SXm$ by $\aut(X, \enrinvol)|\SXm$ 
and that of $\{\pm v_{C\sprime}\}$ are disjoint, 
then the orbits of $C$ and $C\sprime$ by $\aut(Y)$ are disjoint.
Even though the converse does not necessarily hold,
we know a posteriori that once the size of $\Rats_{\temp}\sprime$ is small enough,
this separates the orbits of $\Rats_{\temp}\sprime$.
\subsection{Elliptic fibrations of $Y$}\label{subsec:Ells}
Let $\phi\colon Y\to \P^1$ be an elliptic fibration of  $Y$.
We consider the following graph  $(V_{\phi}, E_{\phi})$.
We define $V_{\phi}$ to be the set of $\LSY$-chambers $D$ contained in $\NefY$
such that the closure $\closure{D}$  of $D$ in $\SY\tensor\R$ contains the 
primitive isotropic ray  $f_{\phi}=[F]/2$, where $F$ is a general fiber of $\phi$,
and $E_{\phi}$ to be the set of pairs of adjacent $\LSY$-chambers in $V_{\phi}$.
The stabilizer subgroup
\[
G_{\phi}:=\aut(Y, \phi):=\set{g\in \aut(Y)}{f_{\phi}^g=f_{\phi}}
\]
of $\phi$ in $\aut(Y)$
acts on $(V_{\phi}, E_{\phi})$.
Then condition~(\AssumpVE-1) is satisfied.
Indeed, the set of $\LSY$-chambers in $V_{\phi}$ adjacent to $D\in V_{\phi}$ is
the set of all 
$ D^{s_r}$, where $r$ runs through the set of  \emph{non-splitting} roots of $\SY$ 
defining walls of $D$
such that $\intf{r, f_{\phi}}=0$.
For $D, D\sprime\in V_{\phi}$,  the subset $\TG (D, D\sprime)$ of $G_{\phi}$ is
the set of isometries belonging to $\isoms(Y, D, D\sprime)$ that fixes $f_{\phi}$.
Therefore~(\AssumpVE-2) and~(\AssumpVE-3) are also satisfied.
\par
We apply Procedure~\ref{procedure:genB}
to every $\phi\in \Ells_{\temp}\sprime$
and obtain a complete set $V_{\phi, 0}$
of representatives of orbits of the action of $G_{\phi}$ on $V_{\phi}$.
We also obtain a finite generating set $\GGG_{\phi}$
of the stabilizer subgroup $\aut(Y, \phi)$. 
\par
The set $\Sigma_{\phi}$ of classes of smooth rational curves $C$
contained in some fiber of $\phi$ is calculated as follows.
Let $a_Y$ be an ample class of $Y$.
Every class $[C]\in \Sigma_{\phi}$ satisfies $\intf{[C], f_{\phi}}=0$ 
and $0<\intf{[C], a_Y}< 2\intf{f_{\phi}, a_Y}$.
We calculate the set  $\Sigma\sprime$ 
of all roots $r$ of $\SY$ satisfying 
$\intf{r, f_{\phi}}=0$ 
and $0<\intf{r, a_Y}<2\intf{f_{\phi}, a_Y}$.
Then $r\in \Sigma\sprime$ belongs to $\Sigma_{\phi}$
if and only if $r$ splits in $\SX$ (see Section~\ref{subsec:splitting})
and there exist no roots $r\sprime\in \Sigma_{\phi}$ such that
$\intf{r\sprime, a_Y}<\intf{r, a_Y}$ and $\intf{r, r\sprime}<0$.
Therefore we can calculate $\Sigma_{\phi}$ by sorting the elements 
$r$ of $\Sigma\sprime$ according to  $\intf{r, a_Y}$ 
and applying the above criterion to  $r\in \Sigma\sprime$ in this order.
\par
Each connected component of the dual graph of roots
in $\Sigma_{\phi}$
corresponds to a reducible fiber of $\phi$,
and is the Dynkin diagram of an affine
$\ADE$-type.
Let $\Gamma$ be a connected component.
The weighted sum of roots in $\Gamma$
with appropriate weights according to the $\ADE$-type of $\Gamma$
(see, for example,~\cite[Theorem 5.12]{SMSbook})
is either $f_{\phi}$ or $2f_{\phi}$.
The former case occurs when the corresponding reducible fiber is a multiple fiber,
while the latter occurs when the  fiber is non-multiple.
\par
Let $\phi\sprime\colon Y\to\P^1$ be another element of $\Ells_{\temp}\sprime$.
Then $\phi$ and $\phi\sprime$ are contained in the same orbit
under the action of $\aut(Y)$ on $\Ells(Y)$
if and only if the following holds.
Let $D$ be an element of $V_{\phi, 0}$.
Then there exists $D\sprime \in V_{\phi\sprime, 0}$
such that $\isoms(Y, D, D\sprime)$ contains an isometry
that maps $f_{\phi}$ to $f_{\phi\sprime}$. 
Note that $D'$ can be computed explicitly.
Applying this method to all pairs $\phi, \phi\sprime$
of distinct elements of $\Ells_{\temp}\sprime$,
we obtain a complete set of representatives of
the action of $\aut(Y)$ on $\Ells(Y)$.
\subsection{Table of elliptic fibrations}\label{subsec:tableells}
Let $\phi\colon Y\to \P^1$ be
an elliptic fibration of an Enriques surface $Y$.
Then $\phi$ has exactly two multiple fibers,
and both of them are of multiplicity $2$.
In the table below,
the first column shows the $\ADE$-types of non-multiple reducible fibers,
and the second column shows the $\ADE$-types of multiple reducible fibers.
The third column gives the number of elliptic fibrations modulo $\aut(Y)$. 
See \cite{AutEnrVolCompdata} for the cases with $\rank \tau \geq 8$.
\par
{\tiny
\setlength{\columnsep}{10pt}
\setlength{\columnseprule}{.4pt}
\begin{multicols}{3}
\renewcommand{\arraystretch}{.85}
\par\noindent
\begin{align*}
&\hbox to 4cm {\textrm{No.~1:}\quad$(A_{1}, A_{1})$\hss}\\ 
&\begin{array}{lll}
\none & \none & 136\\
A_{1} & \none & 255\\
\end{array}
\end{align*}
\begin{align*}
&\hbox to 4cm {\textrm{No.~2:}\quad$(2A_{1}, 2A_{1})$\hss}\\ 
&\begin{array}{lll}
\none & \none & 36\\
\none & A_{1} & 1\\
A_{1} & \none & 128\\
2A_{1} & \none & 126\\
\end{array}
\end{align*}
\begin{align*}
&\hbox to 4cm {\textrm{No.~3:}\quad$(A_{2}, A_{2})$\hss}\\ 
&\begin{array}{lll}
A_{1} & \none & 136\\
A_{2} & \none & 119\\
\end{array}
\end{align*}
\begin{align*}
&\hbox to 4cm {\textrm{No.~4:}\quad$(3A_{1}, 3A_{1})$\hss}\\ 
&\begin{array}{lll}
\none & \none & 10\\
A_{1} & \none & 48\\
A_{1} & A_{1} & 3\\
2A_{1} & \none & 96\\
3A_{1} & \none & 60\\
\end{array}
\end{align*}
\begin{align*}
&\hbox to 4cm {\textrm{No.~5:}\quad$(A_{2}+A_{1}, A_{2}+A_{1})$\hss}\\ 
&\begin{array}{lll}
\none & A_{1} & 1\\
A_{1} & \none & 36\\
A_{2}+A_{1} & \none & 63\\
2A_{1} & \none & 63\\
A_{2} & \none & 28\\
\end{array}
\end{align*}
\begin{align*}
&\hbox to 4cm {\textrm{No.~6:}\quad$(A_{3}, A_{3})$\hss}\\ 
&\begin{array}{lll}
\none & A_{2} & 1\\
2A_{1} & \none & 36\\
A_{2} & \none & 64\\
A_{3} & \none & 54\\
\end{array}
\end{align*}
\begin{align*}
&\hbox to 4cm {\textrm{No.~7:}\quad$(4A_{1}, 4A_{1})$\hss}\\ 
&\begin{array}{lll}
\none & \none & 3\\
A_{1} & \none & 16\\
2A_{1} & \none & 48\\
2A_{1} & A_{1} & 6\\
3A_{1} & \none & 64\\
4A_{1} & \none & 25\\
\end{array}
\end{align*}
\begin{align*}
&\hbox to 4cm {\textrm{No.~8:}\quad$(4A_{1}, D_{4})$\hss}\\ 
&\begin{array}{lll}
\none & \none & 10\\
\none & 2A_{1} & 3\\
2A_{1} & \none & 96\\
4A_{1} & \none & 60\\
\end{array}
\end{align*}
\begin{align*}
&\hbox to 4cm {\textrm{No.~9:}\quad$(A_{2}+2A_{1}, A_{2}+2A_{1})$\hss}\\ 
&\begin{array}{lll}
A_{1} & \none & 10\\
A_{1} & A_{1} & 2\\
A_{2}+A_{1} & \none & 32\\
2A_{1} & \none & 32\\
A_{2}+2A_{1} & \none & 30\\
3A_{1} & \none & 30\\
A_{2} & \none & 6\\
A_{2} & A_{1} & 1\\
\end{array}
\end{align*}
\begin{align*}
&\hbox to 4cm {\textrm{No.~10:}\quad$(A_{3}+A_{1}, A_{3}+A_{1})$\hss}\\ 
&\begin{array}{lll}
A_{1} & A_{1} & 1\\
A_{1} & A_{2} & 1\\
A_{2}+A_{1} & \none & 32\\
A_{3}+A_{1} & \none & 30\\
2A_{1} & \none & 10\\
3A_{1} & \none & 15\\
A_{2} & \none & 16\\
A_{3} & \none & 12\\
\end{array}
\end{align*}
\begin{align*}
&\hbox to 4cm {\textrm{No.~11:}\quad$(2A_{2}, 2A_{2})$\hss}\\ 
&\begin{array}{lll}
\none & A_{1} & 1\\
A_{2}+A_{1} & \none & 56\\
2A_{1} & \none & 35\\
2A_{2} & \none & 35\\
\end{array}
\end{align*}
\begin{align*}
&\hbox to 4cm {\textrm{No.~12:}\quad$(A_{4}, A_{4})$\hss}\\ 
&\begin{array}{lll}
\none & A_{2} & 1\\
A_{2}+A_{1} & \none & 36\\
A_{3} & \none & 27\\
A_{4} & \none & 27\\
\end{array}
\end{align*}
\begin{align*}
&\hbox to 4cm {\textrm{No.~13:}\quad$(D_{4}, D_{4})$\hss}\\ 
&\begin{array}{lll}
\none & A_{3} & 3\\
4A_{1} & \none & 10\\
A_{3} & \none & 48\\
D_{4} & \none & 20\\
\end{array}
\end{align*}
\begin{align*}
&\hbox to 4cm {\textrm{No.~14:}\quad$(5A_{1}, 5A_{1})$\hss}\\ 
&\begin{array}{lll}
\none & \none & 1\\
A_{1} & \none & 5\\
2A_{1} & \none & 20\\
3A_{1} & \none & 40\\
3A_{1} & A_{1} & 10\\
4A_{1} & \none & 40\\
5A_{1} & \none & 5\\
\end{array}
\end{align*}
\begin{align*}
&\hbox to 4cm {\textrm{No.~15:}\quad$(5A_{1}, D_{4}+A_{1})$\hss}\\ 
&\begin{array}{lll}
\none & \none & 3\\
A_{1} & \none & 4\\
A_{1} & 2A_{1} & 3\\
2A_{1} & \none & 24\\
3A_{1} & \none & 48\\
3A_{1} & A_{1} & 4\\
4A_{1} & \none & 16\\
5A_{1} & \none & 24\\
\end{array}
\end{align*}
\begin{align*}
&\hbox to 4cm {\textrm{No.~16:}\quad$(A_{2}+3A_{1}, A_{2}+3A_{1})$\hss}\\ 
&\begin{array}{lll}
A_{1} & \none & 3\\
A_{2}+A_{1} & \none & 12\\
A_{2}+A_{1} & A_{1} & 3\\
2A_{1} & \none & 12\\
2A_{1} & A_{1} & 3\\
A_{2}+2A_{1} & \none & 24\\
3A_{1} & \none & 24\\
A_{2}+3A_{1} & \none & 12\\
4A_{1} & \none & 13\\
A_{2} & \none & 1\\
\end{array}
\end{align*}
\begin{align*}
&\hbox to 4cm {\textrm{No.~17:}\quad$(A_{3}+2A_{1}, A_{3}+2A_{1})$\hss}\\ 
&\begin{array}{lll}
A_{2}+A_{1} & \none & 16\\
A_{3}+A_{1} & \none & 16\\
2A_{1} & \none & 3\\
2A_{1} & A_{1} & 2\\
2A_{1} & A_{2} & 1\\
A_{2}+2A_{1} & \none & 16\\
A_{3}+2A_{1} & \none & 13\\
3A_{1} & \none & 8\\
4A_{1} & \none & 6\\
A_{2} & \none & 4\\
A_{3} & \none & 2\\
A_{3} & A_{1} & 1\\
\end{array}
\end{align*}
\begin{align*}
&\hbox to 4cm {\textrm{No.~18:}\quad$(A_{3}+2A_{1}, D_{5})$\hss}\\ 
&\begin{array}{lll}
\none & A_{2}+A_{1} & 1\\
\none & 2A_{1} & 1\\
A_{2}+A_{1} & \none & 32\\
2A_{1} & \none & 10\\
A_{3}+2A_{1} & \none & 30\\
4A_{1} & \none & 15\\
A_{3} & \none & 6\\
\end{array}
\end{align*}
\begin{align*}
&\hbox to 4cm {\textrm{No.~19:}\quad$(2A_{2}+A_{1}, 2A_{2}+A_{1})$\hss}\\ 
&\begin{array}{lll}
A_{1} & A_{1} & 1\\
A_{2}+A_{1} & \none & 12\\
2A_{2}+A_{1} & \none & 15\\
2A_{1} & \none & 10\\
A_{2}+2A_{1} & \none & 30\\
3A_{1} & \none & 15\\
A_{2} & A_{1} & 2\\
2A_{2} & \none & 10\\
\end{array}
\end{align*}
\begin{align*}
&\hbox to 4cm {\textrm{No.~20:}\quad$(A_{4}+A_{1}, A_{4}+A_{1})$\hss}\\ 
&\begin{array}{lll}
A_{1} & A_{2} & 1\\
A_{2}+A_{1} & \none & 10\\
A_{3}+A_{1} & \none & 15\\
A_{4}+A_{1} & \none & 15\\
A_{2}+2A_{1} & \none & 15\\
A_{2} & A_{1} & 1\\
A_{3} & \none & 6\\
A_{4} & \none & 6\\
\end{array}
\end{align*}
\begin{align*}
&\hbox to 4cm {\textrm{No.~21:}\quad$(D_{4}+A_{1}, D_{4}+A_{1})$\hss}\\ 
&\begin{array}{lll}
A_{1} & A_{3} & 3\\
A_{3}+A_{1} & \none & 24\\
D_{4}+A_{1} & \none & 12\\
3A_{1} & A_{1} & 1\\
4A_{1} & \none & 3\\
5A_{1} & \none & 3\\
A_{3} & \none & 12\\
D_{4} & \none & 4\\
\end{array}
\end{align*}
\begin{align*}
&\hbox to 4cm {\textrm{No.~22:}\quad$(A_{3}+A_{2}, A_{3}+A_{2})$\hss}\\ 
&\begin{array}{lll}
A_{1} & A_{1} & 1\\
A_{2}+A_{1} & \none & 16\\
A_{3}+A_{1} & \none & 12\\
A_{2}+2A_{1} & \none & 6\\
3A_{1} & \none & 9\\
A_{2} & A_{2} & 1\\
A_{3}+A_{2} & \none & 18\\
2A_{2} & \none & 16\\
\end{array}
\end{align*}
\begin{align*}
&\hbox to 4cm {\textrm{No.~23:}\quad$(A_{5}, A_{5})$\hss}\\ 
&\begin{array}{lll}
A_{1} & A_{2} & 1\\
A_{3}+A_{1} & \none & 15\\
2A_{2} & \none & 10\\
A_{4} & \none & 12\\
A_{5} & \none & 15\\
\end{array}
\end{align*}
\begin{align*}
&\hbox to 4cm {\textrm{No.~24:}\quad$(D_{5}, D_{5})$\hss}\\ 
&\begin{array}{lll}
\none & A_{3} & 1\\
\none & A_{4} & 1\\
A_{3}+2A_{1} & \none & 10\\
A_{4} & \none & 16\\
D_{4} & \none & 5\\
D_{5} & \none & 10\\
\end{array}
\end{align*}
\begin{align*}
&\hbox to 4cm {\textrm{No.~25:}\quad$(6A_{1}, D_{4}+2A_{1})$\hss}\\ 
&\begin{array}{lll}
\none & \none & 1\\
A_{1} & \none & 2\\
2A_{1} & \none & 8\\
2A_{1} & 2A_{1} & 3\\
3A_{1} & \none & 24\\
4A_{1} & \none & 28\\
4A_{1} & A_{1} & 9\\
5A_{1} & \none & 16\\
6A_{1} & \none & 3\\
\end{array}
\end{align*}
\begin{align*}
&\hbox to 4cm {\textrm{No.~27:}\quad$(A_{2}+4A_{1}, A_{2}+4A_{1})$\hss}\\ 
&\begin{array}{lll}
A_{1} & \none & 1\\
A_{2}+A_{1} & \none & 4\\
2A_{1} & \none & 4\\
A_{2}+2A_{1} & \none & 12\\
A_{2}+2A_{1} & A_{1} & 6\\
3A_{1} & \none & 12\\
3A_{1} & A_{1} & 4\\
A_{2}+3A_{1} & \none & 16\\
4A_{1} & \none & 16\\
A_{2}+4A_{1} & \none & 1\\
5A_{1} & \none & 4\\
\end{array}
\end{align*}
\begin{align*}
&\hbox to 4cm {\textrm{No.~28:}\quad$(A_{2}+4A_{1}, D_{4}+A_{2})$\hss}\\ 
&\begin{array}{lll}
A_{1} & \none & 3\\
A_{2}+2A_{1} & \none & 24\\
3A_{1} & \none & 24\\
3A_{1} & A_{1} & 4\\
A_{2}+4A_{1} & \none & 12\\
5A_{1} & \none & 12\\
A_{2} & \none & 1\\
A_{2} & 2A_{1} & 3\\
\end{array}
\end{align*}
\begin{align*}
&\hbox to 4cm {\textrm{No.~29:}\quad$(A_{3}+3A_{1}, A_{3}+3A_{1})$\hss}\\ 
&\begin{array}{lll}
A_{2}+A_{1} & \none & 6\\
A_{3}+A_{1} & \none & 6\\
A_{3}+A_{1} & A_{1} & 3\\
2A_{1} & \none & 1\\
A_{2}+2A_{1} & \none & 12\\
A_{3}+2A_{1} & \none & 12\\
3A_{1} & \none & 3\\
3A_{1} & A_{1} & 3\\
3A_{1} & A_{2} & 1\\
A_{2}+3A_{1} & \none & 8\\
A_{3}+3A_{1} & \none & 3\\
4A_{1} & \none & 6\\
5A_{1} & \none & 1\\
A_{2} & \none & 1\\
\end{array}
\end{align*}
\begin{align*}
&\hbox to 4cm {\textrm{No.~30:}\quad$(A_{3}+3A_{1}, D_{5}+A_{1})$\hss}\\ 
&\begin{array}{lll}
A_{1} & A_{2}+A_{1} & 1\\
A_{1} & 2A_{1} & 1\\
A_{2}+A_{1} & \none & 8\\
A_{3}+A_{1} & \none & 4\\
A_{3}+A_{1} & A_{1} & 2\\
2A_{1} & \none & 3\\
A_{2}+2A_{1} & \none & 16\\
A_{3}+2A_{1} & \none & 8\\
3A_{1} & \none & 4\\
3A_{1} & A_{1} & 1\\
A_{3}+3A_{1} & \none & 12\\
4A_{1} & \none & 4\\
5A_{1} & \none & 6\\
A_{3} & \none & 1\\
\end{array}
\end{align*}
\begin{align*}
&\hbox to 4cm {\textrm{No.~31:}\quad$(2A_{2}+2A_{1}, 2A_{2}+2A_{1})$\hss}\\ 
&\begin{array}{lll}
A_{2}+A_{1} & \none & 2\\
A_{2}+A_{1} & A_{1} & 4\\
2A_{2}+A_{1} & \none & 8\\
2A_{1} & \none & 3\\
2A_{1} & A_{1} & 1\\
A_{2}+2A_{1} & \none & 16\\
2A_{2}+2A_{1} & \none & 6\\
3A_{1} & \none & 8\\
A_{2}+3A_{1} & \none & 12\\
4A_{1} & \none & 7\\
2A_{2} & \none & 3\\
2A_{2} & A_{1} & 1\\
\end{array}
\end{align*}
\begin{align*}
&\hbox to 4cm {\textrm{No.~32:}\quad$(A_{4}+2A_{1}, A_{4}+2A_{1})$\hss}\\ 
&\begin{array}{lll}
A_{2}+A_{1} & \none & 3\\
A_{2}+A_{1} & A_{1} & 2\\
A_{3}+A_{1} & \none & 8\\
A_{4}+A_{1} & \none & 8\\
2A_{1} & A_{2} & 1\\
A_{2}+2A_{1} & \none & 8\\
A_{3}+2A_{1} & \none & 7\\
A_{4}+2A_{1} & \none & 6\\
A_{2}+3A_{1} & \none & 6\\
A_{3} & \none & 1\\
A_{4} & \none & 1\\
A_{4} & A_{1} & 1\\
\end{array}
\end{align*}
\begin{align*}
&\hbox to 4cm {\textrm{No.~33:}\quad$(D_{4}+2A_{1}, D_{4}+2A_{1})$\hss}\\ 
&\begin{array}{lll}
A_{3}+A_{1} & \none & 12\\
D_{4}+A_{1} & \none & 8\\
2A_{1} & A_{3} & 3\\
A_{3}+2A_{1} & \none & 12\\
D_{4}+2A_{1} & \none & 3\\
4A_{1} & \none & 1\\
4A_{1} & A_{1} & 2\\
5A_{1} & \none & 2\\
A_{3} & \none & 3\\
D_{4} & A_{1} & 1\\
\end{array}
\end{align*}
\begin{align*}
&\hbox to 4cm {\textrm{No.~34:}\quad$(D_{4}+2A_{1}, D_{6})$\hss}\\ 
&\begin{array}{lll}
\none & A_{3}+A_{1} & 1\\
A_{3}+A_{1} & \none & 16\\
2A_{1} & 2A_{1} & 1\\
2A_{1} & A_{3} & 2\\
A_{3}+2A_{1} & \none & 8\\
D_{4}+2A_{1} & \none & 12\\
4A_{1} & \none & 3\\
6A_{1} & \none & 3\\
A_{3} & \none & 4\\
D_{4} & \none & 2\\
\end{array}
\end{align*}
\begin{align*}
&\hbox to 4cm {\textrm{No.~35:}\quad\hss}\\ 
&(A_{3}+A_{2}+A_{1}, A_{3}+A_{2}+A_{1})\\ 
&\begin{array}{lll}
A_{2}+A_{1} & \none & 4\\
A_{2}+A_{1} & A_{1} & 1\\
A_{2}+A_{1} & A_{2} & 1\\
A_{3}+A_{2}+A_{1} & \none & 6\\
2A_{2}+A_{1} & \none & 8\\
A_{3}+A_{1} & \none & 2\\
2A_{1} & A_{1} & 1\\
A_{2}+2A_{1} & \none & 9\\
A_{3}+2A_{1} & \none & 7\\
3A_{1} & \none & 3\\
A_{2}+3A_{1} & \none & 3\\
4A_{1} & \none & 3\\
A_{3}+A_{2} & \none & 6\\
2A_{2} & \none & 4\\
A_{3} & A_{1} & 1\\
\end{array}
\end{align*}
\begin{align*}
&\hbox to 4cm {\textrm{No.~36:}\quad$(A_{5}+A_{1}, A_{5}+A_{1})$\hss}\\ 
&\begin{array}{lll}
2A_{2}+A_{1} & \none & 4\\
A_{3}+A_{1} & \none & 4\\
A_{4}+A_{1} & \none & 8\\
A_{5}+A_{1} & \none & 7\\
2A_{1} & A_{2} & 1\\
A_{3}+2A_{1} & \none & 6\\
2A_{2} & \none & 3\\
A_{3} & A_{1} & 1\\
A_{4} & \none & 2\\
A_{5} & \none & 4\\
\end{array}
\end{align*}
\begin{align*}
&\hbox to 4cm {\textrm{No.~37:}\quad$(A_{5}+A_{1}, E_{6})$\hss}\\ 
&\begin{array}{lll}
\none & A_{2}+A_{1} & 1\\
A_{5}+A_{1} & \none & 15\\
A_{3}+2A_{1} & \none & 15\\
2A_{2} & \none & 10\\
A_{4} & \none & 6\\
\end{array}
\end{align*}
\begin{align*}
&\hbox to 4cm {\textrm{No.~38:}\quad$(D_{5}+A_{1}, D_{5}+A_{1})$\hss}\\ 
&\begin{array}{lll}
A_{1} & A_{3} & 1\\
A_{1} & A_{4} & 1\\
A_{3}+A_{1} & A_{1} & 1\\
A_{4}+A_{1} & \none & 8\\
D_{4}+A_{1} & \none & 3\\
D_{5}+A_{1} & \none & 6\\
A_{3}+2A_{1} & \none & 3\\
A_{3}+3A_{1} & \none & 3\\
A_{4} & \none & 4\\
D_{4} & \none & 1\\
D_{5} & \none & 2\\
\end{array}
\end{align*}
\begin{align*}
&\hbox to 4cm {\textrm{No.~39:}\quad$(3A_{2}, 3A_{2})$\hss}\\ 
&\begin{array}{lll}
2A_{2}+A_{1} & \none & 30\\
A_{2}+2A_{1} & \none & 15\\
3A_{1} & \none & 10\\
A_{2} & A_{1} & 3\\
3A_{2} & \none & 5\\
\end{array}
\end{align*}
\begin{align*}
&\hbox to 4cm {\textrm{No.~40:}\quad$(3A_{2}, E_{6})$\hss}\\ 
&\begin{array}{lll}
2A_{2}+A_{1} & \none & 30\\
A_{2}+2A_{1} & \none & 15\\
3A_{1} & \none & 10\\
A_{2} & A_{1} & 3\\
3A_{2} & \none & 5\\
\end{array}
\end{align*}
\begin{align*}
&\hbox to 4cm {\textrm{No.~41:}\quad$(A_{4}+A_{2}, A_{4}+A_{2})$\hss}\\ 
&\begin{array}{lll}
2A_{2}+A_{1} & \none & 6\\
A_{3}+A_{1} & \none & 6\\
A_{4}+A_{1} & \none & 6\\
A_{2}+2A_{1} & \none & 9\\
A_{2} & A_{1} & 1\\
A_{2} & A_{2} & 1\\
A_{3}+A_{2} & \none & 9\\
A_{4}+A_{2} & \none & 9\\
\end{array}
\end{align*}
\begin{align*}
&\hbox to 4cm {\textrm{No.~42:}\quad$(D_{4}+A_{2}, D_{4}+A_{2})$\hss}\\ 
&\begin{array}{lll}
A_{3}+A_{1} & \none & 12\\
D_{4}+A_{1} & \none & 4\\
3A_{1} & A_{1} & 1\\
A_{2}+4A_{1} & \none & 1\\
5A_{1} & \none & 2\\
A_{2} & A_{3} & 3\\
A_{3}+A_{2} & \none & 12\\
D_{4}+A_{2} & \none & 8\\
\end{array}
\end{align*}
\begin{align*}
&\hbox to 4cm {\textrm{No.~43:}\quad$(2A_{3}, 2A_{3})$\hss}\\ 
&\begin{array}{lll}
2A_{1} & A_{1} & 1\\
A_{2}+2A_{1} & \none & 8\\
A_{3}+2A_{1} & \none & 4\\
4A_{1} & \none & 2\\
A_{3}+A_{2} & \none & 16\\
2A_{2} & \none & 8\\
A_{3} & A_{2} & 2\\
2A_{3} & \none & 9\\
\end{array}
\end{align*}
\begin{align*}
&\hbox to 4cm {\textrm{No.~44:}\quad$(2A_{3}, D_{6})$\hss}\\ 
&\begin{array}{lll}
\none & 2A_{1} & 1\\
\none & 2A_{2} & 1\\
A_{3}+2A_{1} & \none & 12\\
4A_{1} & \none & 9\\
2A_{2} & \none & 16\\
2A_{3} & \none & 18\\
\end{array}
\end{align*}
\begin{align*}
&\hbox to 4cm {\textrm{No.~45:}\quad$(A_{6}, A_{6})$\hss}\\ 
&\begin{array}{lll}
A_{4}+A_{1} & \none & 6\\
A_{2} & A_{2} & 1\\
A_{3}+A_{2} & \none & 9\\
A_{5} & \none & 6\\
A_{6} & \none & 9\\
\end{array}
\end{align*}
\begin{align*}
&\hbox to 4cm {\textrm{No.~46:}\quad$(D_{6}, D_{6})$\hss}\\ 
&\begin{array}{lll}
\none & A_{5} & 1\\
2A_{1} & A_{3} & 1\\
D_{4}+2A_{1} & \none & 3\\
2A_{3} & \none & 3\\
A_{5} & \none & 8\\
D_{5} & \none & 2\\
D_{6} & \none & 6\\
\end{array}
\end{align*}
\begin{align*}
&\hbox to 4cm {\textrm{No.~47:}\quad$(E_{6}, E_{6})$\hss}\\ 
&\begin{array}{lll}
\none & A_{4} & 1\\
A_{5}+A_{1} & \none & 10\\
D_{5} & \none & 5\\
E_{6} & \none & 5\\
\end{array}
\end{align*}
\begin{align*}
&\hbox to 4cm {\textrm{No.~50:}\quad\hss}\\ 
&(A_{2}+5A_{1}, D_{4}+A_{2}+A_{1})\\ 
&\begin{array}{lll}
A_{1} & \none & 1\\
A_{2}+A_{1} & \none & 1\\
A_{2}+A_{1} & 2A_{1} & 3\\
2A_{1} & \none & 1\\
A_{2}+2A_{1} & \none & 6\\
3A_{1} & \none & 6\\
A_{2}+3A_{1} & \none & 12\\
A_{2}+3A_{1} & A_{1} & 4\\
4A_{1} & \none & 12\\
4A_{1} & A_{1} & 5\\
A_{2}+4A_{1} & \none & 4\\
5A_{1} & \none & 4\\
6A_{1} & \none & 3\\
\end{array}
\end{align*}
\begin{align*}
&\hbox to 4cm {\textrm{No.~51:}\quad$(A_{3}+4A_{1}, D_{5}+2A_{1})$\hss}\\ 
&\begin{array}{lll}
A_{2}+A_{1} & \none & 2\\
A_{3}+A_{1} & \none & 2\\
2A_{1} & \none & 1\\
2A_{1} & A_{2}+A_{1} & 1\\
2A_{1} & 2A_{1} & 1\\
A_{2}+2A_{1} & \none & 8\\
A_{3}+2A_{1} & \none & 4\\
A_{3}+2A_{1} & A_{1} & 5\\
3A_{1} & \none & 2\\
A_{2}+3A_{1} & \none & 8\\
A_{3}+3A_{1} & \none & 8\\
4A_{1} & \none & 3\\
4A_{1} & A_{1} & 2\\
A_{3}+4A_{1} & \none & 1\\
5A_{1} & \none & 4\\
6A_{1} & \none & 1\\
\end{array}
\end{align*}
\begin{align*}
&\hbox to 4cm {\textrm{No.~52:}\quad$(A_{3}+4A_{1}, D_{4}+A_{3})$\hss}\\ 
&\begin{array}{lll}
2A_{1} & \none & 1\\
A_{2}+2A_{1} & \none & 12\\
A_{3}+2A_{1} & \none & 12\\
4A_{1} & \none & 6\\
4A_{1} & A_{1} & 4\\
4A_{1} & A_{2} & 1\\
A_{2}+4A_{1} & \none & 8\\
A_{3}+4A_{1} & \none & 3\\
A_{2} & \none & 1\\
A_{3} & 2A_{1} & 3\\
\end{array}
\end{align*}
\begin{align*}
&\hbox to 4cm {\textrm{No.~54:}\quad$(2A_{2}+3A_{1}, 2A_{2}+3A_{1})$\hss}\\ 
&\begin{array}{lll}
2A_{2}+A_{1} & \none & 3\\
2A_{2}+A_{1} & A_{1} & 3\\
2A_{1} & \none & 1\\
A_{2}+2A_{1} & \none & 6\\
A_{2}+2A_{1} & A_{1} & 6\\
2A_{2}+2A_{1} & \none & 6\\
3A_{1} & \none & 3\\
3A_{1} & A_{1} & 1\\
A_{2}+3A_{1} & \none & 12\\
4A_{1} & \none & 6\\
A_{2}+4A_{1} & \none & 2\\
5A_{1} & \none & 3\\
2A_{2} & \none & 1\\
\end{array}
\end{align*}
\begin{align*}
&\hbox to 4cm {\textrm{No.~55:}\quad$(A_{4}+3A_{1}, A_{4}+3A_{1})$\hss}\\ 
&\begin{array}{lll}
A_{2}+A_{1} & \none & 1\\
A_{3}+A_{1} & \none & 3\\
A_{4}+A_{1} & \none & 3\\
A_{4}+A_{1} & A_{1} & 3\\
A_{2}+2A_{1} & \none & 3\\
A_{2}+2A_{1} & A_{1} & 3\\
A_{3}+2A_{1} & \none & 6\\
A_{4}+2A_{1} & \none & 6\\
3A_{1} & A_{2} & 1\\
A_{2}+3A_{1} & \none & 6\\
A_{3}+3A_{1} & \none & 3\\
A_{2}+4A_{1} & \none & 1\\
\end{array}
\end{align*}
\begin{align*}
&\hbox to 4cm {\textrm{No.~56:}\quad$(D_{4}+3A_{1}, D_{6}+A_{1})$\hss}\\ 
&\begin{array}{lll}
A_{1} & A_{3}+A_{1} & 1\\
A_{3}+A_{1} & \none & 6\\
D_{4}+A_{1} & \none & 2\\
D_{4}+A_{1} & A_{1} & 2\\
A_{3}+2A_{1} & \none & 10\\
D_{4}+2A_{1} & \none & 4\\
3A_{1} & 2A_{1} & 1\\
3A_{1} & A_{3} & 2\\
A_{3}+3A_{1} & \none & 4\\
D_{4}+3A_{1} & \none & 2\\
4A_{1} & \none & 1\\
5A_{1} & \none & 1\\
5A_{1} & A_{1} & 1\\
6A_{1} & \none & 1\\
A_{3} & \none & 1\\
\end{array}
\end{align*}
\begin{align*}
&\hbox to 4cm {\textrm{No.~58:}\quad\hss}\\ 
&(A_{3}+A_{2}+2A_{1}, A_{3}+A_{2}+2A_{1})\\ 
&\begin{array}{lll}
A_{2}+A_{1} & \none & 1\\
A_{3}+A_{2}+A_{1} & \none & 4\\
2A_{2}+A_{1} & \none & 4\\
A_{3}+A_{1} & A_{1} & 2\\
A_{2}+2A_{1} & \none & 4\\
A_{2}+2A_{1} & A_{1} & 2\\
A_{2}+2A_{1} & A_{2} & 1\\
A_{3}+A_{2}+2A_{1} & \none & 1\\
2A_{2}+2A_{1} & \none & 4\\
A_{3}+2A_{1} & \none & 4\\
3A_{1} & \none & 1\\
3A_{1} & A_{1} & 1\\
A_{2}+3A_{1} & \none & 6\\
A_{3}+3A_{1} & \none & 2\\
4A_{1} & \none & 2\\
5A_{1} & \none & 1\\
A_{3}+A_{2} & \none & 2\\
A_{3}+A_{2} & A_{1} & 1\\
2A_{2} & \none & 1\\
\end{array}
\end{align*}
\begin{align*}
&\hbox to 4cm {\textrm{No.~59:}\quad\hss}\\ 
&(A_{3}+A_{2}+2A_{1}, D_{5}+A_{2})\\ 
&\begin{array}{lll}
2A_{2}+A_{1} & \none & 8\\
A_{3}+A_{1} & \none & 1\\
A_{3}+A_{1} & A_{1} & 2\\
A_{2}+2A_{1} & \none & 9\\
A_{3}+A_{2}+2A_{1} & \none & 6\\
3A_{1} & \none & 3\\
3A_{1} & A_{1} & 1\\
A_{3}+3A_{1} & \none & 6\\
A_{2}+4A_{1} & \none & 3\\
5A_{1} & \none & 3\\
A_{2} & A_{2}+A_{1} & 1\\
A_{2} & 2A_{1} & 1\\
A_{3}+A_{2} & \none & 3\\
\end{array}
\end{align*}
\begin{align*}
&\hbox to 4cm {\textrm{No.~60:}\quad$(A_{5}+2A_{1}, A_{5}+2A_{1})$\hss}\\ 
&\begin{array}{lll}
2A_{2}+A_{1} & \none & 2\\
A_{3}+A_{1} & \none & 1\\
A_{3}+A_{1} & A_{1} & 2\\
A_{4}+A_{1} & \none & 4\\
A_{5}+A_{1} & \none & 4\\
2A_{2}+2A_{1} & \none & 2\\
A_{3}+2A_{1} & \none & 4\\
A_{4}+2A_{1} & \none & 4\\
A_{5}+2A_{1} & \none & 2\\
3A_{1} & A_{2} & 1\\
A_{3}+3A_{1} & \none & 1\\
2A_{2} & \none & 1\\
A_{5} & \none & 1\\
A_{5} & A_{1} & 1\\
\end{array}
\end{align*}
\begin{align*}
&\hbox to 4cm {\textrm{No.~61:}\quad$(A_{5}+2A_{1}, E_{6}+A_{1})$\hss}\\ 
&\begin{array}{lll}
A_{1} & A_{2}+A_{1} & 1\\
2A_{2}+A_{1} & \none & 4\\
A_{3}+A_{1} & A_{1} & 1\\
A_{4}+A_{1} & \none & 4\\
A_{5}+A_{1} & \none & 4\\
A_{3}+2A_{1} & \none & 4\\
A_{5}+2A_{1} & \none & 6\\
A_{3}+3A_{1} & \none & 6\\
2A_{2} & \none & 3\\
A_{4} & \none & 1\\
A_{5} & A_{1} & 1\\
\end{array}
\end{align*}
\begin{align*}
&\hbox to 4cm {\textrm{No.~62:}\quad$(D_{5}+2A_{1}, D_{5}+2A_{1})$\hss}\\ 
&\begin{array}{lll}
A_{4}+A_{1} & \none & 4\\
D_{4}+A_{1} & \none & 2\\
D_{5}+A_{1} & \none & 4\\
2A_{1} & A_{3} & 1\\
2A_{1} & A_{4} & 1\\
A_{3}+2A_{1} & \none & 1\\
A_{3}+2A_{1} & A_{1} & 2\\
A_{4}+2A_{1} & \none & 4\\
D_{4}+2A_{1} & \none & 1\\
D_{5}+2A_{1} & \none & 1\\
A_{3}+3A_{1} & \none & 2\\
A_{4} & \none & 1\\
D_{5} & A_{1} & 1\\
\end{array}
\end{align*}
\begin{align*}
&\hbox to 4cm {\textrm{No.~63:}\quad$(D_{5}+2A_{1}, D_{7})$\hss}\\ 
&\begin{array}{lll}
\none & A_{4}+A_{1} & 1\\
A_{4}+A_{1} & \none & 8\\
2A_{1} & A_{3} & 1\\
A_{3}+2A_{1} & \none & 3\\
D_{4}+2A_{1} & \none & 3\\
D_{5}+2A_{1} & \none & 6\\
A_{3}+4A_{1} & \none & 3\\
A_{3} & 2A_{1} & 1\\
D_{4} & \none & 1\\
D_{5} & \none & 1\\
\end{array}
\end{align*}
\begin{align*}
&\hbox to 4cm {\textrm{No.~64:}\quad$(3A_{2}+A_{1}, 3A_{2}+A_{1})$\hss}\\ 
&\begin{array}{lll}
A_{2}+A_{1} & A_{1} & 3\\
2A_{2}+A_{1} & \none & 9\\
3A_{2}+A_{1} & \none & 3\\
A_{2}+2A_{1} & \none & 3\\
2A_{2}+2A_{1} & \none & 9\\
3A_{1} & \none & 3\\
A_{2}+3A_{1} & \none & 9\\
4A_{1} & \none & 4\\
2A_{2} & A_{1} & 3\\
3A_{2} & \none & 1\\
\end{array}
\end{align*}
\begin{align*}
&\hbox to 4cm {\textrm{No.~65:}\quad$(3A_{2}+A_{1}, E_{6}+A_{1})$\hss}\\ 
&\begin{array}{lll}
A_{2}+A_{1} & A_{1} & 3\\
2A_{2}+A_{1} & \none & 9\\
3A_{2}+A_{1} & \none & 3\\
A_{2}+2A_{1} & \none & 3\\
2A_{2}+2A_{1} & \none & 9\\
3A_{1} & \none & 3\\
A_{2}+3A_{1} & \none & 9\\
4A_{1} & \none & 4\\
2A_{2} & A_{1} & 3\\
3A_{2} & \none & 1\\
\end{array}
\end{align*}
\begin{align*}
&\hbox to 4cm {\textrm{No.~66:}\quad\hss}\\ 
&(A_{4}+A_{2}+A_{1}, A_{4}+A_{2}+A_{1})\\ 
&\begin{array}{lll}
A_{2}+A_{1} & A_{1} & 1\\
A_{2}+A_{1} & A_{2} & 1\\
A_{3}+A_{2}+A_{1} & \none & 3\\
A_{4}+A_{2}+A_{1} & \none & 3\\
2A_{2}+A_{1} & \none & 1\\
A_{3}+A_{1} & \none & 1\\
A_{4}+A_{1} & \none & 1\\
A_{2}+2A_{1} & \none & 3\\
2A_{2}+2A_{1} & \none & 3\\
A_{3}+2A_{1} & \none & 4\\
A_{4}+2A_{1} & \none & 3\\
A_{2}+3A_{1} & \none & 3\\
A_{3}+A_{2} & \none & 3\\
A_{4}+A_{2} & \none & 3\\
2A_{2} & A_{1} & 1\\
A_{4} & A_{1} & 1\\
\end{array}
\end{align*}
\begin{align*}
&\hbox to 4cm {\textrm{No.~67:}\quad\hss}\\ 
&(D_{4}+A_{2}+A_{1}, D_{4}+A_{2}+A_{1})\\ 
&\begin{array}{lll}
A_{2}+A_{1} & A_{3} & 3\\
A_{3}+A_{2}+A_{1} & \none & 6\\
A_{3}+A_{1} & \none & 3\\
A_{3}+2A_{1} & \none & 6\\
D_{4}+2A_{1} & \none & 3\\
A_{2}+3A_{1} & A_{1} & 1\\
4A_{1} & A_{1} & 1\\
5A_{1} & \none & 1\\
A_{3}+A_{2} & \none & 3\\
D_{4}+A_{2} & \none & 4\\
D_{4} & A_{1} & 1\\
\end{array}
\end{align*}
\begin{align*}
&\hbox to 4cm {\textrm{No.~68:}\quad$(2A_{3}+A_{1}, 2A_{3}+A_{1})$\hss}\\ 
&\begin{array}{lll}
A_{3}+A_{2}+A_{1} & \none & 8\\
2A_{2}+A_{1} & \none & 4\\
A_{3}+A_{1} & A_{1} & 2\\
A_{3}+A_{1} & A_{2} & 2\\
2A_{3}+A_{1} & \none & 1\\
A_{2}+2A_{1} & \none & 2\\
3A_{1} & A_{1} & 1\\
A_{2}+3A_{1} & \none & 4\\
A_{3}+3A_{1} & \none & 2\\
4A_{1} & \none & 1\\
A_{3}+A_{2} & \none & 4\\
2A_{2} & \none & 2\\
2A_{3} & \none & 4\\
\end{array}
\end{align*}
\begin{align*}
&\hbox to 4cm {\textrm{No.~69:}\quad$(2A_{3}+A_{1}, D_{6}+A_{1})$\hss}\\ 
&\begin{array}{lll}
A_{1} & 2A_{1} & 1\\
A_{1} & 2A_{2} & 1\\
2A_{2}+A_{1} & \none & 8\\
A_{3}+A_{1} & A_{1} & 2\\
2A_{3}+A_{1} & \none & 6\\
A_{3}+2A_{1} & \none & 2\\
A_{3}+3A_{1} & \none & 6\\
4A_{1} & \none & 3\\
5A_{1} & \none & 3\\
2A_{2} & \none & 4\\
2A_{3} & \none & 6\\
\end{array}
\end{align*}
\begin{align*}
&\hbox to 4cm {\textrm{No.~70:}\quad$(2A_{3}+A_{1}, E_{7})$\hss}\\ 
&\begin{array}{lll}
A_{1} & 2A_{1} & 1\\
A_{1} & 2A_{2} & 1\\
2A_{2}+A_{1} & \none & 8\\
A_{3}+A_{1} & A_{1} & 2\\
2A_{3}+A_{1} & \none & 6\\
A_{3}+2A_{1} & \none & 2\\
A_{3}+3A_{1} & \none & 6\\
4A_{1} & \none & 3\\
5A_{1} & \none & 3\\
2A_{2} & \none & 4\\
2A_{3} & \none & 6\\
\end{array}
\end{align*}
\begin{align*}
&\hbox to 4cm {\textrm{No.~71:}\quad$(A_{6}+A_{1}, A_{6}+A_{1})$\hss}\\ 
&\begin{array}{lll}
A_{2}+A_{1} & A_{2} & 1\\
A_{3}+A_{2}+A_{1} & \none & 3\\
A_{4}+A_{1} & \none & 1\\
A_{5}+A_{1} & \none & 4\\
A_{6}+A_{1} & \none & 3\\
A_{4}+2A_{1} & \none & 3\\
A_{3}+A_{2} & \none & 3\\
A_{4} & A_{1} & 1\\
A_{5} & \none & 1\\
A_{6} & \none & 3\\
\end{array}
\end{align*}
\begin{align*}
&\hbox to 4cm {\textrm{No.~72:}\quad$(D_{6}+A_{1}, D_{6}+A_{1})$\hss}\\ 
&\begin{array}{lll}
A_{1} & A_{5} & 1\\
2A_{3}+A_{1} & \none & 1\\
A_{5}+A_{1} & \none & 4\\
D_{4}+A_{1} & A_{1} & 1\\
D_{5}+A_{1} & \none & 2\\
D_{6}+A_{1} & \none & 2\\
D_{4}+2A_{1} & \none & 1\\
3A_{1} & A_{3} & 1\\
2A_{3} & \none & 1\\
A_{5} & \none & 2\\
D_{6} & \none & 2\\
\end{array}
\end{align*}
\begin{align*}
&\hbox to 4cm {\textrm{No.~73:}\quad$(D_{6}+A_{1}, E_{7})$\hss}\\ 
&\begin{array}{lll}
A_{1} & A_{3}+A_{1} & 1\\
A_{1} & A_{5} & 1\\
A_{5}+A_{1} & \none & 4\\
D_{6}+A_{1} & \none & 6\\
D_{4}+3A_{1} & \none & 3\\
2A_{3} & \none & 3\\
A_{5} & \none & 4\\
D_{5} & \none & 1\\
\end{array}
\end{align*}
\begin{align*}
&\hbox to 4cm {\textrm{No.~74:}\quad$(E_{6}+A_{1}, E_{6}+A_{1})$\hss}\\ 
&\begin{array}{lll}
A_{1} & A_{4} & 1\\
A_{5}+A_{1} & \none & 3\\
D_{5}+A_{1} & \none & 3\\
E_{6}+A_{1} & \none & 3\\
A_{5}+2A_{1} & \none & 3\\
A_{5} & A_{1} & 1\\
D_{5} & \none & 1\\
E_{6} & \none & 1\\
\end{array}
\end{align*}
\begin{align*}
&\hbox to 4cm {\textrm{No.~75:}\quad$(A_{3}+2A_{2}, A_{3}+2A_{2})$\hss}\\ 
&\begin{array}{lll}
A_{2}+A_{1} & A_{1} & 2\\
A_{3}+A_{2}+A_{1} & \none & 12\\
2A_{2}+A_{1} & \none & 8\\
A_{2}+2A_{1} & \none & 4\\
2A_{2}+2A_{1} & \none & 3\\
A_{3}+2A_{1} & \none & 1\\
4A_{1} & \none & 3\\
2A_{2} & A_{2} & 1\\
3A_{2} & \none & 4\\
A_{3} & A_{1} & 1\\
\end{array}
\end{align*}
\begin{align*}
&\hbox to 4cm {\textrm{No.~76:}\quad$(A_{5}+A_{2}, A_{5}+A_{2})$\hss}\\ 
&\begin{array}{lll}
A_{2}+A_{1} & A_{2} & 1\\
A_{3}+A_{2}+A_{1} & \none & 3\\
2A_{2}+A_{1} & \none & 3\\
A_{4}+A_{1} & \none & 2\\
A_{5}+A_{1} & \none & 4\\
A_{3}+2A_{1} & \none & 3\\
A_{4}+A_{2} & \none & 6\\
A_{5}+A_{2} & \none & 3\\
3A_{2} & \none & 1\\
A_{3} & A_{1} & 1\\
\end{array}
\end{align*}
\begin{align*}
&\hbox to 4cm {\textrm{No.~77:}\quad$(A_{5}+A_{2}, E_{7})$\hss}\\ 
&\begin{array}{lll}
A_{2}+A_{1} & A_{2} & 1\\
A_{3}+A_{2}+A_{1} & \none & 3\\
2A_{2}+A_{1} & \none & 3\\
A_{4}+A_{1} & \none & 2\\
A_{5}+A_{1} & \none & 4\\
A_{3}+2A_{1} & \none & 3\\
A_{4}+A_{2} & \none & 6\\
A_{5}+A_{2} & \none & 3\\
3A_{2} & \none & 1\\
A_{3} & A_{1} & 1\\
\end{array}
\end{align*}
\begin{align*}
&\hbox to 4cm {\textrm{No.~78:}\quad$(D_{5}+A_{2}, D_{5}+A_{2})$\hss}\\ 
&\begin{array}{lll}
A_{3}+A_{1} & A_{1} & 1\\
A_{4}+A_{1} & \none & 4\\
D_{4}+A_{1} & \none & 1\\
D_{5}+A_{1} & \none & 2\\
A_{3}+A_{2}+2A_{1} & \none & 1\\
A_{3}+3A_{1} & \none & 2\\
A_{2} & A_{3} & 1\\
A_{2} & A_{4} & 1\\
A_{4}+A_{2} & \none & 4\\
D_{4}+A_{2} & \none & 2\\
D_{5}+A_{2} & \none & 4\\
\end{array}
\end{align*}
\begin{align*}
&\hbox to 4cm {\textrm{No.~79:}\quad$(A_{4}+A_{3}, A_{4}+A_{3})$\hss}\\ 
&\begin{array}{lll}
A_{2}+A_{1} & A_{1} & 1\\
A_{3}+A_{2}+A_{1} & \none & 2\\
2A_{2}+A_{1} & \none & 4\\
A_{3}+2A_{1} & \none & 1\\
A_{4}+2A_{1} & \none & 1\\
A_{2}+3A_{1} & \none & 2\\
A_{3}+A_{2} & \none & 4\\
A_{4}+A_{2} & \none & 4\\
A_{3} & A_{2} & 1\\
A_{4}+A_{3} & \none & 4\\
2A_{3} & \none & 5\\
A_{4} & A_{2} & 1\\
\end{array}
\end{align*}
\begin{align*}
&\hbox to 4cm {\textrm{No.~80:}\quad$(D_{4}+A_{3}, D_{4}+A_{3})$\hss}\\ 
&\begin{array}{lll}
A_{3}+2A_{1} & \none & 3\\
4A_{1} & A_{1} & 1\\
A_{2}+4A_{1} & \none & 1\\
A_{3}+A_{2} & \none & 6\\
D_{4}+A_{2} & \none & 4\\
A_{3} & A_{3} & 3\\
D_{4}+A_{3} & \none & 3\\
2A_{3} & \none & 6\\
D_{4} & A_{2} & 1\\
\end{array}
\end{align*}
\begin{align*}
&\hbox to 4cm {\textrm{No.~81:}\quad$(D_{4}+A_{3}, D_{7})$\hss}\\ 
&\begin{array}{lll}
\none & A_{3}+A_{2} & 1\\
2A_{1} & 2A_{1} & 1\\
A_{3}+2A_{1} & \none & 4\\
D_{4}+2A_{1} & \none & 2\\
A_{3}+4A_{1} & \none & 1\\
6A_{1} & \none & 2\\
A_{3}+A_{2} & \none & 8\\
A_{3} & A_{3} & 2\\
D_{4}+A_{3} & \none & 8\\
2A_{3} & \none & 4\\
\end{array}
\end{align*}
\begin{align*}
&\hbox to 4cm {\textrm{No.~82:}\quad$(A_{7}, A_{7})$\hss}\\ 
&\begin{array}{lll}
A_{5}+A_{1} & \none & 2\\
A_{4}+A_{2} & \none & 4\\
A_{3} & A_{2} & 1\\
2A_{3} & \none & 2\\
A_{6} & \none & 4\\
A_{7} & \none & 5\\
\end{array}
\end{align*}
\begin{align*}
&\hbox to 4cm {\textrm{No.~83:}\quad$(A_{7}, E_{7})$\hss}\\ 
&\begin{array}{lll}
\none & 2A_{2} & 1\\
A_{5}+A_{1} & \none & 6\\
2A_{3} & \none & 9\\
A_{7} & \none & 9\\
\end{array}
\end{align*}
\begin{align*}
&\hbox to 4cm {\textrm{No.~84:}\quad$(D_{7}, D_{7})$\hss}\\ 
&\begin{array}{lll}
\none & A_{6} & 1\\
D_{5}+2A_{1} & \none & 1\\
A_{3} & A_{3} & 1\\
D_{4}+A_{3} & \none & 2\\
A_{6} & \none & 4\\
D_{6} & \none & 1\\
D_{7} & \none & 4\\
\end{array}
\end{align*}
\begin{align*}
&\hbox to 4cm {\textrm{No.~85:}\quad$(E_{7}, E_{7})$\hss}\\ 
&\begin{array}{lll}
A_{1} & A_{5} & 1\\
D_{6}+A_{1} & \none & 3\\
A_{7} & \none & 3\\
E_{6} & \none & 1\\
E_{7} & \none & 3\\
\end{array}
\end{align*}

\end{multicols}
}
\section{Examples}\label{sec:examples}
\subsection{An $(E_6, E_6)$-generic Enriques surface}\label{subsec:example47}
In~\cite{Shimada2019}, we  investigated 
an $(E_6, E_6)$-generic Enriques surface
(No.~47 of Table~\ref{table:184}).
We briefly review the result of~\cite{Shimada2019}.
\par
Let $\cloX\subset \P^3$ be a quartic Hessian surface
associated with a
very general cubic homogeneous polynomial,
and $X$ the minimal resolution of $\cloX$.
Then $\cloX$ contains $10$ lines and has $10$ ordinary nodes,
and the $K3$ surface $X$ has a fixed-point free involution $\enrinvol$
that interchanges the strict transforms of the $10$ lines
and the exceptional curves
over the $10$ ordinary nodes.
Let $\pi\colon X\to Y$ be the quotient morphism by
$\enrinvol$.
Then the Enriques surface $Y$ is
$(E_6, E_6)$-generic~(see Kondo~\cite{Kondo2012}).
\par
We can construct
a sequence of primitive
embeddings $\SY(2)\inj \SX \inj L_{26}$
from the primitive embeddings $L_{10}(2) \inj L_{26}$
of type {\tt 20E}.
We see that
$D_0$ is a fundamental domain of the action
of $\aut(Y)$ on $\Nef_Y$, and hence
\[
\vol(\NefY/\aut(Y))=\vol(D_0)=\frac{1_{\BP}}{51840}=
\frac{1_{\BP}}{|W(R_{E_6})|}.
\]
In fact, the $\LSY$-chamber  $D_0$ is equal to the chamber $D_Y$ in~\cite{Shimada2019}.
We then obtain the same result as Table 1.1 of~\cite{Shimada2019}
for $\Ells(Y)/\aut(Y)$.
We also prove that $\aut(Y)$ acts on $\Rats(Y)$ transitively.
\par
The last result contradicts Theorem~1.5~of~\cite{Shimada2019},
because Table~1.2~of~\cite{Shimada2019} says that
there exist  $10$
orbits of the action of $\aut(Y)$ on $\Rats(Y)$.
In fact, the argument in Section 7.6 of~\cite{Shimada2019} for
the calculation of the number of $\aut(Y)$-orbits
of $\RDP$-configurations  is wrong, and Table~1.2~of~\cite{Shimada2019}
should be replaced by Table~\ref{table:erratum} below.
\begin{table}[b]
{\small
$$
\arraycolsep=6pt
\begin{array}{cc}
 \hbox{\rm $\ADE$-type} & {\rm number} \\
 \hline
E_{6} & 1\\
A_{5}+A_{1} & 5\\
3A_{2} & 1\\
D_{5} & 1\\
A_{5} & 1\\
A_{4}+A_{1} & 1\\
A_{3}+2A_{1} & 5\\
2A_{2}+A_{1} & 1\\
D_{4} & 1\\
A_{4} & 1\\
\end{array}
\qquad\quad
\begin{array}{cc}
 \hbox{\rm $\ADE$-type} & {\rm number} \\
 \hline
A_{3}+A_{1} & 1\\
2A_{2} & 1\\
A_{2}+2A_{1} & 1\\
4A_{1} & 5\\
A_{3} & 1\\
A_{2}+A_{1} & 1\\
3A_{1} & 2\\
A_{2} & 1\\
2A_{1} & 1\\
A_{1} & 1\\
\end{array}
$$
}
\vskip .2cm
\caption{$\RDP$-configurations on $Y$}\label{table:erratum}
\end{table}
\par
Here we present a correct method for
the calculation of $\aut(Y)$-orbits
of $\RDP$-configurations.
Let $\psi\colon Y\to \cloY$  be a birational morphism
to a surface $\cloY$ that has only rational double points
as its singularities,
and let $h_{\psi}$ be an ample class of $\cloY$.
Since the $\LSY$-chamber $D_0$ is a fundamental domain
of the action of $\aut(Y)$ on $\NefY$,
we can assume that $\psi^*(h_{\psi}) \in \SY$ belongs to $D_0$
by composing $\psi$ with an automorphism of $Y$.
Let $f$ be the minimal face of $D_0$ containing $\psi^*(h_{\psi})$.
Then the set of the classes of smooth rational curves $C$
contracted by $\psi$ is equal to
\[
\Gamma (f):=\set{[C]}{\textrm{$C$ is a smooth rational curve 
on $Y$ such that $f\subset ([C])\sperp$}}.
\]
For a given face $f$ of $D_0$,
we calculate the set of roots $r$ of $\SY$
such that $f\subset(r)\sperp$.
From this set, we can calculate $\Gamma(f)$ by 
using the ample class $a_Y$ and
the set of $(-4)$-vectors in
$\SXm$.
We calculate $\Gamma(f)$ for all faces  $f$ of $D_0$,
and obtain $750$ $\RDP$-configurations of smooth rational curves.
Every $\RDP$-configuration on $Y$ is equal to one of them
modulo the action of $\aut(Y)$.
\par
Let $\Gamma$ be one of the $750$ $\RDP$-configurations.
We put $\mu:=|\Gamma|$, that is, $\mu$ is 
the total Milnor number 
of the singularities of the surface $\cloY$ corresponding to $\Gamma$.
The sublattice $\gen{\Gamma}$ of $\SY$
generated by the classes in $\Gamma$ is negative definite of rank $\mu$,
and  its orthogonal complement $\gen{\Gamma}\sperp$ is
hyperbolic of rank $10-\mu$.
Let $\PPP_{\gen{\Gamma}\sperp}$ be the positive half-cone
of $\gen{\Gamma}\sperp$ contained in $\PPP_Y$.
Composing the primitive embedding $\gen{\Gamma}\sperp\inj \SY$
with the primitive embedding $\SY(2)\inj L_{26}$ of type {\tt 20E},
we have $L_{26}/\gen{\Gamma}\sperp(2)$-chambers
of  $\PPP_{\gen{\Gamma}\sperp}$.
The intersection $f_0:=\PPP_{\gen{\Gamma}\sperp}\cap D_0$
is one of the $L_{26}/\gen{\Gamma}\sperp(2)$-chambers,
and it is the maximal face of $D_0$ among all 
the faces $f$ of $D_0$ such that $\Gamma(f)=\Gamma$.
Let $(V_{\Gamma}, E_{\Gamma})$ be the graph where
$V_{\Gamma}$ is the set of $L_{26}/\gen{\Gamma}\sperp(2)$-chambers
on $\PPP_{\gen{\Gamma}\sperp}$
contained in $\PPP_{\gen{\Gamma}\sperp}\cap \NefY$
and $E_{\Gamma}$ is the usual adjacency relation of chambers.
Then $D\mapsto \PPP_{\gen{\Gamma}\sperp}\cap D$
gives a bijection to the set $V_{\Gamma}$ of vertices
from the set of $\LSY$-chambers  $D$ contained in $\NefY$
such that $\PPP_{\gen{\Gamma}\sperp}\cap D$
is a face of $D$ of dimension $10-\mu$,
or equivalently, such that 
$\PPP_{\gen{\Gamma}\sperp}\cap D$ contains 
a non-empty open subset of $\PPP_{\gen{\Gamma}\sperp}$.
The group
\[
G_{\Gamma}:=\set{g\in \aut(Y)}{\Gamma^g=\Gamma}
\]
acts on the graph $(V_{\Gamma}, E_{\Gamma})$.
We apply Procedure~\ref{procedure:genB} to
$(V_{\Gamma}, E_{\Gamma})$ and $G_{\Gamma}$,
and obtain
a complete  set $V_{\Gamma, 0}$ of representatives of $V_{\Gamma}/G_{\Gamma}$.
Let $\Gamma\sprime$ be one of the $750$ $\RDP$-configurations
with the same $\ADE$-type as $\Gamma$.
Let $V_{\Gamma\sprime, 0}$ be a complete set of
representatives of $V_{\Gamma\sprime}/G_{\Gamma\sprime}$.
Then the $\RDP$-configurations $\Gamma$ and $\Gamma\sprime$
are in the same orbit under the action of $\aut(Y)$
if and only if
there exists  an $L_{26}/\gen{\Gamma\sprime}\sperp(2)$-chamber
$f\sprime =\PPP_{\gen{\Gamma\sprime}\sperp}\cap D\sprime \in V_{\Gamma\sprime,0}$
with $D\sprime\subset\NefY$ 
such that  $\isoms(Y, D_0, D\sprime)$ contains an element $g$ 
satisfying  $\Gamma^g=\Gamma\sprime$.
Since $|V_{\Gamma\sprime,0}|$ is finite,
we can determine whether 
$\Gamma$ and $\Gamma\sprime$
are in the same orbit or not.
Applying this method to all pairs $\Gamma$ and $\Gamma\sprime$
with the same $\ADE$-type,
we obtain a complete set of representatives of $\RDP$-configurations
modulo $\aut(Y)$.
%
\subsection{$(4A_1, 4A_1)$-generic and $(4A_1, D_4)$-generic Enriques surfaces}
Let $Y$ be a $(4A_1, 4A_1)$-generic Enriques surface (No.~7 of Table~\ref{table:184}).
We construct a sequence $\SY(2)\inj \SX\inj L_{26}$
from the primitive embedding $L_{10}(2)\inj L_{26}$ of type {\tt 96C}.
The complete set $V_0$ of representatives 
of orbits of the action of $\aut(Y)$ on 
the set of $\LSY$-chambers contained in $\NefY$ consists of $5$ elements
with the orders of stabilizer subgroups $1,1,1,2,1$. 
Since $\vol(D_0)=1_{\BP}/72$, we have 
\[
\vol(\NefY/\aut(Y))=
\vol(D_0)\left(\frac{1}{1}+\frac{1}{1}+\frac{1}{1}+\frac{1}{2}+\frac{1}{1} \right)
=\frac{1_{\BP}}{16}=\frac{1_{\BP}}{|W(R_{4A_1})|}.
\]
The set $\Rats_{\temp}$ is of size $56$ and 
the set $\Ells_{\temp}$ is  of size $6270$.
\par 
We also construct  $\SY(2)\inj \SX\inj L_{26}$
for  a $(4A_1, D_4)$-generic Enriques surface (No.~8 of Table~\ref{table:184})
from the primitive embedding  of type {\tt 96C}.
The set $V_0$  consists of $18$ elements
with the orders of stabilizer subgroups $4, \dots, 4$.
We have  $|\Rats_{\temp}|=154$ and 
 $|\Ells_{\temp}|=21452$.
\subsection{A $(D_5, D_5)$-generic  Enriques surface}
We have to use the primitive embedding  of type {\tt 40A}
to construct  $\SY(2)\inj \SX\inj L_{26}$
for  a $(D_5, D_5)$-generic Enriques surface (No.~24 of Table~\ref{table:184}).
The set $V_0$  consists of $6$ elements
with the orders of stabilizer subgroups $2, \dots, 2$.
In this case,  we have $\vol(D_0)=1_{\BP}/5760$ and 
\[
\vol(\NefY/\aut(Y))=
\vol(D_0)\left(\frac{1}{2}+\frac{1}{2}+\frac{1}{2}+\frac{1}{2}+\frac{1}{2}+\frac{1}{2} \right)
=\frac{1_{\BP}}{1920}=\frac{1_{\BP}}{|W(R_{D_5})|}.
\]
We have  $|\Rats_{\temp}|=15$ and 
 $|\Ells_{\temp}|=758$.
\subsection{Enriques surfaces with finite automorphism group}\label{subsec:finiteAut}
Let $Y$ be an Enriques surface
with finite automorphism group
of type I in Kondo's classification~\cite{Kondo1986}.
We assume that  $Y$  is chosen very general
so that the covering $K3$ surface $X$ is
of Picard number $19$ and satisfies $\OG(T_X, \omega)=\{\pm 1\}$.
Then $Y$ is $(E_8+A_1, E_8+A_1)$-generic~(No.~172 of Table~\ref{table:184}).
The automorphism group $\Aut(Y)$ is
a dihedral group of order $8$,
and its image $\aut(Y)$ in $\OGP(\SY)$
is order $4$.
The Enriques surface $Y$ has exactly $12$ smooth rational curves,
and their dual graph is given in~\cite[Fig.~1.4]{Kondo1986}.
The chamber $\NefY$ is isomorphic to an $\LLt$-chamber $D_0$
of the primitive embedding $L_{10}(2)\inj L_{26}$
of type {\tt 12A}, and hence 
$\vol(D_0)=1_{\BP}/174182400$ (see~\cite{BS2019}).
Therefore
\[
\vol(\NefY/\aut(Y))=\frac{\vol(D_0)}{4}=\frac{1_{\BP}}{2^{14} 3^5 5^2 7}
=\frac{2_{\BP}}{ |W(R_{E_8+A_1})|}.
\]
The group $\aut(Y)$ decomposes $\Rats(Y)$
as $2+2+2+2+4$.
\par
For a very general  Enriques surface $Y$
with finite automorphism group
of type II,
the chamber 
 $\NefY$ is isomorphic to an $\LLt$-chamber $D_0$
of the primitive embedding $L_{10}(2)\inj L_{26}$
of type {\tt 12B}.
We have $\vol(D_0)=1_{\BP}/3870720$.
Note that $3870720\cdot |\SSSS_4|=|W(R_{D_9})|$.
The Enriques surface $Y$ is $(D_9, D_9)$-generic~(No.~184 of Table~\ref{table:184}),
and we have  $\Aut(Y)\cong \aut(Y)\cong \SSSS_4$.
The group $\aut(Y)$  decomposes $\Rats(Y)$
as $6+6$.
\bibliographystyle{plain}

%
\end{document}